\documentclass{siamart171218}
\usepackage{amssymb}
\usepackage{enumerate}
\usepackage{graphicx}
\usepackage{psfrag}
\usepackage[percent]{overpic}
\DeclareMathAlphabet{\mathpzc}{OT1}{pzc}{m}{it}

\newcommand{\Tg}{\mathcal{T}^{L,n}_{geo,\sigma}}

\newcommand{\PiT}{\widehat \Pi^\triangle_q}
\newcommand{\tPiT}{\widetilde \Pi_q}
\newcommand{\PiGL}{\widehat \Pi^{\Box}_q}
\newcommand{\tPiGL}{\widetilde \Pi_q}
\newcommand{\M}{\mathcal{M}}

\newcommand{\di}{\text{\rm dist}}
\newcommand{\xh}{{\widehat x}}
\newcommand{\yh}{{\widehat y}}

\newcommand{\boxt}{{\widetilde \bx}}
\newcommand{\xt}{{\widetilde x}}
\newcommand{\yt}{{\widetilde y}}
\newcommand{\Kt}{{\widetilde K}}
\newcommand{\ut}{{\widetilde u}}
\newcommand{\rt}{{\widetilde r}}
\newcommand{\Kh}{{\widehat  K}}
\newcommand{\uh}{{\widehat u}}
\newcommand{\Tf}{{T}^{\text{\rm flip}}}
\newcommand{\Tref}{\widetilde {\calT}}
\newcommand{\BL}{{\sf BL}}
\newcommand{\CL}{{\sf CL}}
\newcommand{\Co}{{\sf C}}
\newcommand{\Te}{{\sf T}}
\newcommand{\Mi}{{\sf M}}
\newcommand{\bx}{\mathbf{x}}
\newcommand{\K}{\mathcal{K}}
\newcommand{\R}{\mathbb{R}}

\newtheorem{remark}[theorem]{Remark}
\newtheorem{example}[theorem]{Example}
\newcommand{\eremk}{\hbox{}\hfill\rule{0.8ex}{0.8ex}}
\numberwithin{equation}{section}

\newcommand{\calO}{{\mathcal O}}
\newcommand{\calT}{{\mathcal T}}

\newcommand{\bA}{{\boldsymbol A}}
\newcommand{\bP}{{\boldsymbol P}}
\newcommand{\bbV}{{\boldsymbol V}}
\newcommand{\bO}{{\boldsymbol 0}}

\newcommand{\bdy}{{\Gamma}}
\newcommand{\eps}{\varepsilon}

\newcommand{\bbN}{{\mathbb N}}
\newcommand{\bbQ}{{\mathbb Q}}
\newcommand{\bbP}{{\mathbb P}}
\newcommand{\bbR}{{\mathbb R}}
\title{
$hp$-FEM for reaction-diffusion equations.
\\
II: Robust exponential convergence for multiple length scales in corner domains.
\thanks{
The research of JMM was supported by the Austrian Science Fund (FWF) project F 65. 
Work performed in part while CS was visiting the Erwin Schr\"odinger Institute (ESI)
in Vienna in June-August 2018 during the ESI thematic period 
``Numerical Analysis of Complex PDE Models in the Sciences''. 
Research of CS supported in part by the Swiss National Science Foundation.
}
}
\author{Lehel Banjai\thanks{Maxwell Institute for Mathematical Sciences, School of Mathematical  
        \& Computer Sciences, Heriot-Watt University, Edinburgh EH14 4AS, UK  (\texttt{l.banjai@hw.ac.uk}).}
\and 
Jens M. Melenk\thanks{Institut  f\"{u}r  Analysis und  Scientific Computing, 
Technische Universit\"{a}t Wien, A-1040  Vienna, Austria (\texttt{melenk@tuwien.ac.at}).}
\and
Christoph Schwab\thanks{Seminar for Applied Mathematics, ETH Z\"{u}rich, ETH Zentrum, HG  G57.1, 
                        CH8092 Z\"{u}rich, Switzerland (\texttt{christoph.schwab@sam.math.ethz.ch}).}}

\date{Draft version of \today.}
\begin{document}
\maketitle
\begin{abstract}
In bounded, polygonal domains $\Omega\subset \bbR^2$ with 
Lipschitz boundary $\partial\Omega$ consisting of a finite number of 
Jordan curves admitting analytic parametrizations,
we analyze $hp$-FEM discretizations of linear, second order,
singularly perturbed reaction diffusion equations 
on so-called \emph{geometric boundary layer meshes}. 
We prove, under suitable analyticity assumptions on the
data, that these $hp$-FEM afford exponential convergence
in the natural ``energy'' norm of the problem,
as long as the geometric boundary layer 
mesh can resolve the smallest length scale present in the problem.
Numerical experiments confirm the robust exponential convergence
of the proposed $hp$-FEM.
\end{abstract}


%
\section{Introduction}
\label{S:intro}
The need for accurate numerical approximations of solutions
to singularly perturbed partial differential equations
in nonsmooth domains arises in a wide range of applications.
Higher order numerical methods must cope with the appearance
of \emph{boundary layers} and their interaction with 
\emph{geometric corner and edge singularities}. 
They are due to length scales introduced into weak solutions
by small or large parameters in the differential operator.
Accordingly, a large body of numerical analysis research has
been developed during the past decades on their efficient
numerical resolution; 
we mention only the texts 
\cite{MillerRiordShishk96,RoosStynsTobiska2ndEd}
and the references there for reaction-advection-diffusion problems,
and \cite{SchotzScSt_NM1999} for viscous, incompressible flow. 
The discretization methods presented and surveyed
in \cite{RoosStynsTobiska2ndEd} are of fixed order and 
of Finite Difference or Finite Element type.
They account specifically for the appearance of boundary and interior layers 
in solutions of the singularly perturbed boundary value problems.
Being of fixed order, the corresponding discretization methods
can afford at best \emph{fixed, algebraic orders of convergence}
whose convergence is, however, \emph{robust}:
the constants implicit in {\sl a priori} error bounds are independent 
of the singular perturbation parameters and, hence, 
of the physical length scales in solutions that are implied by the singular
perturbation of the governing equations.

As it is well-known, elliptic boundary value problems
in domains $\Omega$ with piecewise analytic boundary 
for differential operators 
with analytic in $\overline{\Omega}$ coefficients and
forcing terms admit \emph{exponential convergence rates} 
by Galerkin approximations with local mesh refinement
and concurrent, judicious increase of the polynomial degree. 
This so-called \emph{$hp$-Finite Element approach}
has been analyzed in a series of papers, 
see \cite{BabGuoCurved1988,phpSchwab1998}
and, more recently, in \cite{SSW13_1129,SS15_2016,SS18_2263},
and the references there, for regular elliptic boundary value
problems.

The study of parametric regularity and 
the proof of \emph{robust, algebraic convergence rates} 
of discretizations for singularly perturbed, elliptic
boundary value problems on polygons 
seems to have been initiated
by G.I.~Shiskhin in the 1980s \cite{ShishCrnerBlr87}.

For singular perturbation problems, corresponding results
on robust \emph{exponential} convergence rates
and 
corresponding analytic regularity estimates for
solutions have been obtained in a series of papers 
in the 1990s, see
\cite{schwab-suri96,melenk-schwab98,M97_738,MS99_325,melenk02} 
and the references there. 
The results from \cite{melenk-schwab98,M97_738,MS99_325} were
restricted to domains $\Omega$ with smooth (analytic) boundary $\partial \Omega$.
In \cite{MelCS_RegSingPert} analytic regularity results for elliptic reaction-diffusion
problems in two space dimensions that are uniform with respect to the perturbation parameter
were obtained. These results allowed us in \cite{melenk-schwab98} to prove
\emph{robust exponential convergence} of $hp$-FEM with so-called 
two-element and geometric boundary layer meshes,
in domains $\Omega$ with a smooth (analytic) boundary $\partial\Omega$.  
In \cite{melenk02}, an analytic regularity theory for solutions of  
linear, singularly perturbed elliptic reaction-diffusion problems in polygonal
domains $\Omega$ as depicted in Fig.~\ref{fig:curvilinear-polygon} was developed. 
In these works, elliptic singular perturbations with
exactly one small length scale were considered.
Examples of elliptic singular perturbation problems where
several characteristic length scales appear simultaneously
comprise in particular dimensionally reduced models of curved
thin solids (``shells'') in so-called ``bending-dominated''
states (see, e.g., \cite{DaugeShellBl} and the references there for
a detailed discussion of possible length scales), 
and linear, elliptic 
reaction-diffusion boundary value problems that result
from implicit time-discretizations of parabolic evolution
equations and, more recently, 
from discretizations of fractional powers of elliptic
operators (see, e.g., \cite{BMNOSS17_732,melenk2019hpfem}
and the references there).
Further applications of the presently developed results arise in
electromagnetics in so-called eddy-current models (where the small parameter
is a complex number) (see, e.g., \cite{buret-etal12} and the references there),
and in $hp$-FEM for advection-diffusion problems where stability, in addition
to consistency, is a major issue 
(see, e.g., \cite{MS99_325,demkowicz-gopalakrishnan-niemi12}).

The present work extends in the reaction-diffusion case 
the $hp$-error analysis of Part I \cite{melenk-schwab98} 
to polygons and addresses the question how to 
approximate problems with multiple scales. 
The meshes presented here are also appropriate in situations
where the precise length scales are unavailable, as is the case for example
in problems from computational mechanics, see, e.g. \cite{DaugeShellBl}.

In the following Section~\ref{sec:PrbForm} we present a model singularly perturbed linear elliptic diffusion problem,
where the perturbation parameter $0< \varepsilon \leq 1$ dictates the 
\emph{single boundary length scale} $\varepsilon$. 
We hasten to add that the ensuing $hp$-approximation results hold,
independent of the particular model problem, and apply to a wider range of
singular perturbations, as plate and shell models with possibly
multiple length scales. See, e.g., 
\cite{ArnFalkBLRM1996,GMSShell1998,DaugeShellBl,TemamSgPrt2018}
and the references there.
\subsection{Model reaction-diffusion problem}
\label{sec:PrbForm}
In a bounded domain $\Omega\subset \mathbb{R}^2$, 
which is assumed to be scaled to unit size, and for a parameter $0<\varepsilon \leq 1$,
we consider the $hp$-FE approximation of the 
\emph{model reaction-diffusion Dirichlet problem}
\begin{equation} \label{eq:sg-per}
	-\varepsilon^2 \nabla\cdot \left( A(\bx) \nabla u_\varepsilon\right)
	+ c(\bx) u_\varepsilon = f \quad \mbox{ in $\Omega$}, 
	\qquad u_\varepsilon = 0 \quad\mbox{ on $\partial\Omega$}.
\end{equation}
We assume 
\begin{equation}\label{eq:sg-per-asmp}
\begin{array}{l}
	\text{ $A$, $c$, $f$ analytic on $\overline{\Omega}$, independent of $\varepsilon$,} 
        \\
         \text{$A$ symmetric, positive definite uniformly 
               in $\overline{\Omega}$, $c \ge c_0 > 0$ on $\overline{\Omega}$.}
\end{array}
\end{equation}
 \begin{figure}
\psfragscanon
\psfrag{A01}{\tiny $A_{J_1}^{(1)} = A_{0}^{(1)}$}
\psfrag{A11}{\tiny $A_{1}^{(1)}$}
\psfrag{A21}{\tiny $A_{2}^{(1)}$}
\psfrag{A31}{\tiny $A_{3}^{(1)}$}
\psfrag{Ajm11}{\tiny $A_{J_1-1}^{(1)}$}
\psfrag{G11}{\tiny $\Gamma_1^{(1)}$}
\psfrag{G21}{\tiny $\Gamma_2^{(1)}$}
\psfrag{G31}{\tiny $\Gamma_3^{(1)}$}
\psfrag{Gjm11}{\tiny $\Gamma_{J_1-1}^{(1)}$}
\psfrag{O01}{\tiny $\omega_{J_1}^{(1)}$}
\psfrag{A02}{\tiny $A_0^{(2)}$}
\psfrag{A12}{\tiny $A_1^{(2)}$}
\psfrag{G02}{\tiny $\Gamma_1^{(2)}$}
\psfrag{G12}{\tiny $\Gamma_2^{(2)}$}
 \begin{overpic}[width=0.5\textwidth]{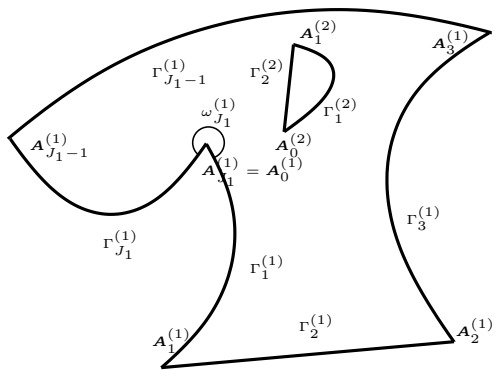}
\put(40,40){\tiny $\bA_{J_1}^{(1)} = \bA_{0}^{(1)}$}
\put(30,05){\tiny $\bA_{1}^{(1)}$}
\put(92,7){\tiny $\bA_{2}^{(1)}$}
\put(87,66){\tiny $\bA_{3}^{(1)}$}
\put(5,45){\tiny $\bA_{J_1-1}^{(1)}$}
\put(50,20){\tiny $\Gamma_1^{(1)}$}
\put(60,8){\tiny $\Gamma_2^{(1)}$}
\put(82,30){\tiny $\Gamma_3^{(1)}$}
\put(30,60){\tiny $\Gamma_{J_1-1}^{(1)}$}
\put(20,25){\tiny $\Gamma_{J_1}^{(1)}$}
\put(40,52){\tiny $\omega_{J_1}^{(1)}$}
\put(55,45){\tiny $\bA_0^{(2)}$}
\put(60,68){\tiny $\bA_1^{(2)}$}
\put(65,52){\tiny $\Gamma_1^{(2)}$}
\put(50,60){\tiny $\Gamma_2^{(2)}$}
\end{overpic}
\psfragscanoff
\caption{\label{fig:curvilinear-polygon} Example of a curvilinear polygon $\Omega$.}
\end{figure}
To design $H^1_0(\Omega)$-conforming,
$hp$-FE approximations of the solutions $\{ u_\eps : 0<\eps\leq 1 \}$ 
of \eqref{eq:sg-per}
under the analyticity assumptions \eqref{eq:sg-per-asmp} which converge
exponentially in the $\eps$-dependent energy 
norm $\| \circ \|_{\eps,\Omega}$ given by 
\begin{equation}\label{eq:epsnrm}
\| v \|_{\eps,\Omega}^2 
:= 
\eps^2 \| \nabla v \|^2_{L^2(\Omega)} + \| v \|^2_{L^2(\Omega)}\;,\quad v\in H^1(\Omega), 
\end{equation}
is the purpose of the present paper.
We prove in particular in Theorem~\ref{thm:singular-approx} 
a \emph{robust exponential convergence error bound}, i.e., 
all constants in the exponential convergence bound do not 
depend on $\eps > 0 $.

We again emphasize that we consider \eqref{eq:sg-per} for illustration. 
The scope of
the robust exponential $hp$ convergence rate bounds
below extends well beyond \eqref{eq:sg-per} to more 
complex, singularly perturbed PDE such as those in
\cite{ArnFalkBLRM1996,GMSShell1998,DaugeShellBl}.
\subsection{Geometric Preliminaries}
\label{sec:GeoPrel}
%
In \eqref{eq:sg-per-asmp}, the domain $\Omega \subset \bbR^2$ 
is a curvilinear polygonal domain,
schematically depicted in Fig.~\ref{fig:curvilinear-polygon}.
Specifically, 
the boundary $\partial \Omega$ is assumed to consist 
of $J \in \bbN$ closed  curves $\Gamma^{(i)}$. 
Each curve $\Gamma^{(i)}$ in turn is assumed to be partitioned 
into finitely many open, disjoint, \emph{analytic arcs} $\Gamma^{(i)}_j$, 
in the sense that there are numbers $J_i \in \bbN$ such that
$$
\overline{\Gamma^{(i)}} = \bigcup_{j=1}^{J_i} \overline{\Gamma^{(i)}_j}\;,
\quad i=1,\hdots,J\;.
$$
Here, \emph{analytic arcs} $\Gamma^{(i)}_j$
admit \emph{nondegenerate, analytic parametrizations}, i.e.,
$$
\Gamma^{(i)}_j
=
\left\{ {\mathbf x}^{(i)}_j (\theta)  | \theta \in (0,1) \right\}\;,
\quad 
i=1,...,J, \quad j=1,...,J_i \;.
$$
with the coordinate functions 
$x_j^{(i)}$, $y_j^{(i)}$ of ${\mathbf x}^{(i)}_j(\theta) = (x_j^{(i)}(\theta),y_j^{(i)}(\theta))$
assumed to be (real) analytic functions of $\theta \in  [0,1]$ and such that 
$$
\min_{\theta \in [0,1]}
\left\{
\left| \frac{d}{d\theta} x_j^{(i)}(\theta) \right|^2 
+ 
\left| \frac{d}{d\theta} y_j^{(i)}(\theta) \right|^2 
\right\}
> 0
\quad 
j=1,...,J_i, 
\quad 
i=1,...,J
\;.
$$
We denote $\partial \Gamma^{(i)}_j = \{ \bA^{(i)}_{j-1}, \bA^{(i)}_j\}$ 
where
$\bA^{(i)}_{j-1} = {\mathbf x}_j^{(i)}(0)$ and $\bA^{(i)}_j = {\mathbf x}_j^{(i)}(1)$. 
For each boundary component $\Gamma^{(i)}$,
we enumerate $\{ \bA^{(i)}_j \}_{j=1}^{J_i}$ 
cyclically, counterclockwise by indexing with $j$ modulo $J_i$, 
thereby identifying in particular $\bA^{(i)}_{j} := \bA^{(i)}_{j+J_i}$.
%
The interior angle at $A_j^{(i)}$ is denoted $\omega_j^{(i)} \in (0,2\pi)$. 
For notational simplicity, we assume henceforth that $J=1$, 
i.e., $\partial\Omega$ consists of a single component of connectedness. 
We write $\bA_j = \bA_j^{(1)}$, $\Gamma_j$ for $\Gamma_j^{(1)}$, 
$x_j = x_j^{(1)}$, $y_j = y_j^{(1)}$. 
%
Then, in a vicinity of any point $x\in \partial \Omega$, a curvilinear polygon
$\Omega$
is \emph{analytically diffeomorphic} to either a half-space, or to a plane 
sector with vertex situated at the origin.
\subsection{Contributions}
\label{S:Contrib}
The principal contribution of the present paper is 
\emph{robust, exponential convergence of a class of $hp$-FEM approximations}
of the singular perturbation problem \eqref{eq:sg-per} under the 
analyticity assumptions \eqref{eq:sg-per-asmp}, in curvilinear polygons 
$\Omega$ as described in Section~\ref{sec:GeoPrel}.
The convergence proof 
in Sections~\ref{S:AppRefElt} and \ref{sec:sing-approx-geo-mesh}
is done under a \emph{scale resolution condition} that corresponds, 
roughly speaking, to the $hp$-FE partitions resolving the shortest length
scale that occurs in the solution $u_\eps$, and it
is strongly based on \emph{parameter-explicit, analytic regularity results} for the
parametric solution family $\{ u_\eps: 0< \eps \leq 1 \}\subset H^1_0(\Omega)$
of \eqref{eq:sg-per}, which were obtained by one of the authors in \cite{melenk02}.
Importantly, and distinct from earlier work on robust exponential $hp$-FE convergence for 
\eqref{eq:sg-per}, a \emph{patch-based convergence proof} is developed which 
also enables an algorithmic, patchwise-structured anisotropic mesh specification,
described in Section~\ref{sec:MacTriGBLMes}, which is applicable in domains
$\Omega$ of the generality admitted in Section~\ref{sec:GeoPrel}. 
As we 
show in numerical experiments in Section~\ref{S:NumExp}, the
recently developed, automatic mesh generator Netgen \cite{netgen1} does 
produce automatically, i.e., without ``expert pruning'', 
anisotropic, geometric meshes in $\Omega$ with the required boundary and 
corner refinement that satisfy the requirement of the present robust, 
exponential convergence bounds.
\subsection{Outline of this paper}
\label{S:Struct}
In Section~\ref{sec:MacTriGBLMes},
we introduce the geometric mesh families in $\Omega$ 
that underlie our robust exponential convergence results. 
The meshes require concurrent anisotropic geometric partitions 
of $\Omega$ towards the boundary $\partial \Omega$ 
and isotropic geometric refinement towards the corners $\bA_j$.
We define these meshes in a macro-element fashion based on 
an initial, coarse regular partition of the physical domain $\Omega$
into a \emph{macro-triangulation} consisting of a 
\emph{regular, finite, and fixed partition of the physical domain $\Omega$} 
that is described in Section~\ref{sec:MacroTri}.
Its elements will be referred to
as \emph{(macro) patches} and are assumed to be images of a finite number of
quadrilateral reference patches under \emph{analytic patch maps}.
The reference patches are key in ensuring robust exponential convergence
rate bounds of our $hp$-FEM approximation.
Following earlier work \cite{M97_738,MS99_325,melenk02}, 
we consider so-called \emph{geometric boundary layer meshes}, 
denoted by $\Tg$.
We introduce these in Def.~\ref{def:bdylayer-mesh}. 
Unlike the so-called ``two-element'' meshes 
considered earlier in \cite{SSX98_321,schwab-suri-xenophontos98}, 
which are designed to approximate only a single small scale
in a robust way, the presently considered geometric boundary layer meshes
afford robust exponential convergence rates of $hp$-FEM also 
in the presence of multiple physical length scales. 
This situation arises in a number of applications
(e.g., \cite{DaugeShellBl,BMNOSS17_732}).
Section~\ref{S:GeoBL} introduces the geometric boundary
layer mesh, first on the reference patches, and then in 
Section~\ref{sec:generate-bdy-layer-mesh} in curvilinear polygons.

Section~\ref{S:AppRefElt} is devoted to the polynomial approximation 
of functions on geometric boundary layer meshes.  
The approximation is based on Gauss-Lobatto interpolation operators in 
the reference triangle, indicated by $\triangle$, and in the 
reference square, indicated by $\Box$, 
in Section~\ref{sec:approx-reference-element}.
These are then assembled into (nodal) patch approximation operators on the 
geometric boundary layer patches in Section~\ref{sec:approx-reference-patches}.
Then, the robust $hp$-approximation of corner singularities and 
boundary layer functions is proved. 
These functions are the key solution components 
of singularly perturbed problems in polygons such as the model problem
(\ref{eq:sg-per}). 

Section~\ref{sec:sing-approx-geo-mesh} assembles the patch $hp$-approximation
results and interpolants into a global approximation operator, 
and presents the main result of this paper: robust exponential convergence rate
bounds for the global $hp$-interpolation of the solution of (\ref{eq:sg-per}) 
assembled from the patch approximations.

Section~\ref{S:NumExp} presents several illustrative numerical experiments
in polygons, which underline the theoretical results. They are based 
on $hp$-meshes that are furnished by the automated mesh generation 
procedure Netgen \cite{netgen1} to make the point that the
somewhat technical construction of geometric boundary layer meshes
is, in principle, available and feasible automatically.
Appendix~\ref{AppA:AnChVar} contains proofs of auxiliary results on 
analytic regularity estimates under analytic changes of variables.
Appendix~\ref{sec:AppB} collects (mostly known)
results on univariate polynomial approximation 
for convenient reference in the main text.
\subsection{Notation}
\label{S:Notat}
We employ standard notation for Sobolev spaces. 
Constants $C$, $\gamma$, $b > 0$ may be different in different instances.
However, they will be independent of parameters of interest such as 
$\varepsilon \in (0,1]$ and the polynomial degree $q$. 
The notation $\nabla^n u$ stands for the collection of all partial derivatives of order $n$
and $|\nabla^n u|^2 = \sum_{|\alpha| = n}\frac{n!}{\alpha!} |D^\alpha u|^2$. 
Points in ${\mathbb R}^2$ will be denoted depending on the context as either
$\bx = (x,y)$ (physical domain) or $\widetilde{\bx}= (\xt,\yt)$ (patch domains) 
or $\widehat{\bx} = (\xh,\yh)$ (reference domains).
We abbreviate $\{\xt = 0\}$, $\{\yt = 0\}$, and $\{\xt = \yt\}$ 
for the line segments $\{\widetilde \bx = (0,\yt)\,|\, 0 < \yt < 1\}$,  
$\{\widetilde \bx = (\xt,0)\,|\, 0 < \xt < 1\}$, and 
$\{\widetilde \bx = (\xt,\yt)\,|\, 0 < \xt = \yt < 1\}$, respectively. 
We write $\{\yt\leq \xt\}$ for 
$\{\widetilde \bx = (\xt,\yt)\,|\,  0 < \yt \leq  \xt < 1\}$. 
The origin will be denoted $\bO=(0,0)$. 
The reference square and triangle are $\widehat{S}:= (0,1)^2$ and 
$\widehat T:= \{(\xh,\yh)\,|\, 0 < \xh < 1, 0 < \yh < \xh\}$. 
The region covered by the reference patch will be denoted 
$\widetilde S:= \widehat S = (0,1)^2$. It will be convenient 
to introduce $\widetilde T:= \widehat T$ and 
set $\Tf:= \widetilde S \setminus \overline{\widetilde T}$. 
We denote the space of polynomials of total degree $q$ by 
$\bbP_q = \operatorname{span} \{ x^i  y^j\,|\, 0 \leq i + j \leq q\}$; 
the tensor product space $\bbQ_q$ is 
$\bbQ_q = \operatorname{span} \{ x^i  y^j\,|\, 0 \leq i , j \leq q\}$.
\section{Macro triangulation. Geometric boundary layer mesh}
\label{sec:MacTriGBLMes}
\begin{figure}[h]
\psfragscanon
\psfrag{T}{$T$}
\psfrag{B}{$B$}
\psfrag{M}{$M$}
\psfrag{C}{$C$}
\begin{overpic}[width=0.25\textwidth]{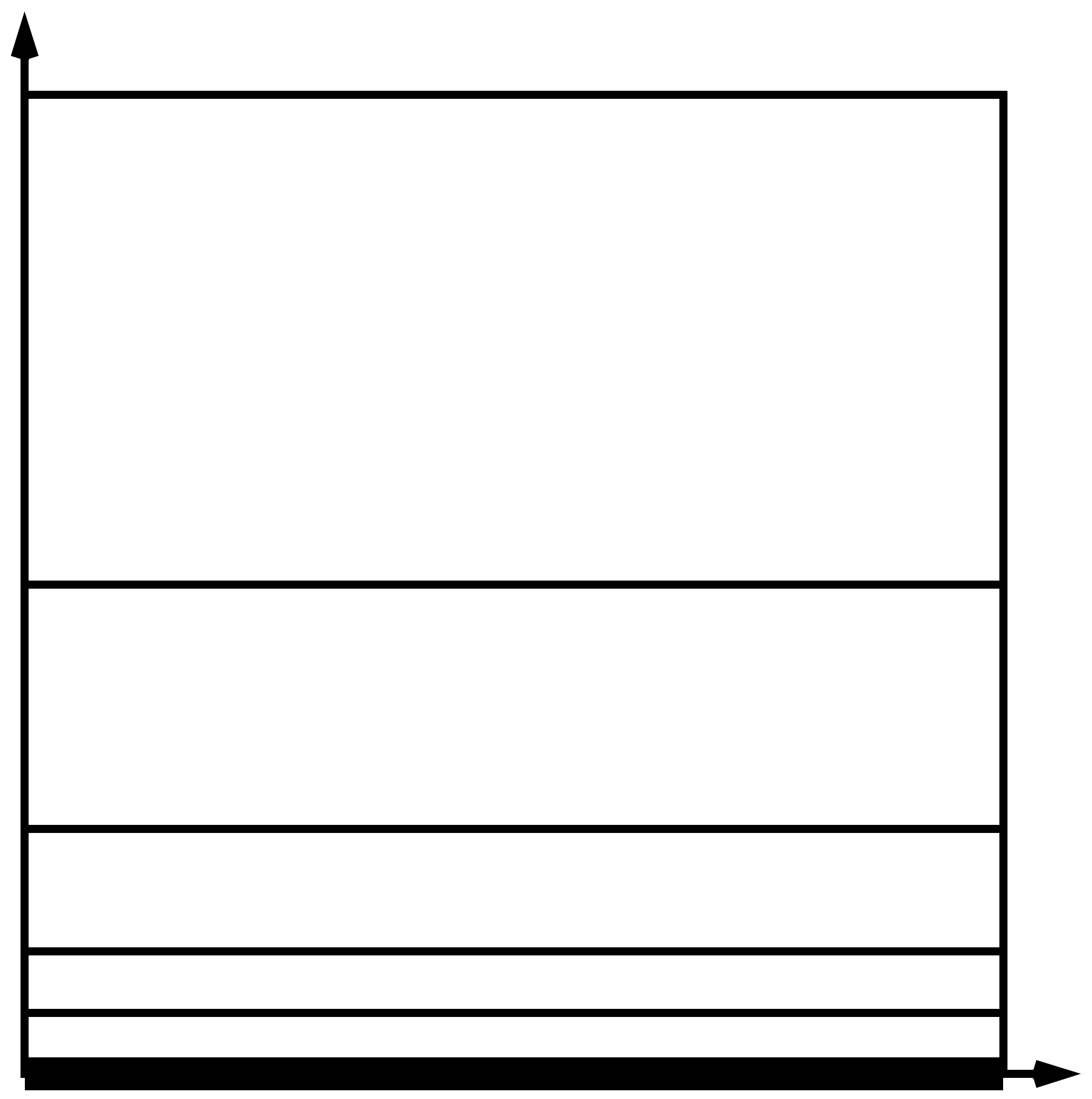}
\put(45,70){$ \widetilde{\calT}^{\BL,L}_{geo,\sigma}$}
\put(-6,90){$\yt$}
\put(95,-5){$\xt$}
\end{overpic}
\hfill 
\begin{overpic}[width=0.25\textwidth]{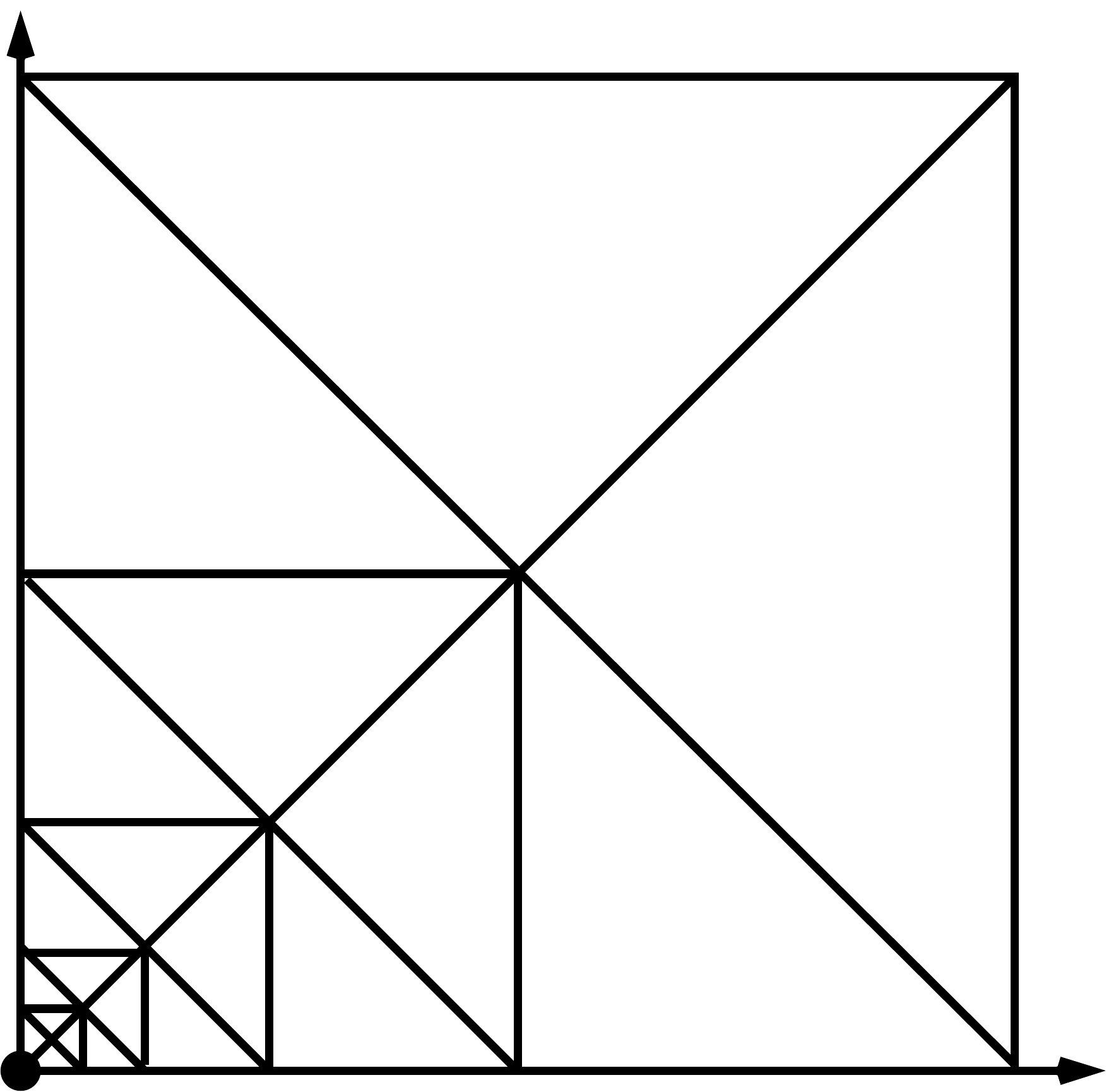}
\put(40,70){$ \widetilde{\calT}^{\Co,n}_{geo,\sigma}$}
\put(-6,90){$\yt$}
\put(95,-7){$\xt$}
\end{overpic}
\hfill 
\begin{overpic}[width=0.25\textwidth]{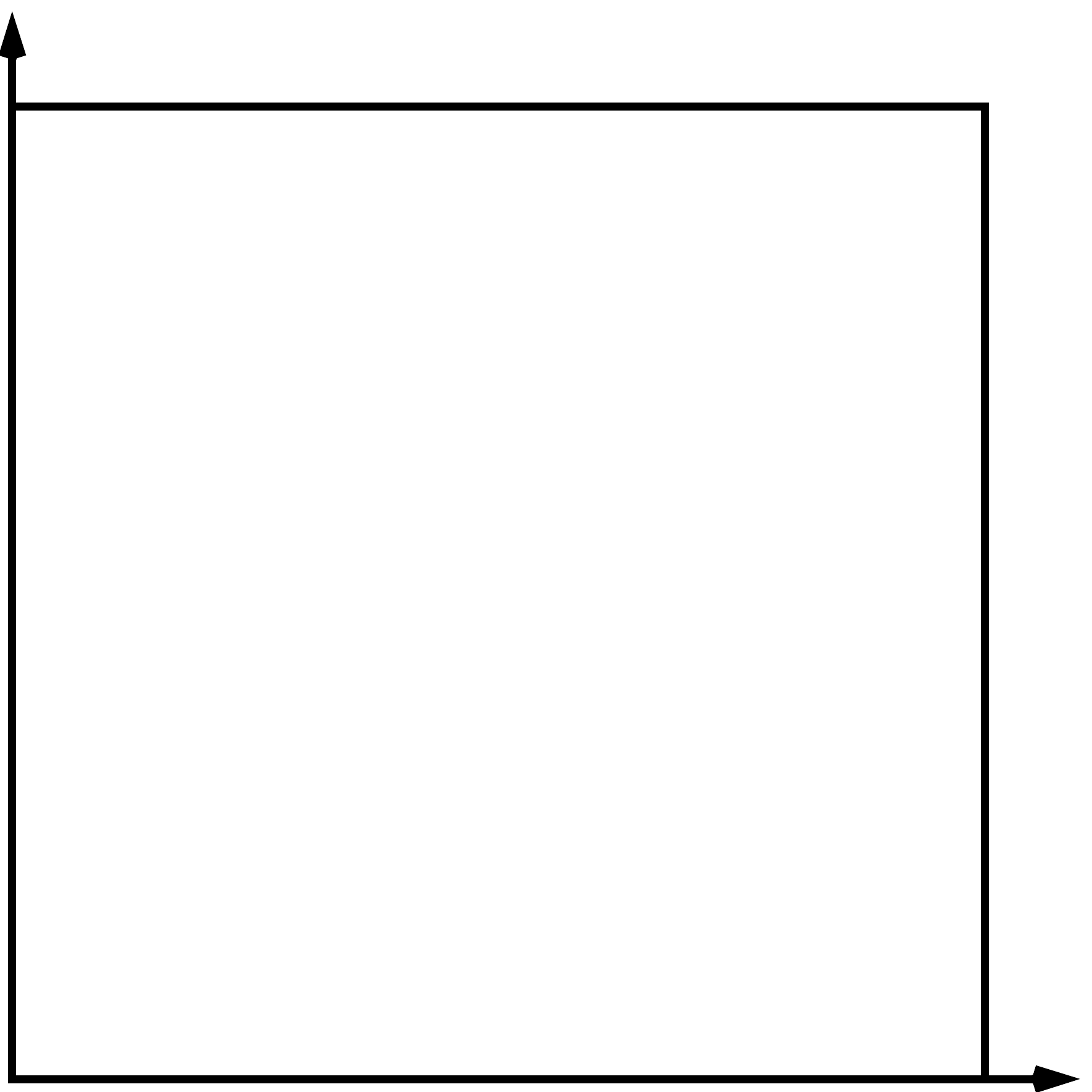}
\put(25,70){trivial patch}
\put(-6,90){$\yt$}
\put(95,-8){$\xt$}
\end{overpic}
\begin{overpic}[width=0.45\textwidth]{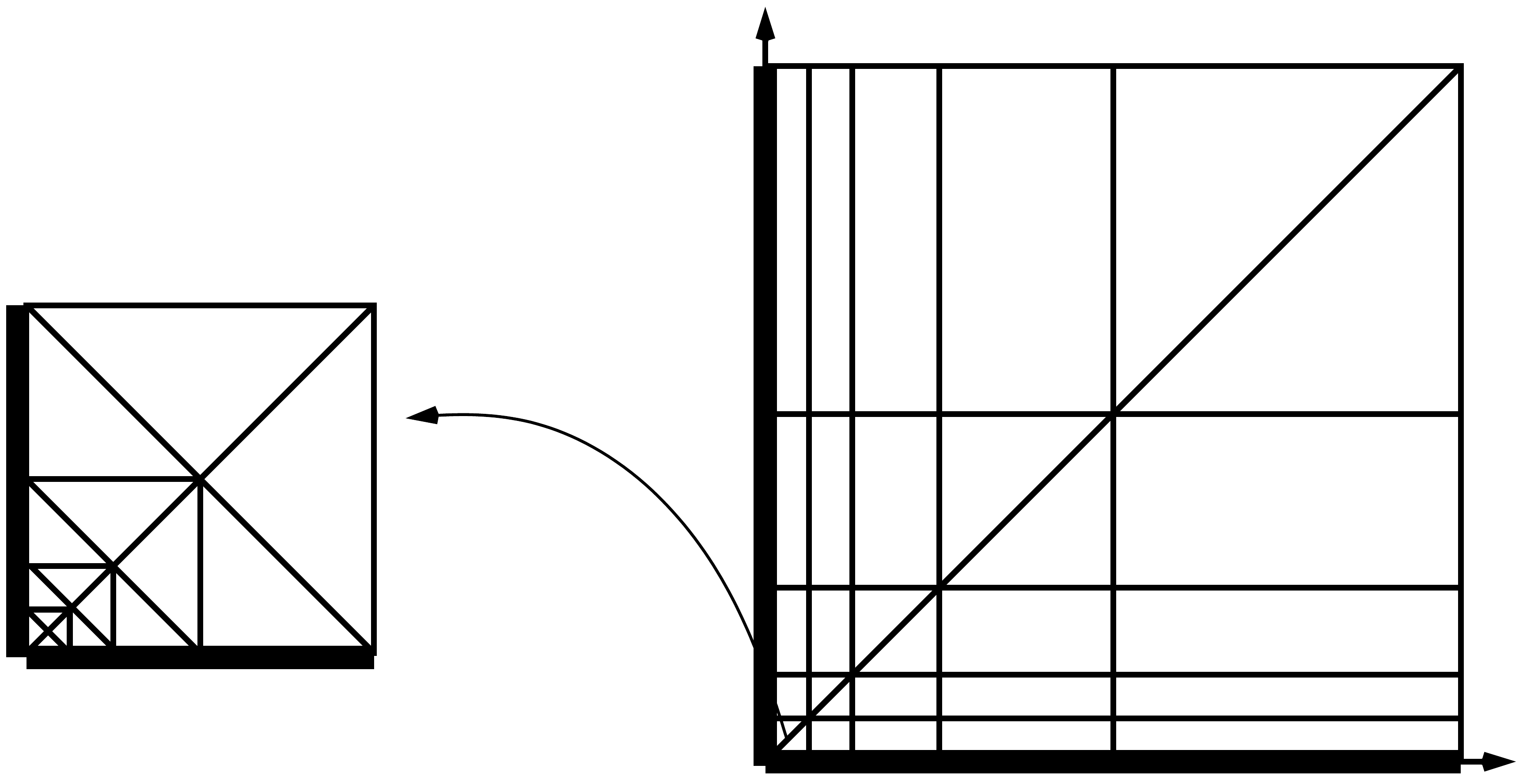}
\put(70,40){$ \widetilde{\calT}^{\Te,L,n}_{geo,\sigma}$}
\put(0,0){$ \widetilde S_1 = (0,\sigma^L)^2$}
\put(45,45){$\yt$}
\put(95,-5){$\xt$}
\end{overpic}
\hfill 
\begin{overpic}[width=0.45\textwidth]{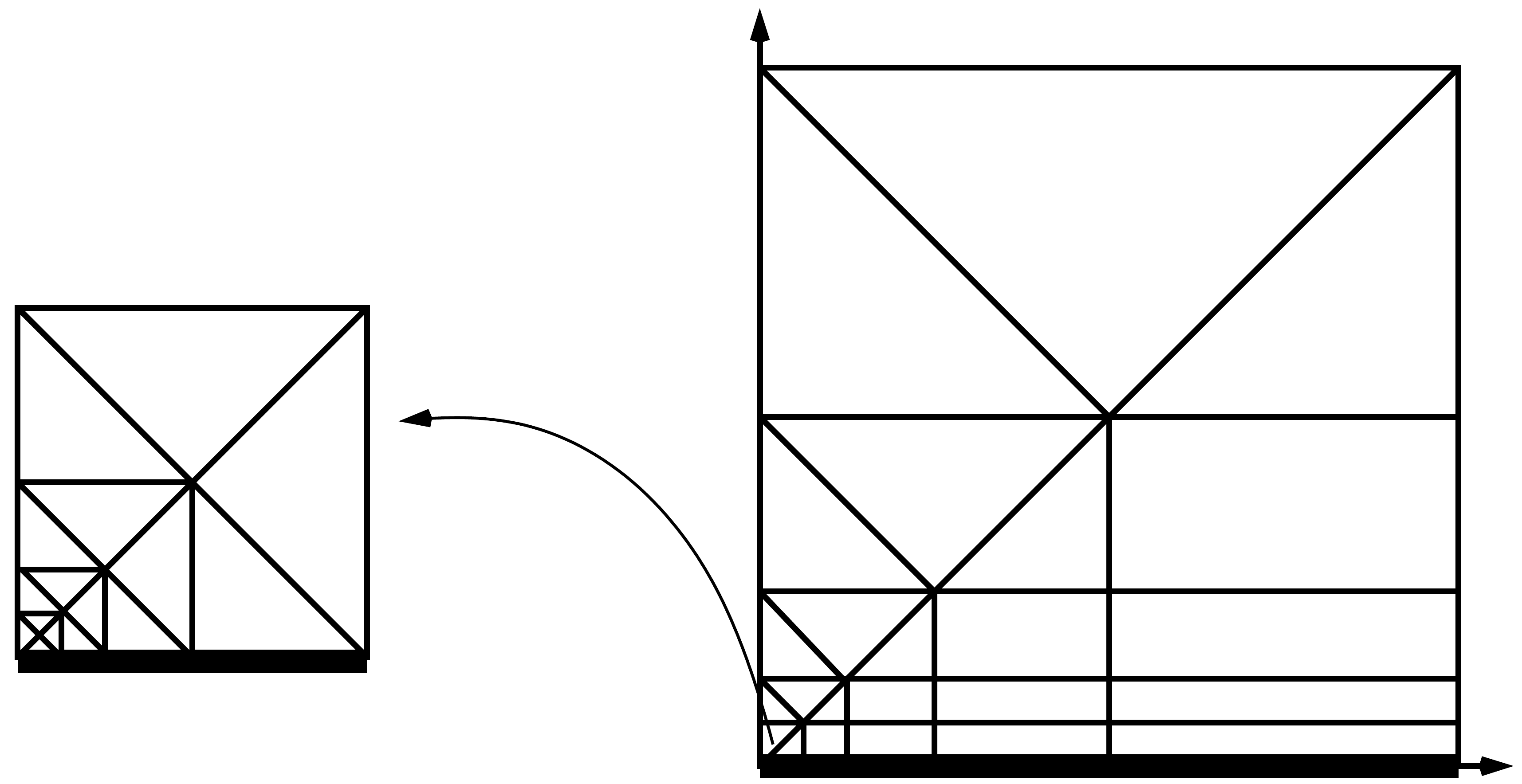}
\put(70,40){$ \widetilde{\calT}^{\Mi,L,n}_{geo,\sigma}$}
\put(0,0){$ \widetilde S_1 = (0,\sigma^L)^2$}
\put(45,45){$\yt$}
\put(95,-5){$\xt$}
\end{overpic}
\psfragscanoff
\caption{
\label{fig:patches} 
Catalog ${\mathfrak P}$ of mesh patches of geometric boundary layer meshes $\tilde{\calT}_{geo}$.
Top row: boundary layer patch $\widetilde{\calT}^{\BL,L}_{geo,\sigma}$
with $L$ layers of geometric refinement towards $\{\yt=0\}$; 
	corner patch $\widetilde{\calT}^{\Co,n}_{geo,\sigma}$
with $n$ layers of geometric refinement towards $(0,0)$; trivial patch. 
Bottom row: 
	tensor patch $\widetilde{\calT}^{\Te,L,n}_{geo,\sigma}$
with $n$ layers of isotropic geometric refinement towards $(0,0)$ and 
$L$ layers of anisotropic geometric refinement towards $\{\xt = 0\}$ and $\{\yt=0\}$; 
	mixed patch $\widetilde{\calT}^{\Mi,L,n}_{geo,\sigma}$
with $L$ layers of refinement towards $\{y=0\}$ and $n$ layers of refinement towards $(0,0)$. 
Geometric entities shown in boldface indicate 
parts of $\partial \widetilde S$ that are mapped to $\partial\Omega$.  
Patch meshes are transported into the curvilinear polygon $\Omega$ shown 
in Fig.~\ref{fig:curvilinear-polygon} via analytic patch maps $F_{K^{\M}}$.
}%
\end{figure}
\begin{figure}[h]
\begin{overpic}[width=0.4\textwidth]{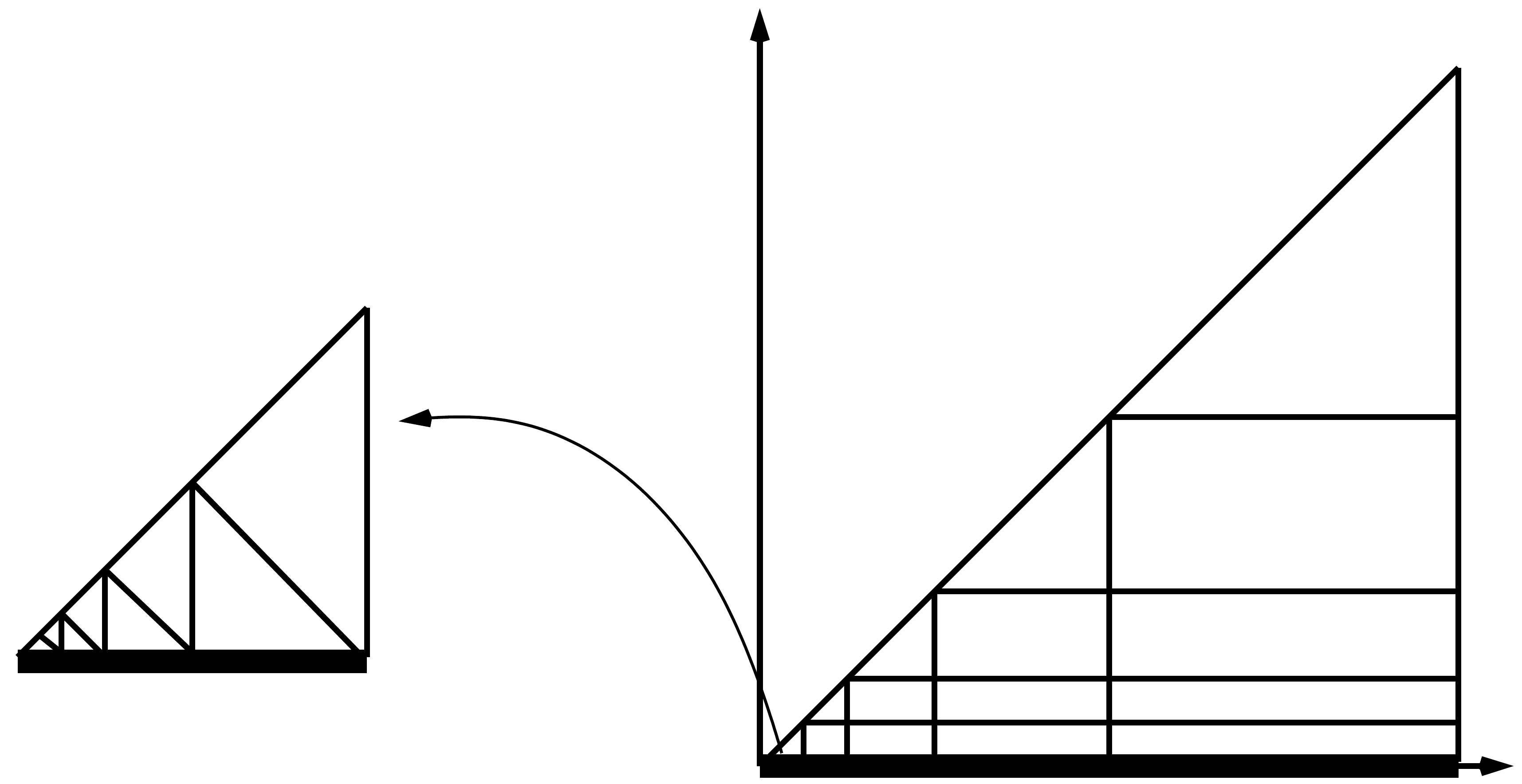}
\put(60,35){$ \widetilde{\calT}^{\Mi,\text{\rm half},L,n}_{geo,\sigma}$}
\put(0,38){\small $\widetilde T_1:= \{0 <\xt < \sigma^L,$}
\put(16,32){\small $0 <\yt < \xt\}$}
\put(80,15){$\widetilde T$}
\put(98,-4){\small $\xt$}
\put(52,45){\small $\yt$}
\end{overpic}
\hfill 
\begin{overpic}[width=0.2\textwidth]{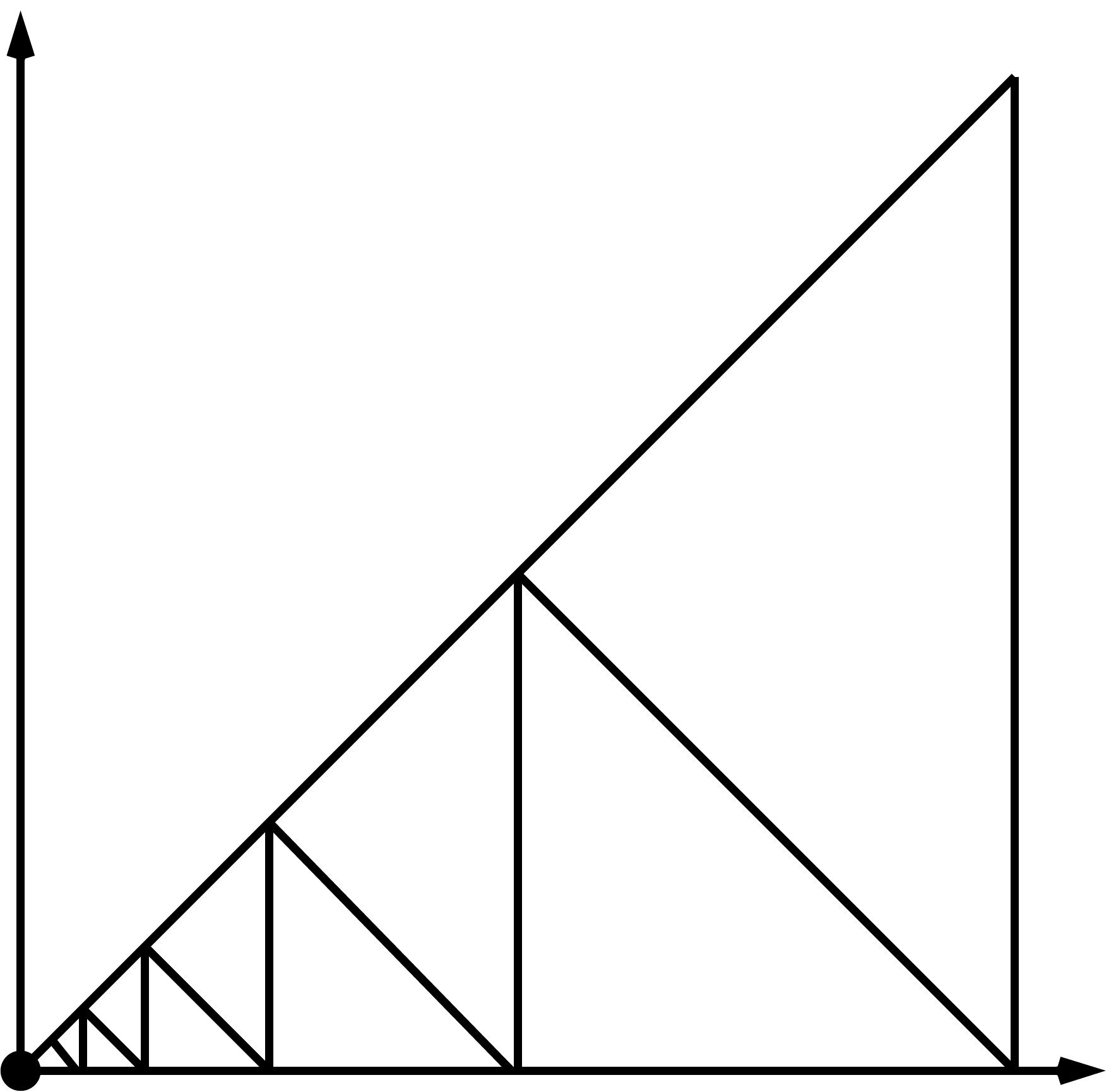}
\put(40,70){$ \widetilde{\calT}^{\Co,\text{\rm half},n}_{geo,\sigma}$}
\put(80,15){$\widetilde T$}
\put(95,-8){\small $\xt$}
\put(-6,90){\small $\yt$}
\end{overpic}
\hfill 
\begin{overpic}[width=0.2\textwidth]{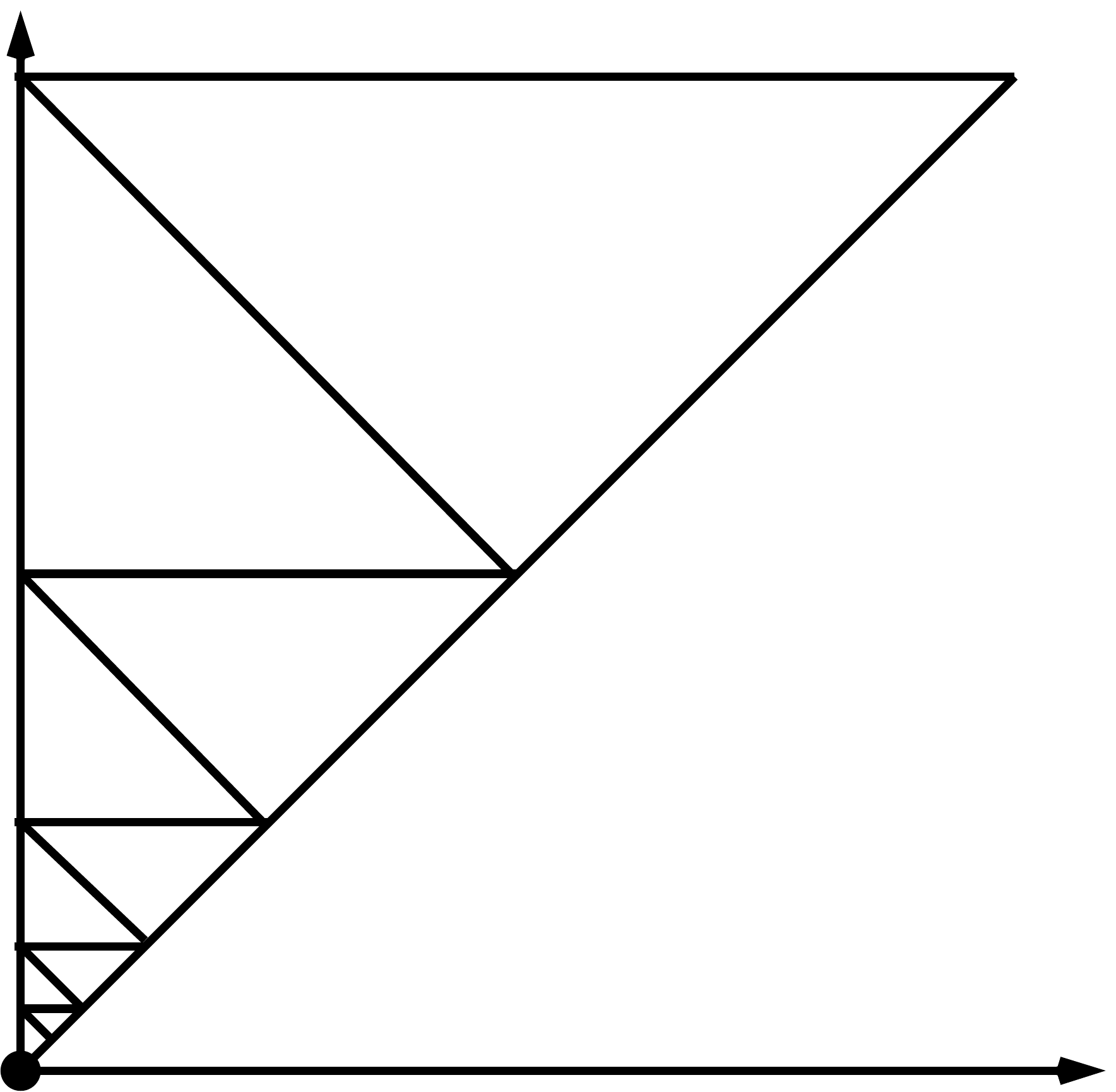}
\put(45,70){$ \widetilde{\calT}^{\Co,\text{\rm half},\text{\rm flip},n}_{geo,\sigma}$}
\put(95,-8){\small $\xt$}
\put(-7,90){\small $\yt$}
\end{overpic}
\caption{\label{fig:half-figures} From left to right: half-patches 
$ \widetilde{\calT}^{\Mi,\text{\rm half},L,n}_{geo,\sigma}$, 
$ \widetilde{\calT}^{\Co,\text{\rm half},n}_{geo,\sigma}$, and 
$ \widetilde{\calT}^{\Co,\text{\rm half},\text{\rm flip},n}_{geo,\sigma}$. They are  
given by the elements of 
$ \widetilde{\calT}^{\Co,n}_{geo,\sigma}$ and $ \widetilde{\calT}^{\Mi,L,n}_{geo,\sigma}$
below the diagonal $\{\yt = \xt\}$ and the mirror image of 
$\widetilde{\calT}^{\Co,\text{\rm half},n}_{geo,\sigma}$ at the diagonal $\{\yt = \xt\}$. 
}
\end{figure}
\begin{figure}
\begin{overpic}[width=0.25\textwidth]{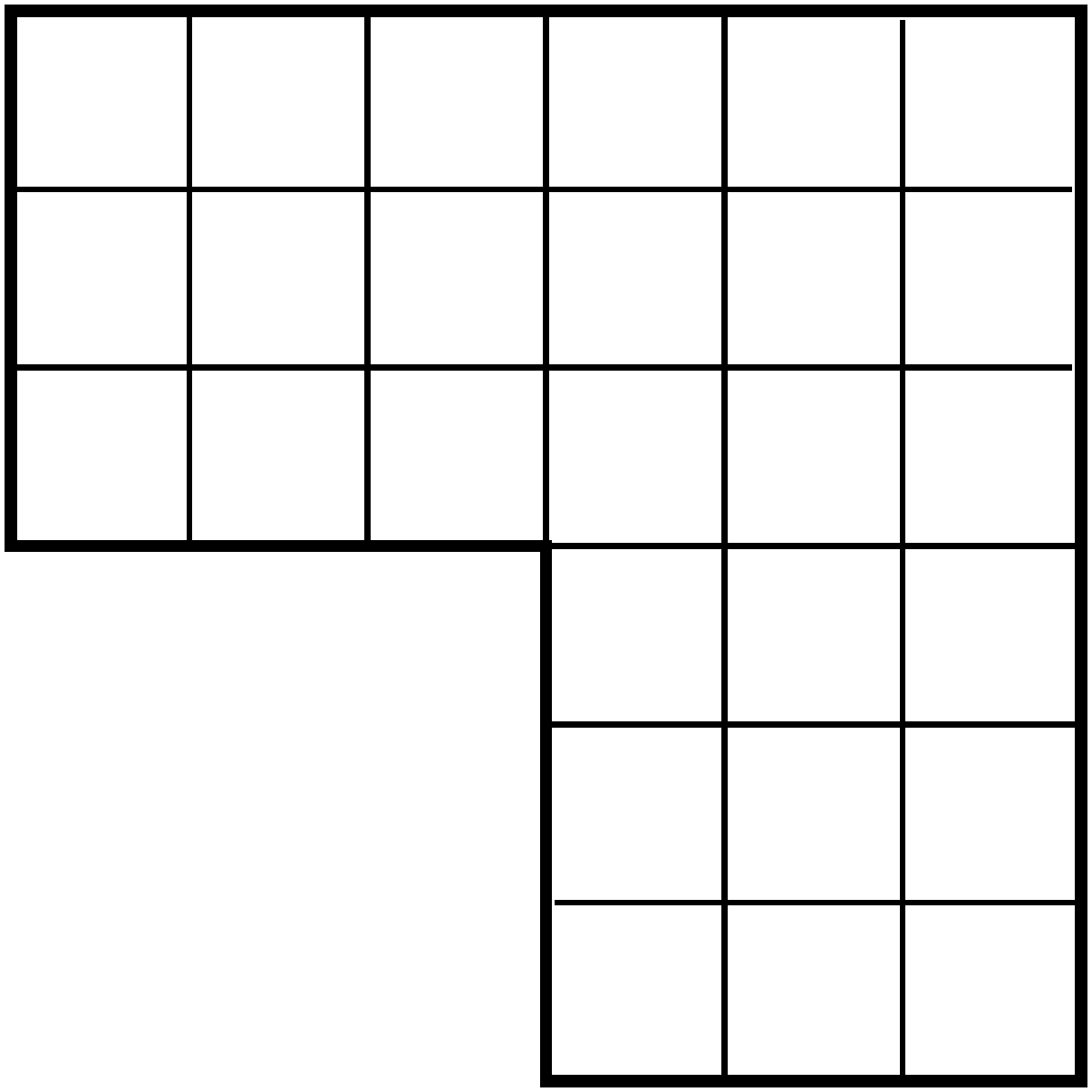}
\put(8,55){T}
\put(25,55){B}
\put(40,55){M}
\put(55,55){C}
\put(55,40){M}
\put(55,22){B}
\put(55,08){T}
\put(75,08){B}
\put(90,08){T}
\put(90,25){B}
\put(90,40){B}
\put(90,55){B}
\put(90,75){B}
\put(90,90){T}
\put(75,90){B}
\put(55,90){B}
\put(40,90){B}
\put(25,90){B}
\put(08,90){T}
\put(08,75){B}
\end{overpic}
\hfill 
\begin{overpic}[width=0.5\textwidth]{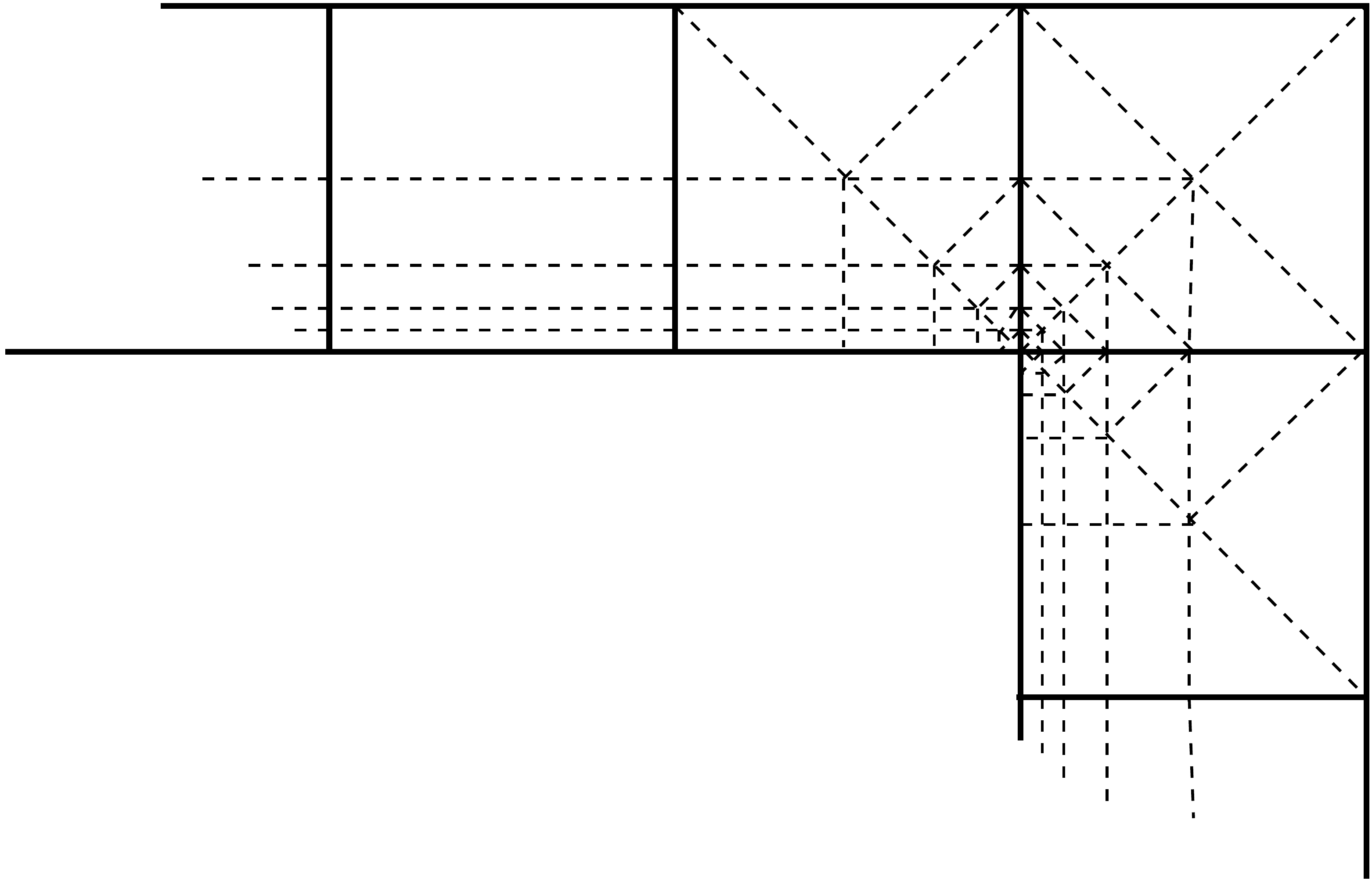}
\put(15,55){$\Omega$}
\put(35,50){B}
\put(55,50){M}
\put(85,50){C}
\put(95,25){M}
\put(40,30){$\partial\Omega$}
\put(65,30){$A_j$}
\end{overpic}

\caption{\label{fig:patch-examples} 
Left panel: example of an L-shaped domain decomposed into 27 patches 
($T$, $B$, $M$, $C$ indicate tensor, boundary layer, mixed, corner patches;
empty squares stand for trivial patches). 
Right panel: Zoom-in near the reentrant corner. 
}
\end{figure}
Our robust exponentially convergent
$hp$ approximation is based on
so-called \emph{geometric boundary layer meshes}, 
denoted by $\calT^{L,n}_{geo,\sigma}$. 
To facilitate our error analysis, the 
$\calT^{L,n}_{geo,\sigma}$ are generated as push-forwards 
of a small number of so-called \emph{reference patches}, which 
are partitions of $\widetilde S$, under the \emph{patch maps}. 
The images of $\widetilde S$ under the patch maps form a (coarse) 
\emph{macro triangulation} of $\Omega$ satisfying some minimal 
conditions, which are described in Section~\ref{sec:MacroTri}.
This concept was also used in the context of $hp$-FEM 
for singular perturbations in 
\cite[Sec.~{3.3.3}]{melenk02} and in 
\cite{melenk-xenophontos16,FstmnMM_hpBalNrm2017}.
\subsection{Macro triangulations}
\label{sec:MacroTri}
We assume given a \emph{fixed macro-triangulation}
${\mathcal T}^{\M} = \{K^{\M} \,|\, K^{\M} \in {\mathcal T}^{\M}\}$
of $\Omega$ consisting of curvilinear quadrilaterals $K^{\M}$
with bijective element maps $F_{K^{\M}}:\widetilde S \rightarrow K^{\M}$
that are analytic in $\overline{\widetilde S}$ 
and that in addition satisfy the usual compatibility conditions.
I.e., the partition ${\mathcal T}^{\M}$ does not have hanging nodes and, 
for any two distinct elements $K_1^{\M}, K_2^{\M} \in {\mathcal T}^{\M}$ that
share an edge $e$, their respective element maps induce compatible parametrizations
of $e$ (cf., e.g.,  \cite[Def.~{2.4.1}]{melenk02} for the precise conditions). 

Each element of the fixed macro-triangulation ${\mathcal T}^{\M}$
is further subdivided according to one of the 
refinement patterns in Definition~\ref{def:admissible-patterns}
(see also \cite[Sec.~{3.3.3}]{melenk02} or \cite{FstmnMM_hpBalNrm2017}). 
The actual triangulation
is then obtained by transplanting refinement patterns on the square reference 
patch $\widetilde{S}$ into the physical domain $\Omega$
by the element maps $F_{K^\M}$ of the macro-triangulation
resulting in the \emph{physical triangulation} ${\mathcal T}$.  
For any element $K\in {\mathcal T}$, 
the element maps $F_K :\widehat K \rightarrow K$ 
are then concatenations of affine maps 
$A_K:\widehat K \rightarrow \widetilde K$, which realize 
the mapping from $\widehat K \in \{\widehat S, \widehat T\}$ to the elements 
in the patch refinement pattern, and the analytic patch maps $F_{K^\M}$. 
That is, the element maps have the form $F_K = F_{K^\M} \circ A_K$
for an affine $A_K$. Throughout the article, we will denote
by $\widehat K \in \{\widehat S, \widehat T\}$ the reference element
corresponding to an element $K$ of a triangulation, and we will denote by $\Kt$
the elements of the triangulation of the reference patterns.  Points 
in the reference patch $\widetilde S$ are denoted 
$\widetilde{\bx}=(\xt,\yt) \in \widetilde S$; variables $(x,y)$ are employed 
to indicate points in $\Omega$, and $\widehat{\bx} = (\xh,\yh)$ are used 
for points of the reference square $\widehat S$ 
and reference triangle $\widehat{T}$. 
\subsection{Refinement patterns in the reference configuration (patch catalog ${\mathfrak P}$)}
\label{sec:refinement-patterns}
The admissible \emph{patch refinement patterns} are collected
in a catalog ${\mathfrak P}$ and are
depicted in Fig.~\ref{fig:patches}. 
They  
are based on geometric refinement towards a 
vertex and/or an edge; 
the parameter $L$ controls the number of 
layers of refinement towards an edge whereas the 
natural number
$n \ge L$ measures the 
number of geometric refinements towards vertices. 
\begin{definition}[catalog ${\mathfrak P}$ of refinement patterns]
\label{def:admissible-patterns}
Given $\sigma \in (0,1)$, $L$, $n \in {\mathbb N}_0$ with $n \ge L$ 
the catalog ${\mathfrak P}$ of admissible refinement patterns
consists of the following patches: 
\begin{enumerate}
\item
The \emph{trivial patch}: The reference square $\widetilde S$ is not further refined. 
The corresponding triangulation of $\widetilde S$ consists of the single element: 
$\widehat{\mathcal T} = \{\widetilde S\}$. 
\item
The \emph{geometric boundary layer patch $\widetilde{\calT}^{\BL,L}_{geo,\sigma}$}: 
$\widetilde S $ is refined anisotropically towards 
$\{\yt =0\}$ into $L$ elements as depicted
in Fig.~\ref{fig:patches} (top left). The mesh 
$\widetilde{\calT}^{\BL,L}_{geo,\sigma}$ 
is characterized by the nodes $(0,0)$, $(0,\sigma^i)$, $(1,\sigma^i)$,
$i=0,\ldots,L$, and the corresponding rectangular elements generated by these nodes.
\item
The \emph{geometric corner patch $\widetilde{\calT}^{\Co,n}_{geo,\sigma}$}: 
$\widetilde S $ is refined isotropically towards $(0,0)$ as depicted 
in Fig.~\ref{fig:patches} (top middle). Specifically,
the reference geometric corner patch mesh $\widetilde{\calT}^{\Co,n}_{geo,\sigma}$
in $\widetilde S$ with geometric refinement towards $(0,0)$ and $n$ layers
is given by triangles and based on the nodes
$(0,0)$, and $(0,\sigma^i)$, $(\sigma^i,0)$,
$(\sigma^i,\sigma^i)$, $i=0,1,\ldots,n$. 
\item
The \emph{tensor product patch $\widetilde{\calT}^{\Te,L,n}_{geo,\sigma}$}: 
$\widetilde S$ is triangulated in $\widetilde S_1:= (0,\sigma^L)^2$ and 
$\widetilde S_2:= \widetilde S \setminus \widetilde S_1$ separately as 
depicted in Fig.~\ref{fig:patches} (bottom left). 
The triangulation of $\widetilde S_1$ is a scaled version of 
$\widetilde{\calT}^{\Co,n-L}_{geo,\sigma}$ and based on the 
nodes $(0,\sigma^i)$, $(\sigma^i,0)$, $i=L,\ldots,n$. The triangulation of 
$\widetilde S_2$ is based on the nodes $(\sigma^i,\sigma^j)$, $i$, $j=0,\ldots,L$. 
\item
The \emph{mixed patches $\widetilde{\calT}^{\Mi,L,n}_{geo,\sigma}$}:
The triangulation consists of both anisotropic elements and isotropic elements
as depicted in Fig.~\ref{fig:patches} (bottom right) and is obtained 
by triangulating the regions 
$\widetilde S_1:= (0,\sigma^L)^2$, 
$\widetilde S_2:= \bigl( \widetilde S \setminus \widetilde S_1\bigr)  \cap \{\yt \leq \xt\}$, 
$\widetilde S_3:= \widetilde S \setminus (\widetilde S_1 \cup \widetilde S_2)$ 
separately. 
The set 
$\widetilde S_1$ is a scaled version of $\widetilde{\calT}^{\Co,n-L}_{geo,\sigma}$ 
based on the 
nodes $(0,\sigma^i)$, $(\sigma^i,0)$, $i=L,\ldots,n$. 
The triangulation of $\widetilde S_2$ is based on the nodes 
$(\sigma^i,0)$, $(\sigma^i,\sigma^{j})$, $0 \leq i \leq L$, $i \leq j \leq L$
and consists of rectangles and triangles, 
and only the triangles abut on the diagonal $\{\xt=\yt\}$. 
The triangulation of $\widetilde S_3$ consists of triangles only 
and is based on the nodes $(0,\sigma^i)$, $(\sigma^i,\sigma^i)$, $i=0,\ldots,L$. 
\end{enumerate}
\end{definition}
\begin{remark}
\label{remk:further-refinement-patterns} 
We kept the catalog ${\mathfrak P}$ of admissible patch refinement patterns in 
Definition~\ref{def:admissible-patterns} small in order
to reduce the number of cases to be discussed for the $hp$-FE error bounds.
A larger number of refinement patterns provides 
greater flexibility in the mesh generation. 
In particular, 
the reference patch meshes of Def.~\ref{def:admissible-patterns} 
do not contain general quadrilaterals but only rectangles; 
this restriction is not 
essential but simplifies the $hp$-FE error analysis.
Also certain types of anisotropic triangles (e.g., splitting anisotropic 
rectangles along the diagonal), which are altogether excluded in 
the present analysis, could be accommodated at the expense of additional
technicalities. 

The addition of the diagonal line in the reference corner,
tensor, and mixed patches is done 
to be able to apply the regularity theory of \cite{melenk02}.
It is likely not necessary in actual computations. 
We also mention that with additional constraints on the macro triangulation $\calT^\M$ 
the diagonal line could be dispensed with in certain situations 
as is illustrated in Section~\ref{sec:generate-bdy-layer-mesh}. 
\eremk
\end{remark}

\subsection{Geometric boundary layer mesh}
\label{S:GeoBL}
The following definition of the geometric boundary layer mesh $\Tg$ formalizes 
the patchwise construction of meshes on $\Omega$ based on transplanting 
meshes of the reference configurations to $\Omega$ via the patch maps $F_{K^\M}$. 
%
\begin{definition}[geometric boundary layer mesh $\Tg$ in $\Omega$]
\label{def:bdylayer-mesh}
Let ${\mathcal T}^{\M}$ 
be a fixed macro-triangulation consisting of quadrilaterals 
with analytic element maps that satisfy \cite[Def.~{2.4.1}]{melenk02}. 

Given $\sigma \in (0,1)$, $L$, $n \in {\mathbb N}_0$ with $n \ge L$, 
a regular mesh $\Tg$ in $\Omega$ 
is called a \emph{geometric boundary layer mesh} 
if the following conditions are satisfied:
\begin{enumerate}
\item $\Tg$ is obtained by refining each element $K^{\M} \in {\mathcal T}^{\M}$
according to one of the refinement patterns given in Definition~\ref{def:admissible-patterns} using 
the given parameters $\sigma$, $L$, and $n$.
\item The resulting mesh $\Tg$ is a regular triangulation of $\Omega$, i.e., 
it does not have hanging nodes. Since the element maps for the refinement patterns
are assumed to be affine, this requirement ensures that the resulting triangulation satisfies
\cite[Def.~{2.4.1}]{melenk02}.
\end{enumerate}
For each macro-patch $K^\M \in {\mathcal T}^{\M}$, 
exactly one of the following cases is possible: 
\begin{enumerate}
\setcounter{enumi}{2}
\item 
\label{item:def-geo-3}
$\overline{K^\M} \cap \partial\Omega = \emptyset$. 
Then the trivial patch is selected as the reference patch. 
\item 
\label{item:def-geo-4} 
$\overline{K^\M} \cap \partial \Omega$ is a single point. Then two cases
can occur: 
\begin{enumerate}[(a)]
\item 
$\overline{K^\M} \cap \partial \Omega =  \{\bA_j\}$ for a vertex $\bA_j$ of 
$\Omega$. Then 
the corresponding reference patch is the corner patch
$\widetilde{\calT}^{\Co,n}_{geo,\sigma}$ with $n$ layers of refinement towards
$\bO$. Additionally, $F_{K^\M}(\bO) = \bA_j$. 
\item 
$\overline{K^\M} \cap \partial\Omega = \{\bP\}$, where $\bP$ is not a 
vertex of $\Omega$. 
Then the refinement pattern is the
corner patch $\widetilde{\calT}^{\Co,L}_{geo,\sigma}$ with 
$L$ layers of geometric mesh refinement towards $\bO$. 
Additionally, 
it is assumed that $F_{K^\M}(\bO)  = \bP \in \partial\Omega$. 
\end{enumerate}
%
\item 
\label{item:def-geo-5}
$\overline{K^\M} \cap \partial \Omega = \overline{e}$ for an edge $e$ of 
$K^\M$ and neither  endpoint of $e$ is a vertex of $\Omega$. Then
the refinement pattern is the boundary layer patch $\widetilde{\calT}^{\BL,L}_{geo,\sigma}$ and additionally 
$F_{K^\M}(\{\yt = 0\}) \subset \partial\Omega$. 
\item 
\label{item:def-geo-6}
$\overline{K^\M} \cap \partial \Omega = \overline{e}$ for an edge $e$ of 
$K^\M$ and exactly one endpoint of $e$ is a vertex $\bA_j$ of $\Omega$. Then
the refinement pattern is the mixed layer patch
$\widetilde{\calT}^{\Mi,L,n}_{geo,\sigma}$ and additionally 
$F_{K^\M}(\{\yt = 0\}) \subset \partial\Omega$ as well as 
$F_{K^\M}(\bO) = \bA_j$. 
\item 
\label{item:def-geo-7}
Exactly two edges of a macro-element $K^\M$ are situated on $\partial\Omega$. 
Then the refinement pattern 
is the tensor patch $\widetilde{\calT}^{\Te,L,n}_{geo,\sigma}$. 
Additionally, 
it is assumed that 
$F_{K^\M}(\{\yt = 0\}) \subset \partial\Omega$, 
$F_{K^\M}(\{\xt = 0\}) \subset \partial\Omega$, and 
$F_{K^\M}(\bO) = \bA_j$ for a vertex $\bA_j$ of $\Omega$. 
\end{enumerate}
Finally, 
the following technical condition ensures the 
existence of certain meshlines: 
\begin{enumerate}
\setcounter{enumi}{7}
\item 
\label{item:def-geo-mesh-9}
For each vertex $\bA_j$ of $\Omega$, introduce a set of lines 
$$\ell = \bigcup_{K^\M \colon \bA_j \in \overline{K^\M}} 
\{\, F_{K^\M}(\{\yt = 0\}), F_{K^\M}(\{\xt = 0\}), F_{K^\M}(\{\xt = \yt\})\, \}.
$$
Let $\Gamma_j$, $\Gamma_{j+1}$ be the two boundary arcs of $\Omega$ that meet at $\bA_j$. 
Then there exists a line $e \in \ell$ such that 
the interior angles $\angle(e,\Gamma_j)$ and $\angle(e,\Gamma_{j+1})$ 
are both less than $\pi$. 
\end{enumerate}
\end{definition}
\begin{remark}
\label{rem:bl-meshes}
The last condition, requirement \ref{item:def-geo-mesh-9}.\ 
in Definition~\ref{def:bdylayer-mesh}, is merely a 
technical condition that results from our applying the regularity
theory for singular perturbations of \cite{melenk02}. 
Very likely, it could be dropped. 

The condition that $F_{K^\M}(\bO) \in \partial\Omega$ or that 
$F_{K^\M}(\{\yt = 0\}) \subset \partial\Omega$ are not conditions on the 
patch geometry but on the maps $F_{K^\M}$. They are not essential but introduced
for notational simplicity. They could be enforced by suitably concatenating
the maps $F_{\K^\M}$ with an orthogonal transformation.
\eremk
\end{remark}
\begin{remark}
The meshes $\Tg$ are refined towards both vertices and edges of $\Omega$. 
The parameter $L\in \bbN_0$ measures the number of layers of geometric refinement towards $\partial\Omega$
whereas the parameter $n\in \bbN$ characterizes the number of layers 
of geometric refinement towards the vertices. 
For $L = 0$ (or, more generally, $L$ fixed), 
the meshes $\calT^{0,n}_{geo,\sigma}$, $n=1,2,\ldots$, 
realize the ``geometric meshes'' introduced in \cite{babuska-guo86a, babuska-guo86b} 
(see also \cite[Sec.~{4.4.1}]{phpSchwab1998})
for the $hp$-FEM applied to elliptic boundary value problems with piecewise analytic data. 
\eremk
\end{remark}
\begin{example}
Fig.~\ref{fig:patch-examples} (left and middle) shows an example 
of an $L$-shaped domain with macro triangulation and 
suitable refinement patterns. 
\eremk
\end{example}
\subsection{Geometric boundary layer meshes in curvilinear polygons}
\label{sec:generate-bdy-layer-mesh}
\begin{figure}
\begin{overpic}[width=0.5\textwidth]{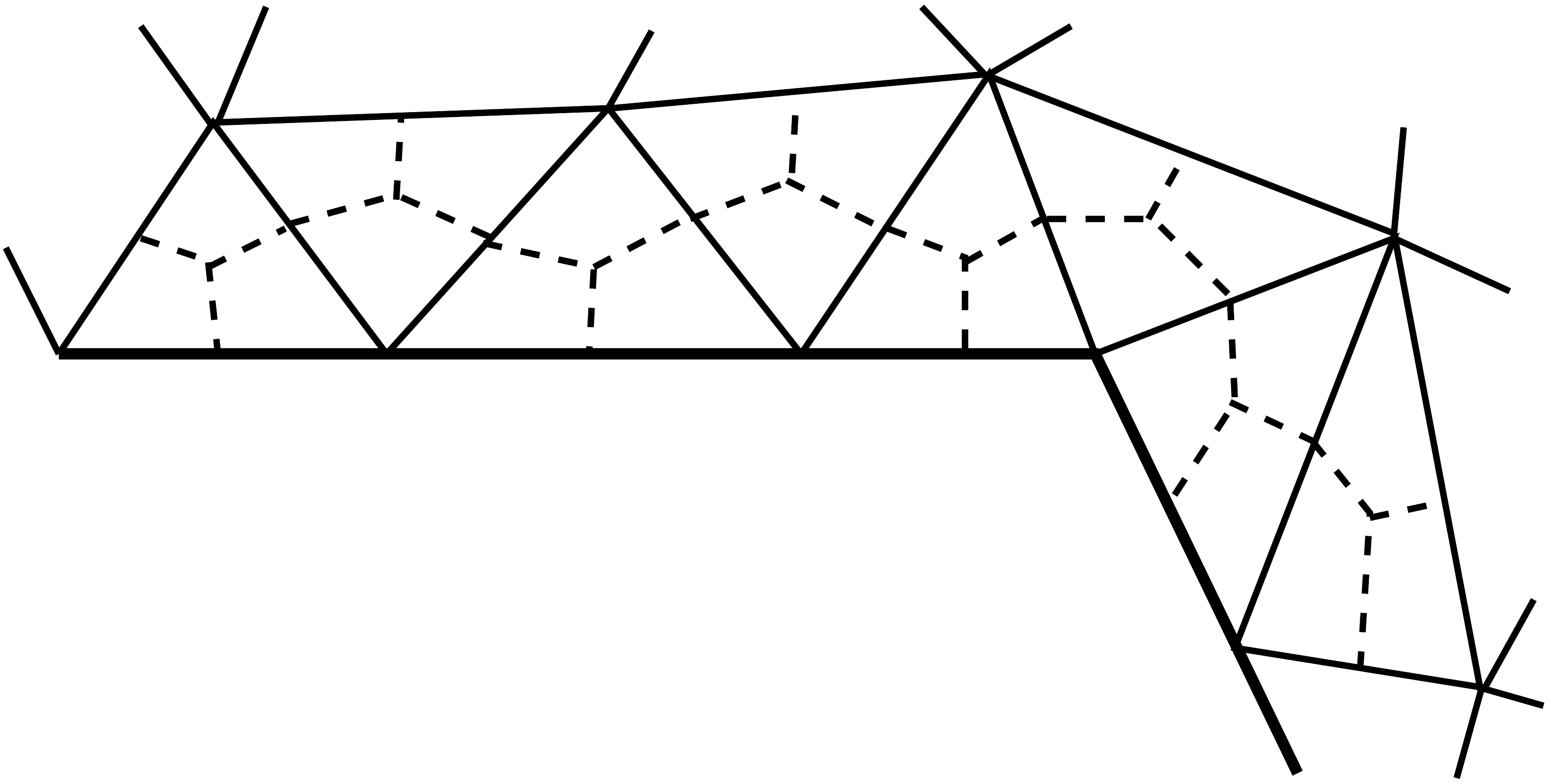}
\put(9,30){\small $B$}
\put(18,30){\small $B$}
\put(23,32){\small $C^L$}
\put(30,30){\small $B$}
\put(45,30){\small $B$}
\put(50,32){\small $C^L$}
\put(57,30){\small $B$}
\put(65,30){\small $M$}
\put(72,30){\small $C^n$}
\put(75,25){\small $M$}
\put(79,15){\small $B$}
\put(82,10){\small $C^L$}
\end{overpic}
\caption{\label{fig:bdylayermesh_from_triangulation} Generating a boundary layer mesh from a regular triangulation $\calT^0$: solid lines are the 
triangulation $\calT^0$, 
dashed lines connect edge midpoints with element barycenters to create a mesh consisting of quadrilaterals. 
$B$ stands for boundary layer, $M$ for mixed, 
$C^L$ for corner layer patches $\widetilde{\calT}^{\Co,L}_{geo,\sigma}$, 
$C^n$ for corner layer patches $\widetilde{\calT}^{\Co,n}_{geo,\sigma}$, 
empty quadrilaterals are trivial patches.}
\end{figure}
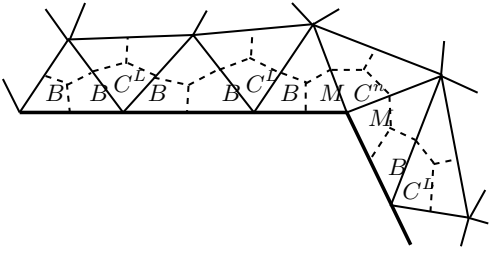
Geometric boundary layer meshes can be constructed in various ways. 
A first approach, which is in line with the illustration
in Fig.~\ref{fig:patch-examples}, is to create one layer of quadrilateral
elements that partition a tubular neighborhood $T_{\partial\Omega} $ 
of $\partial\Omega$. Each quadrilateral $K$ should fall into one of 
the following
3 categories: a) $\overline{K} \cap \partial\Omega$ is an edge of $K$; 
b) $\overline{K} \cap \partial\Omega$ consists of two contiguous edges 
and the shared vertex is a vertex of $\Omega$;  
c) $\overline{K} \cap \partial\Omega$ is a vertex of $\Omega$. 
In the second step, refinement patterns from Definition~\ref{def:bdylayer-mesh} are applied to each quadrilateral.
In  the final step,
$\Omega':= \Omega\setminus T_{\partial\Omega}$ is triangulated under the 
constraint that the boundary nodes of the triangulation of $\Omega'$ 
on $\partial\Omega'$ coincide with the nodes 
of the triangulation of $T_{\partial\Omega}$ that also lie on 
$\partial\Omega'$. This triangulation of $\Omega'$ could be chosen 
to consist of triangles (and/or quadrilaterals). All elements of that 
triangulation will be denoted ``trivial patches''; we mention without proof 
that the approximation result holds also if we include ``trivial'' triangles 
in the list of refinement patterns. 

Geometric boundary layer meshes can also be constructed for general (curvilinear)
polygons $\Omega$ starting from any regular initial triangulation 
${\mathcal T}^0$ of $\Omega$. 
This triangulation ${\mathcal T}^0$ is assumed 
to consist of (curvilinear) triangles with analytic
element maps and satisfying the ``usual'' conditions for triangulations as 
spelled out in \cite[Def.~{2.4.1}]{melenk02}.  
Then, the geometric boundary layer mesh is generated in 3 steps
(cf.\ Fig.~\ref{fig:bdylayermesh_from_triangulation}):
\begin{enumerate} 
\item 
(Ensure condition \ref{item:def-geo-mesh-9} of Def.~\ref{def:bdylayer-mesh}) 
For each vertex $\bA_j$ of $\Omega$ verify if an edge $e$ of ${\mathcal T}^0$  
splits the interior angle at $\bA_j$ into two angles each less than $\pi$. 
If not, then suitably split an appropriate triangle abuting on $\bA_j$ 
into two triangles
(so that the newly introduced edge will satisfy this condition) 
and remove the newly introduced hanging node by a mesh closure. 
The resulting triangulation has again analytic element maps and satisfies 
\cite[Def.~{2.4.1}]{melenk02}; it is again denoted ${\mathcal T}^0$. 
\item 
(Create a macro triangulation ${\mathcal T}^\M$ 
consisting of quadrilaterals only.) 
Split each triangle $K \in {\mathcal T}^0$ into 3 quadrilaterals as follows: split the 
reference triangle $\widehat T$ into 3 quadrilaterals 
$\widehat K_i$, $i=1,2,3$, characterized by the vertices of $\widehat T$, 
its barycenter, and by the $3$ midpoints of the edges of $\widehat T$. 
The element maps of the 3 quadrilaterals $F_K(\widehat K_i)$, $i=1,2,3$, 
are obtained by concatenating the bilinear bijections 
$F_{\widehat K_i}:\widetilde S \rightarrow \widehat K_i$ with $F_{K}$. 
The triangulation ${\mathcal T}^\M$ of $\Omega$ obtained in this way realizes 
a decomposition of $\Omega$ into (curvilinear) quadrilaterals, 
and the element maps satisfy \cite[Def.~{2.4.1}]{melenk02}. 
\item (Generate the geometric boundary layer mesh.)
The refinement pattern for each $K \in {\mathcal T}^\M$ is 
determined since $K$ falls into exactly one of the categories 
\ref{item:def-geo-3}---\ref{item:def-geo-7} of 
Definition~\ref{def:bdylayer-mesh} 
as can be seen by the following observations:
a) At most 2 edges of $K$ are on $\partial\Omega$ (since the two edges that 
meet in the barycenter of the parent triangle cannot be on $\partial\Omega$). 
b) If two edges of $K$ are situated on $\partial\Omega$, 
then they have to be subsets of the two edges 
of the parent triangle with common vertex $\bbV$; since ${\mathcal T}^0$ is 
a regular triangulation, the common vertex $\bbV$ has to be a vertex of $\Omega$. 
Additionally, if necessary, the assumptions on where the 
reference element vertex $\bO$ and/or 
the edges $\{\yt = 0\}$, $\{\xt=  0\}$ are mapped can be ensured by suitably 
adjusting the element map with the aid of an orthogonal transformation 
of $\widetilde S$. 
Finally, condition \ref{item:def-geo-mesh-9} of Def.~\ref{def:bdylayer-mesh} 
is satisfied by step 1.
\end{enumerate}
It remains to see that after selecting the refinement patterns the 
resulting triangulation satisfies \cite[Def.~{2.4.1}]{melenk02}. This
follows from the fact that the parameters $\sigma$, $L$, $n$ are the same 
for all macro elements and the structure of the refinement patterns: 
If an edge $e$ of the macro triangulation inherits a further refinement 
from a refinement pattern, then the edge either lies on 
$\partial\Omega$ (which is immaterial for the question of 
satisfying \cite[Def.~{2.4.1}]{melenk02}) or it is in $\Omega$ 
and exactly one of its endpoints $\bbV$ lies on $\partial\Omega$. 
This edge $e$ is shared by two macro elements. If $\bbV$ is a vertex 
of $\Omega$, then the refinement patterns are such that the induced
1D-mesh on $e$ is the same geometric mesh with $n$ layers for 
both macro elements. If $\bbV\in\partial\Omega$ is not a vertex of $\Omega$, 
then the induced 1D-mesh on $e$ is the same geometric mesh with $L$ layers 
for both macro elements. Hence, the resulting mesh satisfies 
\cite[Def.~{2.4.1}]{melenk02}. 

\subsection{Properties of the mesh patches}
\label{sec:PrpMshPtch}
We note that parts of the mixed patch, the tensor patch, 
and the corner patch are identical or at least structurally similar. 
For the analysis of the approximation properties of 
$hp$-FEM on geometric boundary layer meshes
it is therefore convenient to single out these meshes:  
%
\begin{definition}[half-patches, cf.\ \protect{Fig.~\ref{fig:half-figures}}]
\label{def:half-patches}
The \emph{ mixed half-patch}
$\widetilde{\calT}^{\Mi,\text{\rm half},L,n}_{geo,\sigma}$ 
and the \emph{ corner half-patch }
$\widetilde{\calT}^{\Co,\text{\rm half},n}_{geo,\sigma}$ 
on $\widetilde T = \{(\xt,\yt)\,|\, 0 < \xt < 1, 0 < \yt  <\xt\}$
are obtained by restricting $\widetilde{\calT}^{\Mi,L,n}_{geo,\sigma}$
and $\widetilde{\calT}^{\Co,n}_{geo,\sigma}$ to $\widetilde T$. The 
\emph{ flipped corner half-patch }
$\widetilde{\calT}^{\Co,\text{\rm half},\text{\rm flip},n}_{geo,\sigma}$ on 
$\Tf:=\{(\xt,\yt)\,|\, 0 < \xt < 1, \xt < \yt < 1\}$ is obtained
by reflecting $\widetilde{\calT}^{\Co,\text{\rm half},n}_{geo,\sigma}$ at the 
diagonal $\{(\xt,\xt)\,|\, \xt \in (0,1)\}$ of $\widetilde S$. 
\end{definition}
We will approximate functions on boundary layer meshes $\Tg$ with the aid
of an elementwise defined operator $\Pi_q$. 
To estimate the total
error in $L^2$-based norms, the elemental error contributions are 
summed up on each mesh patch separately. The following Lemma~\ref{lemma:properties-mesh-patches} 
provides tools to conveniently do that. In order to formulate 
Lemma~\ref{lemma:properties-mesh-patches}, we introduce some additional
notation, which represents the pull-back of the parts of the boundary
of the reference patch that is mapped to $\partial\Omega$ and is marked
by bold lines or dots in Figs.~\ref{fig:patches} and \ref{fig:half-figures}: 
\begin{subequations}
\label{eq:bdy-parts}
\begin{align}
\bdy^C:= \bdy^{\Co,\text{\rm half}} :=  \bdy^{\Co,\text{\rm half},\text{\rm flip}}& := \{\bO\}, \\
\bdy^{\BL}:= \bdy^M := \bdy^{\Mi,\text{\rm half}}& := \{\yt=0\}, \\
\bdy^{\Te}& := \{\yt=0\} \cup \{\xt=0\} \cup \{\bO\}. 
\end{align}
\end{subequations}
\begin{lemma}[properties of mesh patches]
\label{lemma:properties-mesh-patches}
The reference patches (cf.\ Def.~\ref{def:admissible-patterns}) 
and half patches (cf.\ Def.~\ref{def:half-patches}) have the following properties: 
\begin{enumerate}[(i)]
\item 
\label{item:lemma:properties-mesh-patches-i}
The triangular elements $\Kt$ of the reference patches are 
shape regular with shape regularity constant depending solely on $\sigma$. 
For the rectangular elements $\Kt$ of the reference patches, the 
element maps $A_{\Kt}: \widehat K \rightarrow \Kt$ are affine
with 
$$
A_{\Kt}^\prime =  \left(\begin{array}{cc} h_{\Kt,\xt} & \\ 
                                 & h_{\Kt,\yt}\end{array}\right), 
$$
where $h_{\Kt,\xt}$, $h_{\Kt,\yt} \leq 1$ 
are the side lengths (in $\yt$ and $\xt$-direction) of $\Kt$. 
We denote 
\begin{equation}
h_{\Kt,min}:= \min\{h_{\Kt,\xt},h_{\Kt,\yt}\}, 
\qquad 
h_{\Kt,max}:= \max\{h_{\Kt,\xt},h_{\Kt,\yt}\}. 
\end{equation}
\item 
\label{item:lemma:properties-mesh-patches-ii}
There is $c_\di > 0$ depending only on $\sigma$ such that 
for all triangular elements $\Kt$ 
of a reference patch $\widetilde{\calT}$ or a half-patch  
$\widetilde{\calT}$ the following dichotomy holds: 
$$
\text{either $\overline{\Kt} \cap \bdy \ne \emptyset$}
\quad \text{ or } \quad 
\operatorname{dist}(\Kt,\bdy) \ge c_\di h_{\Kt}, 
$$
where $h_{\Kt} = \operatorname{diam}(\Kt)$ 
and 
$\bdy\in 
\{
\bdy^{\Co}, 
\bdy^{\Co,\text{\rm half}},
\bdy^{\Co,\text{\rm half},\text{\rm flip}}, 
\bdy^{\Te}, 
\bdy^{\Mi}, 
\bdy^{\Mi,\text{\rm half}}, 
\bdy^{\BL}\}$ 
for $\widetilde{\calT} \in \{
\widetilde{\calT}^{\Co,n}_{geo,\sigma}, 
\widetilde{\calT}^{\Co,\text{\rm half}, n}_{geo,\sigma}, 
\widetilde{\calT}^{\Co,\text{\rm half}, \text{\rm flip}, n}_{geo,\sigma}, 
\widetilde{\calT}^{\Te,n}_{geo,\sigma}, 
\widetilde{\calT}^{\Mi,L,n}_{geo,\sigma}, 
\widetilde{\calT}^{\Mi,\text{\rm half},L,n}_{geo,\sigma}, 
\widetilde{\calT}^{\BL,L}_{geo,\sigma} 
\},$
respectively. 
\item 
\label{item:lemma:properties-mesh-patches-iii}
There is $c_\di >0$ depending only on $\sigma$ such that 
for all rectangular elements $\Kt$ 
of a reference patch $\widetilde{\calT}$ 
or half-patch $\widetilde{\calT}$,
the following dichotomy holds: 
$$
\text{Either $\overline{\Kt} \cap \bdy=\emptyset$}
\quad 
\text{ or } \quad \operatorname{dist}(\Kt,\bdy) \ge c_\di h_{\Kt,min}, 
$$
where $\bdy\in \{
\bdy^{\Te}, 
\bdy^{\Mi}, 
\bdy^{\Mi,\text{\rm half}}, 
\bdy^{\BL}\}$ 
for $\widetilde{\calT} \in \{
\widetilde{\calT}^{\Te,n}_{geo,\sigma}, 
\widetilde{\calT}^{\Mi,L,n}_{geo,\sigma}, 
\widetilde{\calT}^{\Mi,\text{\rm half},L,n}_{geo,\sigma}, 
\widetilde{\calT}^{\BL,L}_{geo,\sigma}
\},$
respectively. 
\item 
\label{item:lemma:properties-mesh-patches-iv}
There is $c_\di > 0$ depending only on $\sigma$ such that
for all rectangular elements $\Kt$ of a mixed patch, a mixed half-patch, 
or a tensor patch there holds 
$\operatorname{dist}(\Kt,\bO) \ge c_\di h_{\Kt,max}$. 
\item
\label{item:lemma:properties-mesh-patches-iva}
There is $C > 0$ depending only on $\sigma$ such that for
all elements $\Kt$ of a reference patch or half-patch there holds 
$\operatorname{dist}(\Kt,\bO) \leq C \operatorname{diam} \Kt$. 
\item 
\label{item:lemma:properties-mesh-patches-v} 
Let $\delta > 0$ and consider a reference patch or half-patch. 
Let ${\Tref}^\triangle$ be the collection of triangles 
of that reference patch or half-patch 
that do not abut on the vertex $\bO$. 
Then, there exists a constant $C>0$ 
depending solely on $\delta$ and $\sigma$ such that
$$
\sum_{\Kt \in {\Tref}^\triangle} h_{\Kt}^{\delta} \leq C. 
$$
\item 
\label{item:lemma:properties-mesh-patches-vi}
Let $\delta > 0$ and consider a reference mixed patch, 
tensor patch, mixed half-patch or corner half-patch. 
Let ${\Tref}^\square$ be the collection of rectangles of that reference patch. 
Then there exists a constant $C>0$ 
depending solely on $\delta$ and the parameter $\sigma$
such that
$$
\sum_{\Kt \in {\Tref}^\square} \frac{h_{\Kt,min}}{h_{\Kt,max}} h_{\Kt,max}^{\delta} \leq C. 
$$
\item 
\label{item:lemma:properties-mesh-patches-vii}
Let $\delta \in (0,1]$, $\alpha > 0$, 
and consider a reference patch or half-patch. 
Let ${\Tref}^\triangle$ be the collection of triangles 
of that reference patch or half-patch that do not abut on the vertex 
$\bO$. 
Then, there holds, with a $C>0$ 
depending solely on $\delta$, $\alpha$, and $\sigma$,
$$
\forall \varepsilon \in (0,1]\colon \quad
\sum_{\Kt \in {\Tref}^\triangle} 
(h_{\Kt}/\varepsilon)^{\delta} e^{-\alpha h_{\Kt}/\varepsilon} \leq C. 
$$
\item 
\label{item:lemma:properties-mesh-patches-viii}
Let $\delta \in (0,1]$, $\alpha > 0$, 
and consider a reference mixed patch, mixed half-patch, 
or a reference tensor patch. 
Let ${\Tref}^\square$ be the collection of rectangles of that reference patch. 
Then there exists a constant $C>0$ depending solely on $\delta$, $\alpha$, and $\sigma$ such that
$$
\forall \varepsilon \in (0,1]\colon \quad 
\sum_{\Kt \in {\Tref}^\square} \frac{h_{\Kt,min}}{h_{\Kt,max}} 
   (h_{\Kt,max}/\varepsilon)^{\delta} e^{-\alpha h_{\Kt,max}/\varepsilon}
\leq C. 
$$
\end{enumerate}
\end{lemma}
\begin{proof}
Items~(\ref{item:lemma:properties-mesh-patches-i})--(\ref{item:lemma:properties-mesh-patches-iva}) follow by construction. 

Since items (\ref{item:lemma:properties-mesh-patches-v}), 
(\ref{item:lemma:properties-mesh-patches-vi}) are shown by similar arguments, 
we only prove the case of (\ref{item:lemma:properties-mesh-patches-vi}) for the
specific case of the mixed patch as shown in Fig.~\ref{fig:patches}, bottom right panel.
Inspection of that panel shows that for each
$\Kt \in {\Tref}^\square$ we have 
$h_{\Kt,min} = h_{\Kt,\yt}$ 
and $h_{\Kt,max} = h_{\Kt,\xt}$. 
Additionally, the elements can be enumerated 
as $K_{i,j}$, $i=1,\ldots,L$, $j=1,\ldots,i$ 
with 
$h_{K_{i,j},x} \sim \sigma^{L-i}$,  $h_{K_{i,j},y} \sim \sigma^{L-j}$.  
Hence, 
\begin{align}
\label{eq:lemma:properties-mesh-patches-90}
\sum_{\Kt \in {\Tref}^\square} 
\frac{h_{\Kt,min}}{h_{\Kt,max}} h_{\Kt,max}^\delta
& \lesssim 
\sum_{i=1}^L \sum_{j=1}^i \frac{\sigma^{L-j}}{\sigma^{L-i}} \sigma^{(L-i)\delta}
\lesssim 
\sum_{i=1}^L \sigma^{\delta(L-i)}  \lesssim 1. 
\end{align}
The proof of items 
(\ref{item:lemma:properties-mesh-patches-vii}), 
(\ref{item:lemma:properties-mesh-patches-viii}) is also done in similar ways. 
Therefore, 
we will only show 
(\ref{item:lemma:properties-mesh-patches-viii}). The key observation is 
that by comparing sums with integrals, 
there is a constant $C > 0$ depending solely on $\delta$, $\alpha$, and $\sigma$ 
such that  
\begin{equation}
\label{eq:lemma:properties-mesh-patches-100}
\forall \eps\in (0,1]\colon \quad
\sum_{i=0}^\infty (\sigma^i/\varepsilon)^\delta 
e^{-\alpha \sigma^i/\varepsilon} \leq C. 
\end{equation}
The proof of (\ref{item:lemma:properties-mesh-patches-viii}) now follows by a reasoning
similar to that in (\ref{eq:lemma:properties-mesh-patches-90}). 
\end{proof}

\section{Approximation on the reference elements and on the reference configurations}
\label{S:AppRefElt}
In Sec.~\ref{sec:approx-reference-element} we construct polynomial approximation operators 
on the reference square and triangle that coincide with the Gauss-Lobatto interpolant
on the edges, which affords convenient $H^1$-conforming approximations. 
Sec.~\ref{sec:approx-reference-patches} studies the approximation properties 
of spaces of piecewise polynomials on the reference patches. It is shown 
that functions of boundary layer or corner layer type can be approximated at exponential
rates, robustly in the parameter $\varepsilon$ that characterizes the strength of 
the layer. 
\subsection{Polynomial approximation operators on the reference element}
\label{sec:approx-reference-element}
We introduce polynomial approximation operators on the reference triangle 
$\widehat T$ in Lemma~\ref{lemma:hat_Pi_infty} and the reference square 
$\widehat S$ in Lemma~\ref{lemma:hat_Pi_1_infty}. 
Before actually doing so, we highlight a technical detail: 
the triangular elements (on the reference patches) 
are shape-regular so that isotropic scaling arguments can be brought to bear; 
only the rectangles (of the reference patches) may be anisotropic, for which 
tensor product polynomial approximation operators (specifically, 
the Gauss-Lobatto interpolation operator) are used for their favorable 
anisotropic scaling properties. 
 
\begin{lemma}[element-by-element approximation on triangles]
\label{lemma:hat_Pi_infty}
Let $\widehat T$ be the reference triangle. 
Then for every $q \in {\mathbb N}$, 
there exists a linear operator 
$\widehat \Pi^\triangle_q:C(\overline{\widehat T})\rightarrow {\bbP}_q$ 
with the following properties: 
\begin{enumerate}[(i)]
\item 
\label{item:lemma:hat_Pi_infty--1}
For each edge $e$ of $\widehat T$, 
$(\PiT u)|_e$ coincides with the 
Gauss-Lobatto interpolant $i_q (u|_e)$ of degree $q$ on edge $e$.
\item (projection property)
\label{item:lemma:hat_Pi_infty-0}
$\PiT v = v$ for all $v\in {\bbP}_q$. 
\item 
\label{item:lemma:hat_Pi_infty-i}
(stability) 
There exists a constant $C>0$ such that for every $q\in \mathbb{N}$ 
there holds
\begin{align*}
 \forall u \in W^{1,\infty}(\widehat T) \colon  
\qquad 
\|u - \PiT u\|_{W^{1,\infty}(\widehat T)} 
&\leq C q^4 \|\nabla u\|_{L^\infty(\widehat  T)}, 
\\
\forall u \in C(\overline{\widehat T}) 
\colon \qquad 
\|u - \PiT u\|_{L^{\infty}(\widehat T)} 
&\leq C q^2 \|u\|_{L^\infty(\widehat  T)} . 
\end{align*}
\item 
\label{item:lemma:hat_Pi_infty-ii} 
Let $\bA$ be one of the vertices of $\widehat  T$ and $\beta \in [0,1)$. 
Then there is $C > 0$ (depending only on $\beta$) such that, provided the right-hand side is finite,  
$$
\|u - \PiT u\|_{L^\infty(\widehat T)} 
+ 
\|u - \PiT u\|_{H^1(\widehat T)} 
\leq 
C q^4 \|\operatorname{dist}(\cdot,\bA)^\beta \nabla^2 u\|_{L^2(\widehat T)}. 
$$
\item 
\label{item:lemma:hat_Pi_infty-iii}
Let $u \in C^\infty(\widehat T)$ satisfy, for some $C_u$, $\gamma > 0$ 
and for some $h$, $\varepsilon \in (0,1]$,  
$$ 
\forall n \in {\mathbb N}_0 \colon\;\;
\|\nabla^n u\|_{L^\infty(\widehat T)} 
\leq 
C_u \gamma^n h^n \max\{n+1,\varepsilon^{-1}\}^n . 
$$
Then there are $\delta$, $C$, $\eta$, $b >0$ depending solely on $\gamma$ 
such that, under the provision that the 
\emph{scale resolution condition}
\begin{equation}
\label{eq:hat-approx-constraint-triangle}
\frac{h}{q \varepsilon} \leq \delta
\end{equation}
is satisfied, there holds  (with the constant hidden in $\lesssim$ independent of $u$, $h$, $q$ and $\varepsilon$)
\begin{align*}
\|u - \PiT u\|_{W^{1,\infty}(\widehat T)} 
& \lesssim C_u \left( \left(\frac{h}{h+\eta}\right)^{q+1} 
  + \left(\frac{h}{q \varepsilon \eta}\right)^{q+1}\right) 
\lesssim C_u e^{-b q} \min\{1,h/\varepsilon\}.
\end{align*}
\end{enumerate}
\end{lemma}
\begin{proof}
The operator $\PiT$ is taken as the one defined in 
\cite[Thm.~{3.2.20}]{melenk02}, where items 
(\ref{item:lemma:hat_Pi_infty--1})--(\ref{item:lemma:hat_Pi_infty-i}) are 
shown (the $W^{1,\infty}$-estimate follows with an 
additional polynomial inverse estimate). 
Item (\ref{item:lemma:hat_Pi_infty-ii}) is taken from 
\cite[Prop.~{3.2.21}]{melenk02}.  
For Item (\ref{item:lemma:hat_Pi_infty-iii}), we note that 
the projection property of (\ref{item:lemma:hat_Pi_infty-0}) and 
the stability assertions (\ref{item:lemma:hat_Pi_infty-i}) 
reduce the error estimate to a best approximation problem, which can
be taken from \cite[Lemma~{C.2}]{melenk-sauter10}. 
\end{proof}
\begin{lemma}[approximation properties of the Gauss-Lobatto interpolant]
\label{lemma:hat_Pi_1_infty}
Let $\widehat S$ be the reference square. 
For each $q \in {\mathbb N}$
the tensor-product Gauss-Lobatto interpolation operator 
$\PiGL: C(\overline{\widehat S}) \rightarrow {\bbQ}_q$ 
satisfies the following: 
\begin{enumerate}[(i)]
\item (projection property)
\label{item:lemma:hat_Pi_1_infty--1}
$\PiGL v = v$ for all $v \in {\bbQ}_q$. 
\item 
\label{item:lemma:hat_Pi_1_infty-0}
For each edge $e$, the restriction $(\PiGL u)|_e$ coincides with 
the univariate Gauss-Lobatto interpolant $i_q (u|_e)$
on $e$.
\item 
\label{item:lemma:hat_Pi_1_infty-i}
(stability) 
\begin{align*}
\forall u \in C(\overline{\widehat S}) \colon &&
\|u - \PiGL  u\|_{L^{\infty}(\widehat S)} 
& \leq C q \|u\|_{L^\infty(\widehat S)} 
,\\ 
\forall u \in C^{1}(\overline{\widehat S})\colon &&
\|\partial_\xh (u - \PiGL u)\|_{L^{\infty}(\widehat S)} 
&\leq C q^4 \|\partial_x u\|_{L^\infty(\widehat S)} 
, \\
\forall u \in C^{1}(\overline{\widehat S}) \colon &&
\|\partial_\yh (u - \PiGL u)\|_{L^{\infty}(\widehat S)}
& \leq C q^4 \|\partial_y u\|_{L^\infty(\widehat S)} . 
\end{align*}
\item 
\label{item:lemma:hat_Pi_1_infty-ii}
Let $u \in C^\infty(\widehat S)$ satisfy for some $C_u$, $\gamma > 0$, 
$\varepsilon_x$, $\varepsilon_y$, $h_x$, $h_y \in (0,1]$ 
and all
$(n,m) \in {\mathbb N}_0^2 $
\begin{equation}
\label{eq:lemma:hat_Pi_1_infty-reg}
\|\partial_\xh^m \partial_\yh^n u\|_{L^\infty(\widehat S)} \leq 
C_u \gamma^{n+m} h_x^m h_y^n \max\{n+1,\varepsilon_y^{-1}\}^n 
\max\{m+1,\varepsilon_x^{-1}\}^m . 
\end{equation}
Then there are constants $\delta$, $C$, $\eta$, $b > 0$ 
depending solely on $\gamma$ such that under the 
scale-resolution condition
\begin{equation}
\label{eq:hat-approx-constraint-rectangle}
\frac{h_x}{q \varepsilon_x } + \frac{h_y}{q \varepsilon_y } 
\leq \delta
\end{equation}
there holds 
\begin{align*}
&  \|\partial_\xh( u - \PiGL u)\|_{L^{\infty}(\widehat S)} \\
& \leq C C_u
\frac{h_x}{\varepsilon_x}
\left[
\varepsilon_x 
\left( \frac{h_x}{h_x+\eta} \right)^{q} +
\left( \frac{h_x}{\varepsilon_x q \eta } \right)^{q} +
\left( \frac{h_y}{h_y+\eta} \right)^{q+1} +
\left( \frac{h_y}{\varepsilon_y q \eta } \right)^{q+1}
\right], \\
& \|\partial_\yh( u - \PiGL u)\|_{L^{\infty}(\widehat S)} 
\\
& \leq C C_u
\frac{h_y}{\varepsilon_y}
\left[
\left( \frac{h_x}{h_x+\eta} \right)^{q+1} +
\left( \frac{h_x}{\varepsilon_x q \eta } \right)^{q+1} +
\varepsilon_y \left( \frac{h_y}{h_y+\eta} \right)^{q} +
\left( \frac{h_y}{\varepsilon_y q \eta } \right)^{q}
\right], \\
& 
\|u - \PiGL u\|_{L^{\infty}(\widehat S)}  \\
& \leq C C_u
\left[
\left( \frac{h_x}{h_x+\eta} \right)^{q+1} +
\left( \frac{h_x}{\varepsilon_x q \eta } \right)^{q+1} +
\left( \frac{h_y}{h_y+\eta} \right)^{q+1} +
\left( \frac{h_y}{\varepsilon_y q \eta } \right)^{q+1}
\right].
\end{align*}
\end{enumerate}
\end{lemma}
\begin{proof}
Items (\ref{item:lemma:hat_Pi_1_infty--1}), 
(\ref{item:lemma:hat_Pi_1_infty-0}) are well-known. 
We let $\Lambda_q$ denote the Lebesgue constant of the 
univariate Gauss-Lobatto interpolation operator of polynomial degree $q\in {\mathbb N}$
(cf. Lemma~\ref{lemma:1d-GL}).
The $L^\infty$-stability 
in (\ref{item:lemma:hat_Pi_1_infty-i}) 
follows from tensor product arguments, viz. 
$\|\widehat \Pi^{\Box}_q u\|_{L^\infty(\widehat S)} 
    \leq \Lambda_q^2 \|u\|_{L^\infty(\widehat S)}$ 
and the (generous) bound $\Lambda_q^2 \leq C q$ for $q\geq 1$.
For the remaining estimates, we introduce the tensor-product 
Gauss-Lobatto interpolation operator 
$\PiGL = i_q \otimes i_q = i^\xh_q \otimes i^\yh_q$, where we use 
the superscipts $\xh$ and $\yh$ to emphasize 
the
variable with respect to which the univariate Gauss-Lobatto 
interpolant acts. 
From 
\begin{align*}
u - i^\yh_q \otimes i^\xh_q u &= 
u - (\operatorname{I} \otimes i^\yh_q) u 
+ \operatorname{I} \otimes i^\yh_q (u - i^\xh_q \otimes \operatorname{I} u) 
\end{align*}
we get in view of the univariate 
stability bound Lemma~\ref{lemma:1d-GL} 
\begin{align} 
\nonumber 
\|\partial_\xh (u - i^\xh_q \otimes i^\yh_p u)\|_{L^\infty(\widehat S)} & \lesssim
\Lambda_q \sup_{\xh \in (0,1)} \inf_{v \in {\bbP}_q} 
\|\partial_\xh u(x,\cdot) - v\|_{L^\infty(0,1)} \\
\label{eq:2d-GL-der}
& \quad \mbox{} +   
q^2 \Lambda_q^2 \sup_{\yh \in (0,1)} \inf_{v \in {\bbP}_q} 
\|\partial_\xh (u(\cdot,\yh) - v)\|_{L^\infty(0,1)}; 
\end{align}
an analogous estimate holds for $\partial_y (u - i^\xh_q \otimes i^\yh_q u)$. 
The estimate \eqref{eq:2d-GL-der} gives the stability estimates in 
$W^{1,\infty}$ of (\ref{item:lemma:hat_Pi_1_infty-i}) by selecting
$v = 0$ in the infima. 
The estimate \eqref{eq:2d-GL-der} reduces the question of approximation 
on $\widehat{S}$ to questions of univariate polynomial approximation.  
The pertinent approximation results to prove 
item~(\ref{item:lemma:hat_Pi_1_infty-ii}) are given in 
Lemma~\ref{lemma:1D-poly-approx}.
\end{proof}
\subsection{Approximation on the reference patches}
\label{sec:approx-reference-patches}
In this section, we study the approximation of functions on the reference patches
(or the half-patches) described in 
Defs.~\ref{def:admissible-patterns}, \ref{def:half-patches}. 
The non-trivial reference patches
consist of meshes that are refined towards $\bO$, which can resolve algebraic singularities
at $\bO$, and  meshes that are anisotropically refined towards the edge $\{\yt = 0\}$, which 
can resolve algebraic singularities at $\{\yt=0\}$ or boundary layers. We show exponential
approximability of functions that have algebraic singularities at $\bO$ or boundary
layers at $\{\yt=0\}$. 

Throughout this section, we will use the notation 
\begin{equation}
\label{eq:def-rt}
\rt(\cdot):= \operatorname{dist}(\bO,\cdot) .
\end{equation}
In this section, 
we present piecewise polynomial approximations on reference patches using the 
following elementwise defined interpolation operator: 
\begin{equation}
\label{eq:tildePiq} 
(\widetilde \Pi_q)|_{\Kt} u:= 
\begin{cases} 
\widehat\Pi_q^\triangle (u \circ A_\Kt) & \mbox{ if $\Kt$ is a triangle $\triangle$} \\
\widehat\Pi_q^{\Box} (u \circ A_\Kt) & \mbox{ if $\Kt$ is a rectangle $\Box$}, \\
\end{cases} 
\end{equation}
where $A_\Kt: \widehat K \rightarrow \Kt = A_\Kt(\widehat K)\subset \widetilde S$ is the
affine bijection between the reference element and the 
corresponding element on the reference patch.
The edge-traces of the interpolators $\widehat \Pi_q^\triangle$ and $\widehat \Pi_q^{\Box}$ 
coincide with the univariate Gauss-Lobatto interpolation operator on the edges of $\Kh$.
Hence, $H^1$-conformity of the elementwise defined operator $\widetilde \Pi_q$ is ensured. 
We will frequently use the stability estimates
\begin{equation}
\label{eq:tildePiq-stable}
\|\widetilde \Pi_q u\|_{L^\infty(\Kt)} 
\leq C q^2 \|u\|_{L^\infty(\Kt)}, 
\qquad 
\|\nabla \widetilde \Pi_q u\|_{L^\infty(\Kt)}
\leq C q^4 \|\nabla u\|_{L^\infty(\Kt)}; 
\end{equation}
these estimates are easily seen to hold
for triangles with the isotropic scaling property 
and Lemma~\ref{lemma:hat_Pi_infty}, (\ref{item:lemma:hat_Pi_infty-i}). 
The anisotropic nature of the rectangles 
is accounted for by separately scaling the bounds for the partial derivatives
    in Lemma~\ref{lemma:hat_Pi_1_infty}, (\ref{item:lemma:hat_Pi_1_infty-i}). 
%
\subsubsection{$hp$-FE approximation of corner singularity functions}
\label{sec:singularity-fct-approx-geo-mesh}
%
%
\begin{lemma}[approximation of corner singularity functions]
\label{lemma:asymptotic-case}
\begin{enumerate}[(i)]
\item
\label{item:lemma:asymptotic-case-i}
Let $\widetilde{\calT} \in \{
\widetilde{\calT}^{\Mi,\text{\rm half},L,n}_{geo,\sigma}, 
\widetilde{\calT}^{\Co,\text{\rm half},n}_{geo,\sigma}, 
\widetilde{\calT}^{\Co,\text{\rm half},\text{\rm flip}, n}_{geo,\sigma}, 
\widetilde{\calT}^{\Co,n}_{geo,\sigma}, 
\widetilde{\calT}^{\Te,n}_{geo,\sigma}\}. 
$
Let $\calO$ be the region covered by the elements of $\widetilde{\calT}$, i.e., 
let $\calO = \widetilde S$ if $\widetilde{\calT}$ is a reference patch, 
$\calO = \widetilde T$ if 
$\widetilde{\calT} \in \{ \widetilde{\calT}^{\Co,\text{\rm half},n}_{geo,\sigma}, \widetilde{\calT}^{\Mi,\text{\rm half},L,n}_{geo,\sigma}\}$ 
is a reference half-patch, and 
$\calO = \Tf$ if 
$\widetilde{\calT} = \widetilde{\calT}^{\Co,\text{\rm half},\text{\rm flip},n}_{geo,\sigma}$. 
Let $\widetilde u$ be analytic on $\overline{\calO}$ 
and assume there exist constants
$\varepsilon \in (0,1]$, 
$\beta \in [0,1)$, $\gamma$, $C_u > 0$
such that for all $p \in {\mathbb N}_0$ and all $\widetilde \bx \in \widetilde S$
\begin{equation}
\label{eq:lemma:asymptotic-case-10}
|\nabla^p (\ut(\widetilde \bx)- \ut(\bO))|
\leq 
C_u \varepsilon^{-1} \gamma^{p} (\rt(\boxt)/\varepsilon)^{1-\beta} \rt(\boxt)^{-p} \max\{p+1,\rt(\boxt)/\varepsilon\}^{p+1} 
\;.
\end{equation}
Then, 
there are constants $C$, $b$, $\kappa > 0$ 
depending only on $\gamma$, $\sigma$, and $\beta$ 
(in particular, independent of $\varepsilon$, $n$, $L$) 
such that under the scale resolution condition
\begin{equation}
\label{eq:lemma:asymptotic-case-20}
q \varepsilon \ge \kappa   
\end{equation}
there holds 
\begin{equation}
\|\ut - \widetilde \Pi_q \ut\|_{L^\infty(\calO)}
+ \|\nabla (\ut - \widetilde \Pi_q \ut)\|_{L^2(\calO)}
\leq 
C C_u \left( q^{9} \sigma^{n(1-\beta)} + e^{-b q} \right).
\end{equation}
\item
\label{item:lemma:asymptotic-case-ii}
Let $\widetilde{\calT} \in \{\widetilde{\calT}^{\BL,L},\widetilde{S}\}$. 
Let $\widetilde u$ be analytic on $\widetilde{S}$ and 
assume that there are constants $C_u$, $\gamma>0$
such that for all $p\in {\mathbb N}_0$
\begin{equation}
\|\nabla^p \ut\|_{L^\infty(\widetilde S)}
\leq C_u \gamma^{p} \max\{p+1,\varepsilon^{-1}\}^{p+1}. 
\end{equation}
Then there are constants $C$, $b$, $\kappa > 0$ 
depending only on $\gamma$ and $\sigma$ (in particular, they are independent
of $\varepsilon$ and $L$) such that under the constraint 
(\ref{eq:lemma:asymptotic-case-20}) there holds 
\begin{equation}
\|\ut - \widetilde \Pi_q \ut\|_{W^{1,\infty}(\widetilde{S})}
\leq C C_u e^{-b q}. 
\end{equation}
\end{enumerate}
\end{lemma}
\begin{proof}
\emph{Proof of (\ref{item:lemma:asymptotic-case-i}):} 
\emph{Step 1: Elements abutting on $\bO$:}
Only triangles $\triangle$ may abut on $\bO$. 
Let $\Kt$ be such a triangle. 
From (\ref{eq:lemma:asymptotic-case-10}) and estimating (generously)
$\max\{1,\rt(\cdot)/\varepsilon\}^3 \lesssim \varepsilon^{-3}$, 
we get the existence of $C>0$ independent of $\varepsilon \in (0,1]$
with 
\begin{equation}
\label{eq:thm:singular-approx-16}
\forall \boxt \in \Kt: \quad 
|\nabla^2 \ut(\boxt)| 
\leq C (\rt(\boxt))^{-1-\beta} \varepsilon^{-5+\beta}. 
\end{equation}
By scaling this bound and invoking
Lemma~\ref{lemma:hat_Pi_infty}, (\ref{item:lemma:hat_Pi_infty-ii}), 
we get with $h_{\Kt}:=\operatorname{diam} \Kt$
for any fixed $\widetilde \beta \in (\beta,1)$ 
\begin{align}
\label{eq:thm:singular-approx-17}
& \|\ut - \tPiT \ut\|_{L^\infty(\Kt)} + 
\|\nabla( \ut - \tPiT \ut)\|_{L^2(\Kt)} 
\leq C q^4 h_{\Kt}^{1-\widetilde \beta} \|\rt^{\widetilde \beta} \nabla^2 \ut\|_{L^2(\Kt)} 
\\
\nonumber 
& \qquad 
\stackrel{(\ref{eq:thm:singular-approx-16})}{\leq} 
C q^4 h_{\Kt}^{1-\widetilde \beta} \varepsilon^{-5+\beta} h_{\Kt}^{\widetilde \beta - \beta} 
\stackrel{q\varepsilon \ge \kappa}{\leq} C  q^{4+5-\beta} h_{\Kt}^{1-\beta}
\lesssim  q^{9-\beta} \sigma^{n(1-\beta)}. 
\end{align}

\emph{Step 2: Elements not abutting on $\bO$:}
From Lemma~\ref{lemma:properties-mesh-patches}, (\ref{item:lemma:properties-mesh-patches-iva})
we get $\rt (\cdot) \lesssim h_{\Kt}= \operatorname{diam} \Kt$ on $\Kt$. 
The regularity assumption (\ref{eq:lemma:asymptotic-case-10}) then implies 
that there exist (suitably adjusted) constants $C$, $\gamma$ such that 
for all $p \in {\mathbb N}_0$
\begin{align}
\label{eq:thm:singular-approx-30}
\|\nabla^p (\widetilde u - \widetilde u(\bO))\|_{L^\infty(\Kt)} 
&\leq 
C \gamma^p \varepsilon^{-3+\beta} h_{\Kt}^{1-\beta} \rt^{-p} 
   \max\{p+1,\varepsilon^{-1}\}^p  \;.
\end{align}
We now consider the approximation on triangles and rectangles separately. 

\emph{Step 2.1: $\Kt$ is a triangle $\triangle$.} 
Lemma~\ref{lemma:properties-mesh-patches}, 
(\ref{item:lemma:properties-mesh-patches-ii}) 
implies
in particular that $h_{\Kt} \lesssim \rt (\cdot)$ on $\Kt$. 
Scaling the bounds (\ref{eq:thm:singular-approx-30})
to the reference element $\widehat K = \widehat T$ therefore gives 
for $\widehat u:= \ut \circ A_{\Kt}$, where $A_{\Kt}:\Kh\rightarrow \Kt$ 
is the affine element map for $\Kt$, 
the existence of constants $C$, $\gamma>0$ such that 
%
\begin{equation} \label{eq:thm:singular-approx-45}
\forall p \in {\mathbb N}_0 \colon \;\;
\|\widehat{\nabla}^p (\widehat u - \widehat u(\bO))\|_{L^\infty(\Kh)} 
\leq C \gamma^p \varepsilon^{-3+\beta} 
h_{\Kt}^{1-\beta} 
\max\{p+1,\varepsilon^{-1}\}^{p}
\;. 
\end{equation}
In order to be able to apply the approximation properties of 
Lemma~\ref{lemma:hat_Pi_infty}, we note 
\begin{align}
\nonumber 
\max\{p+1,\varepsilon^{-1}\}^p  &= 
\max\{(p+1)^p,\varepsilon^{-p} (p+1)^{-p} (p+1)^p\} \\
\nonumber 
& = 
(p+1)^p \max\{1,\varepsilon^{-p} (p+1)^{-p}\} 
\leq (p+1)^p \max\left\{1,\frac{(1/\varepsilon)^p}{p!}\right\} \\
\label{eq:estimate-max}
& \leq 
(p+1)^p e^{1/\varepsilon}  
 \stackrel{q \varepsilon \ge \kappa}{\leq} (p+1)^p e^{q/\kappa} . 
\end{align}
Inserting \eqref{eq:estimate-max} into \eqref{eq:thm:singular-approx-45}
yields that there are constants $C>0$, $\gamma>0$ such that
\begin{equation}
\label{eq:thm:singular-approx-47}
\forall p\in {\mathbb N}_0\colon \;\; 
\|\nabla^p (\widehat u - \widehat u(\bO))\|_{L^\infty(\Kh)}
\leq C e^{q/\kappa} \varepsilon^{-3+\beta} h_{\Kt}^{1-\beta} \gamma^p 
(p+1)^p.
\end{equation}
We are in a position to apply Lemma~\ref{lemma:hat_Pi_infty}. 
The parameter $\delta$ in (\ref{eq:hat-approx-constraint-triangle}) 
is determined by $\gamma$. In view of $q \varepsilon \ge \kappa$, 
we can ensure condition 
\eqref{eq:hat-approx-constraint-triangle} by selecting $\kappa$
sufficiently large to obtain from Lemma~\ref{lemma:hat_Pi_infty} with
a $b > 0$ depending only on $\gamma$ 
\begin{equation*}
\|\widehat u - \PiT \widehat u\|_{W^{1,\infty}(\widehat K)} 
\leq C h_{\Kt}^{1-\beta} e^{q/\kappa} \varepsilon^{-3+\beta} e^{-bq}, 
\end{equation*}
We may assume that $\kappa$ is so large that 
$1/\kappa - b \leq - b/2$ to absorb $e^{q/\kappa}$. 
Finally, in view of $q \varepsilon \ge \kappa$, we can also absorb 
the factor $\varepsilon^{-3 + \beta} \lesssim q^{3-\beta}$ in the 
exponentially decaying one by adjusting $b$. 
Upon scaling from $\widehat{K}$ to $\widetilde{K}$ we get the existence
of 
constants $b$, $C>0$  such that 
\begin{equation}
\label{eq:thm:singular-approx-49}
\forall q\in \bbN \colon\;\;
\|\ut - \tPiT \ut\|_{L^\infty(\Kt)}  + 
\|\nabla( \ut - \tPiT \ut)\|_
{L^2(\Kt)} 
\leq C h_{\Kt}^{1-\beta} e^{-bq}.
\end{equation}

\emph{Step 2.2: $\Kt$ is a rectangle $\Box$.} 
We argue as in the case of a triangle in Step~{2.1}.
Starting point is again the regularity assertion 
\eqref{eq:lemma:asymptotic-case-10}. 
The rectangle $\Kt$ has side lengths $h_{\Kt,\yt} \leq h_{\Kt,\xt} \leq 1$.
From Lemma~\ref{lemma:properties-mesh-patches}, (\ref{item:lemma:properties-mesh-patches-iv})
we have $h_{\Kt,\yt} \leq h_{\Kt,\xt} \leq C \rt (\cdot)$ on $\Kt$. 
Hence, the (anisotropic) scaling to the reference square
$\widehat{S}$
of the estimates \eqref{eq:lemma:asymptotic-case-10}
yields, for all $(n,m) \in {\mathbb N}_0^2$, in view of 
\eqref{eq:estimate-max}
\begin{align}
\nonumber 
\|\partial_x^m \partial_y^n (\widehat{u} - \widehat u(\bO))\|_{L^\infty(\widehat{S})}
&\leq C \varepsilon^{-3+\beta} h_{\Kt,\xt}^{1-\beta} \gamma^{n+m} 
h_{\Kt,\xt}^m h_{\Kt,\yt}^n h_{\Kt,\xt}^{-(n+m)} 
e^{q/\kappa} (n+m)^{n+m}  \\
\label{eq:thm:singular-approx-60}
&\stackrel{(\ref{eq:binom})}{\leq} \varepsilon^{-3+\beta} h_{\Kt,\xt}^{1-\beta} \gamma^{n+m} 
e^{q/\kappa} n! m!,
\end{align}
where we again suitably adjusted the value of $\gamma$.
Lemma~\ref{lemma:hat_Pi_1_infty} 
(with $h_y = h_{\Kt,\yt}/h_{\Kt,\xt} \leq 1$ and $\varepsilon_x = \varepsilon_y = 1$ there) 
yields with the regularity estimates \eqref{eq:thm:singular-approx-60}
the existence of constants $C$, $b>0$ such that 
for all $q\in {\mathbb N}_0$
\begin{align*}
\|\widehat u - \PiGL \widehat u\|_{L^\infty(\widehat K)} + 
\|\partial_\xh (\widehat u - \PiGL \widehat u)\|_{L^\infty(\widehat K)}
&\leq C \varepsilon^{-3+\beta} h_{\Kt,\xt}^{1-\beta} e^{-b q}, \\
\|\partial_\yh (\widehat u - \PiGL \widehat u)\|_{L^\infty(\widehat K)}
&\leq C \varepsilon^{-3+\beta} h_{\Kt,\xt}^{1-\beta} \frac{h_{\Kt,\yt}}{h_{\Kt,\xt}}e^{-b q}, 
\end{align*}
where the factor $e^{q/\kappa}$ was absorbed again in the 
exponentially decaying term by taking $\kappa$ sufficiently large. 
We obtain on $\widetilde{S}$
\begin{subequations}
\label{eq:thm:singular-approx-70} 
\begin{align}
\nonumber 
\nonumber 
\|\partial_\xt (\ut - \tPiGL \ut)\|_{L^2(\widetilde{S})}
&\leq 
C \varepsilon^{-3+\beta} 
\sqrt{h_{\Kt,\xt} h_{\Kt,\yt}} h_{\Kt,\xt}^{-1} h_{\Kt,\xt}^{1-\beta} e^{-b q} 
\\
& \stackrel{q \varepsilon \ge \kappa}{\leq}
C \sqrt{h_{\Kt,\yt}/h_{\Kt,\xt}} h_{\Kt,\xt}^{1-\beta} e^{-b q}, 
\\
\|\partial_\yt (\ut - \tPiGL \ut)\|_{L^2(\widetilde{S})}
&\leq C \sqrt{{h_{\Kt,\yt}}/{h_{\Kt,\xt}}} h_{\Kt,\xt}^{1-\beta} e^{-b q}, 
\\
\|(\ut - \tPiGL \ut)\|_{L^\infty(\widetilde{S})}
&\leq 
C h_{\Kt,\xt}^{1-\beta} e^{-b q},
\end{align}
\end{subequations}
where again we adjusted the values of the constants $b$, $C$ 
in both estimates to absorb algebraic factors in $q$.  

\emph{Step 3: Summation of the elemental errors:}
We note that the element size $h_{\Kt}$ 
of the elements abutting on $\bO$ is $h_{\Kt} \sim \sigma^n$.  
For the finitely many contributions 
from the (triangular) elements $\Kt$ touching $\bO$
we have by (\ref{eq:thm:singular-approx-17}) 
the existence of $C>0$ such that
for every $q\geq 1$ 
\begin{align*}
\sum_{\Kt\colon  \bO \in \overline{\Kt}}
\|\ut - \widetilde \Pi_q \ut\|^2_{H^1(\Kt)} 
&\stackrel{\eqref{eq:thm:singular-approx-17}}{\leq} 
C q^{18} \sigma^{2n (1-\beta)}. 
\end{align*}
The sum of squared error contributions
over all triangular elements not touching $\bO$ 
is also bounded by $e^{-2 b q}$ by combining (\ref{eq:thm:singular-approx-49}) 
and Lemma~\ref{lemma:properties-mesh-patches}, (\ref{item:lemma:properties-mesh-patches-v}). 
Likewise, the sum over all rectangular elements is bounded by 
$e^{-2 b q}$ by combining (\ref{eq:thm:singular-approx-70}) 
and Lemma~\ref{lemma:properties-mesh-patches}, (\ref{item:lemma:properties-mesh-patches-vi}). 

\emph{Proof of (\ref{item:lemma:asymptotic-case-ii}):} 
The proof is similar to the proof of case 
(\ref{item:lemma:asymptotic-case-i}) and can be obtained from it by 
formally setting $\beta = 0$ and $\rt \equiv 1$ and dropping 
the error contribution $q^9 \sigma^{n (1-\beta)}$ that is due to the 
small elements touching $\bO$. 
\end{proof}
\subsubsection{$hp$-FE approximation of boundary layer functions}
\label{sec:bdy-layer-approx-geo-mesh}
\begin{lemma}[approximation of boundary layer functions]
\label{lemma:bdy-layer-approx} 
Fix $c_1 > 0$.  
\begin{enumerate}[(i)] 
\item 
\label{item:lemma:bdy-layer-approx-i}
Let $\widetilde{\calT} \in \{\widetilde{\calT}^{\Mi,\text{\rm half},L,n}_{geo,\sigma}, 
\widetilde{\calT}^{\BL,L}_{geo,\sigma}\}$. 
Let $\widetilde{\calT}' \subset \widetilde{\calT}$ and 
let $\calO:= \operatorname*{interior} \bigl( \cup \{\overline{\Kt}\,|\, \Kt \in \widetilde{\calT}'\}\bigr)$ 
be the union of the elements of $\widetilde{\calT}'$. 
Let $\widetilde u$ be analytic on $\calO$ and satisfy for some 
$C_u$, $\gamma$, $\alpha > 0$, $\varepsilon \in (0,1]$ 
and for all $(m,n) \in {\mathbb N}_0^2$ and all 
$\boxt =  (\xt,\yt)\in \calO$
\begin{equation}
\label{eq:lemma:bdy-layer-approx-10} 
|\partial_{\xt}^{m} \partial_{\yt}^{n} \widetilde u(\boxt)| 
\leq C_u \gamma^{n+m} m! \max\{n,\varepsilon^{-1}\}^{n} e^{-\alpha \yt/\varepsilon} . 
\end{equation}
Assume that $L$ is such that the \emph{scale resolution condition} 
\begin{equation}
\label{eq:L-eps-resolution}
\sigma^L \leq c_1 \varepsilon
\end{equation}
is satisfied. 
Then there are constants $C$, $b > 0$ depending only on $\gamma$, $\alpha $, $c_1$, $\sigma$ 
such that 
\begin{equation}
\label{lemma:bdy-layer-approx-100} 
\forall q \in {\mathbb N}\colon\;\;
\|\ut - \widetilde \Pi_q \ut\|_{L^\infty(\calO)}
+ \varepsilon \|\nabla (\ut - \widetilde \Pi_q \ut)\|_{L^\infty(\calO)}
\leq C C_u e^{-bq} . 
\end{equation}
\item 
\label{item:lemma:bdy-layer-approx-ii}
Let 
$\widetilde{\calT}^{\prime\prime} \subset 
\widetilde{\calT}^{\Co,\text{\rm half},n}_{geo,\sigma}$ 
or 
$\widetilde{\calT}^{\prime\prime} \subset \widetilde{\calT}^{\Co,\text{\rm half},\text{\rm flip},n}_{geo,\sigma}$. 
Let 
$\calO:= \operatorname*{interior}\bigl(\cup \{\overline{\Kt}\,|\, 
          \Kt \in \widetilde{\calT}^{\prime\prime}\}\bigr)$
be the union of the elements of $\widetilde{\calT}^{\prime\prime}$. 
Let $\widetilde u$ be analytic on $\calO$ and satisfy  for some 
$C_u$, $\gamma$, $\alpha > 0$, $\varepsilon \in (0,1]$ 
\begin{equation}
\label{eq:lemma:bdy-layer-approx-111}
\forall p \in {\mathbb N}_0\ 
\forall \boxt \in \calO\colon \;\;
|\nabla^p \ut(\boxt)|
\leq C_u \gamma^{p} \max\{p,\varepsilon^{-1}\}^{p} e^{-\alpha \rt(\boxt)/\varepsilon}.  
\end{equation}
Assume that $n$ is such that 
$c_1>0$, $n\in \bbN$ satisfies the \emph{scale scale resolution condition}
\begin{equation}
\label{eq:L-eps-resolution-n}
\sigma^n \leq c_1 \varepsilon
\end{equation}
is satisfied.
Then, there are constants 
$C$, $b > 0$ (depending only on $\gamma$, $\alpha $, $c_1$, $\sigma$)
such that 
\begin{equation}
\label{lemma:bdy-layer-approx-110} 
\forall q \in {\mathbb N}\colon \;\;
\|\ut - \widetilde \Pi_q \ut\|_{L^\infty(\calO)}
+ \varepsilon \|\nabla (\ut - \widetilde \Pi_q \ut)\|_{L^\infty(\calO)}
\leq C C_u e^{-bq} . 
\end{equation}
\end{enumerate}
\end{lemma}
\begin{proof}
\emph{Proof of (\ref{item:lemma:bdy-layer-approx-ii}):} 
We only consider the case $\calT^{\prime\prime} \subset 
\widetilde{\calT}^{\Co,\text{\rm half},n}_{geo,\sigma}$ as the case 
$\calT^{\prime\prime} \subset 
\widetilde{\calT}^{\Co,\text{\rm half},\text{\rm flip},n}_{geo,\sigma}$ is handled
similarly. 
We note that the patch $\widetilde{\calT}^{\Co,\text{\rm half},n}_{geo,\sigma}$ consists of 
triangles only, which are all shape-regular. Let $\Kt \subset \calO$ be a triangle and let 
$h_{\Kt} = \operatorname{diam}\Kt$. 
In the case that $\Kt$ touches $\bO$, 
the condition \eqref{eq:L-eps-resolution-n} implies that 
$h_{\Kt} \lesssim \sigma^n \lesssim c_1 \varepsilon$ so that 
\begin{equation}\label{eq:ScalRes}
\frac{h_{\Kt}}{q \varepsilon} \lesssim \frac{1}{q}. 
\end{equation}
Hence, 
for every fixed choice of the constant $c_1$
there exists $q_0 = q_0(c_1) \in {\mathbb N}$ (independent of $\eps$) 
such that 
for every $q \ge q_0$ one has the 
scale resolution condition \eqref{eq:hat-approx-constraint-triangle}.
Then, 
Lemma~\ref{lemma:hat_Pi_infty}, (\ref{item:lemma:hat_Pi_infty-iii}) 
implies for  suitable $b > 0$ (independent of $\eps$) 
\begin{equation}
\label{eq:thm:singular-approx-100}
\|\widetilde u - \tPiT \widetilde u\|_{L^{\infty}(\Kt)} + \varepsilon 
\|\nabla( \widetilde u - \tPiT \widetilde u)\|_{L^{\infty}(\Kt)} 
\lesssim e^{-bq}. 
\end{equation}
If $\Kt$ does not touch $\bO$, we distinguish between two further cases. 
In the first case,  we assume that $h_{\Kt}/(q \varepsilon) \leq \delta$ 
and proceed as above: The 
scale resolution condition \eqref{eq:hat-approx-constraint-triangle} 
is satisfied, and we arrive again at \eqref{eq:thm:singular-approx-100}. 
In the case $h_{\Kt}/(q \varepsilon) \ge \delta$, 
we note that 
Lemma~\ref{lemma:properties-mesh-patches}, (\ref{item:lemma:properties-mesh-patches-ii}) 
implies 
$\operatorname{dist}(\Kt,\bO) \ge c_2 h_{\Kt} \ge c_2 \delta q \varepsilon$. 
Hence, by the decay properties of $\ut$ in (\ref{eq:lemma:bdy-layer-approx-111})
we have 
\begin{equation}
\|\ut\|_{L^\infty(\Kt)}  + \varepsilon \|\nabla \ut\|_{L^\infty(\Kt)} 
\leq C e^{-\alpha c_2 h_{\Kt}/\varepsilon}
\leq C e^{-\alpha c_2 q \delta}. 
\end{equation}
In view of the  stability properties (\ref{eq:tildePiq-stable}), we conclude 
\begin{align}
\label{eq:thm:singular-approx-110}
\|\ut - \widetilde \Pi_q \ut\|_{L^{\infty}(\Kt)} + 
\varepsilon \|\nabla( \ut - \widetilde \Pi_q \ut)\|_{L^{\infty}(\Kt)} 
\lesssim e^{-bq}. 
\end{align}
\emph{Proof of (\ref{item:lemma:bdy-layer-approx-i}):} 
We distinguish between triangular and rectangular elements. 

\emph{Approximation of $\ut$ on triangular elements $\Kt$:} 
Triangular elements do not appear in boundary layer patches 
$\widetilde{\calT}^{\BL,L}_{geo,\sigma}$ 
but only in $\widetilde{\calT}^{\Mi,\text{\rm half},L,n}_{geo,\sigma}$. 
For patches $\widetilde{\calT}^{\Mi,\text{\rm half},L,n}_{geo,\sigma}$ 
inspection (cf.\ Fig.~\ref{fig:half-figures}) shows that two 
types of triangles occur: the first type are the triangles $\Kt$ 
in $\widetilde T\setminus \widetilde T_1$ on which one has 
$\rt(\xt,\yt) \sim \yt$ (uniformly in $L$, $n$). 
The second type are  the triangles in $\widetilde T_1$. 
For the first type, we have from (\ref{eq:lemma:bdy-layer-approx-10}) 
the regularity assertion 
(with suitably adjusted $C_u$, $\gamma$, $\alpha$ independent of $\varepsilon$)
$$
\forall \boxt \in {\Kt} \;\forall  n \in {\mathbb N}_0\colon \quad 
|\nabla^n \ut(\boxt)| \leq C_u \gamma^n \max\{n,\varepsilon^{-1}\}^n 
e^{-\alpha \rt(\boxt)/\varepsilon} 
.
$$
This is the same regularity assumption that underlies the proof of part 
(\ref{item:lemma:bdy-layer-approx-ii}) of the lemma so that the same arguments
can be brought to bear as in the case of part
(\ref{item:lemma:bdy-layer-approx-ii}). For the second type of triangles, 
i.e., $\Kt \subset \widetilde T_1 \subset (0,\sigma^L)^2$, the resolution
assumption (\ref{eq:L-eps-resolution}) implies for the element size 
$h_{\Kt} \lesssim \sigma^L \lesssim \varepsilon$. Hence, again 
Lemma~\ref{lemma:hat_Pi_infty}, (\ref{item:lemma:hat_Pi_infty-iii}) is applicable 
and yields the desired exponential approximation. 

\emph{Approximation of $\ut$ on rectangular elements $\Kt$:} 
The case of rectangular elements $\Kt$ with side lengths 
$h_{\Kt,x}$, $h_{\Kt,y}$ is similar to the case of triangles. We note that 
the patches $\widetilde{\calT}^{\BL,L}_{geo,\sigma}$
and $\widetilde{\calT}^{\Mi,\text{\rm half},L,n}_{geo,\sigma}$ are such that 
$h_{\Kt,\yt} \leq h_{\Kt,\xt} \leq 1$. 
The anisotropic scaling from $\Kt$ to $\Kh$ and the regularity 
assumption (\ref{eq:lemma:bdy-layer-approx-10}) show that the 
pull-back $\widehat{u}$ to $\Kh$ satisfies for 
all $(m,n) \in {\mathbb N}_0^2$
\begin{equation*}
\|\partial_\xh^m \partial_\yh^n \widehat{u}\|_{L^\infty(\Kh)} 
\leq C e^{-\alpha \operatorname{dist}(\Kt,\{\yt=0\})/\varepsilon} 
h_{\Kt,\xt}^m h_{\Kt,\yt}^n
\gamma^{n+m} m! \max\{n+1,\varepsilon^{-1}\}^n\; . 
\end{equation*}
That is, $\widehat{u}$ satisfies the 
analytic regularity condition (\ref{eq:lemma:hat_Pi_1_infty-reg}) 
with $\varepsilon_y = \varepsilon$, $\varepsilon_x = 1$, $h_x = h_{\Kt,\xt}$, 
$h_y = h_{\Kt,\yt}$ and $C_u = 
C e^{-\alpha \operatorname{dist}(\Kt,\{\yt=0\})/\varepsilon}$. 
We observe that the resolution condition 
(\ref{eq:hat-approx-constraint-rectangle}) 
can be achieved if $\Kt$ touches the line $\{\yt = 0\}$ 
in view of \eqref{eq:L-eps-resolution} provided that 
$q \ge q_0 \ge 1$ for suitable $q_0$ (depending on $c_1$, $\sigma$, $\gamma$).
If $\Kt$ does not touch the line $\{\yt = 0\}$, then two cases may occur: 
If the resolution condition \eqref{eq:hat-approx-constraint-rectangle} 
is still satisfied then we obtain again exponential convergence.
If not, we note that $h_{\Kt,\yt} \leq h_{\Kt,\xt}$ and that 
we may assume $h_{\Kt,\xt}/q \leq \delta/2$ by assuming $q \ge q_0 \ge 1$
(note: trivially, $h_{\Kt,\xt} \leq 1$ so that $q_0 \ge 2/\delta$ will work).
Furthermore, 
Lemma~\ref{lemma:properties-mesh-patches}, (\ref{item:lemma:properties-mesh-patches-iii}) 
reveals again that $\operatorname{dist} (\Kt,\{\yt = 0\}) \ge c_3 h_{\Kt,\yt}$; 
since $h_{\Kt,\yt}/(\varepsilon q) \ge \delta/2$ we 
get $\operatorname{dist}(\Kt,\{\yt=0\})/\varepsilon \ge q  c_3 \delta/2$. 
Hence, 
$\exp(-\alpha \operatorname{dist}(\Kt,\{\yt=0\})/\varepsilon) 
\leq \exp(-q \alpha c_2 \delta/2)$ and we may 
argue 
as in the case of triangles that $\ut$ is exponentially (in $q$) 
small on $\Kt$. 
The stability of $\widetilde \Pi_q$ given in
(\ref{eq:tildePiq-stable}) then concludes the argument. 
\end{proof}
\subsubsection{$hp$-FE approximation of corner layer functions}
\label{sec:corner-layer-approx-geo-mesh}
\begin{lemma}[approximation of corner layer functions]
\label{lemma:corner-layer-approx} 
Fix $c_1 > 0$. 
Let $\widetilde{\calT} \in 
\{
\widetilde{\calT}^{\Mi,\text{\rm half},L,n}_{geo,\sigma}, 
\widetilde{\calT}^{\Co,\text{\rm half},n}_{geo,\sigma}, 
\widetilde{\calT}^{\Co,\text{\rm half},\text{\rm flip}, n}_{geo,\sigma}, 
\widetilde{\calT}^{\Co,n}_{geo,\sigma}, 
\widetilde{\calT}^{\Te,n}_{geo,\sigma}
\}. 
$
Let $\widetilde{\calT}^\prime \subset \widetilde{\calT}$ and 
let $\calO :=\operatorname*{interior} 
     \bigl(\cup \{\overline{\Kt}\,|\, \Kt \in \widetilde{\calT}^\prime\}\bigr)$ 
be the union of the elements of $\widetilde{\calT}^\prime$. 
Let $\ut$ be analytic on $\calO$ and satisfy for some 
$\beta \in [0,1)$, $\varepsilon \in (0,1]$, 
$C_u$, $\gamma$, $\alpha  > 0$
\begin{equation}
\forall \boxt \in \calO \quad \forall p \in {\mathbb N}_0 \colon\;\;
|\nabla^p \ut(\boxt)|
\leq C_u \varepsilon^{\beta-1} \gamma^{p} (\rt(\boxt))^{1-\beta-p} p!  e^{-\alpha \rt(\boxt)/\varepsilon} . 
\end{equation}
Assume that $n \in {\mathbb N}$ is such that the scale resolution condition 
\begin{equation}
\label{eq:L-eps-resolution-n-foo}
\sigma^n \leq c_1 \varepsilon. 
\end{equation}
is satisfied. 
Then there are constants $C$, $b > 0$ depending only on $\gamma$, $\alpha $, $c_1$, $\sigma$, 
and $\beta$ (in particular, they are independent of $\varepsilon$, $q$, $n$, $L$)  such that for all $q \in {\mathbb N}$ 
\begin{align}
\label{eq:lemma:corner-layer-approx-10} 
\|\ut - \widetilde \Pi_q \ut\|_{L^\infty(\calO)}
& \leq C C_u \left( e^{-b q} + q^4 \varepsilon^{\beta-1} \sigma^{n(1-\beta)}
             \right),\\
\label{eq:lemma:corner-layer-approx-20} 
\|\ut - \widetilde \Pi_q \ut\|_{L^2(\calO)}
+ \varepsilon \|\nabla (\ut - \widetilde \Pi_q \ut)\|_{L^2(\calO)}
& \leq C C_u \varepsilon \left( e^{-b q} + q^4 \varepsilon^{\beta-1} \sigma^{n(1-\beta)}
             \right). 
\end{align}
In the estimates  
(\ref{eq:lemma:corner-layer-approx-10}), 
(\ref{eq:lemma:corner-layer-approx-20})
the term $q^4 \varepsilon^{\beta-1} \sigma^{n(1-\beta)}$
can be dropped if $\calO$ does not touch $\bO$. 
\end{lemma}
\begin{proof}
The approximation of functions of corner layer type $\ut$ 
proceeds structurally along the same lines 
as in the case of the singularity functions in 
Lemma~\ref{lemma:asymptotic-case}. 
We distinguish between the elements touching $\bO$ and the remaining ones. 

\emph{$\Kt$ touches $\bO$:} 
Selecting $\widetilde \beta \in (\beta,1)$ we obtain by arguing
as in \eqref{eq:thm:singular-approx-17} 
\begin{align}
\nonumber 
 \|\ut - \tPiT \ut\|_{L^\infty(\Kt)} + 
\|\nabla( \ut - \tPiT \ut)\|_{L^2(\Kt)} 
&\lesssim h_{\Kt}^{1-\widetilde \beta} q^4 \| \rt^{\widetilde\beta} \nabla^2 \ut\|_{L^2(\Kt)} \\
\label{eq:thm:singular-approx-170} 
& \lesssim  q^4 (h_{\Kt}/\varepsilon)^{1-\beta} . 
\end{align}
Since $h_{\Kt} \lesssim \sigma^{n}$ for elements $\Kt$ touching $\bO$, 
their contributions lead to the term $q^4 \varepsilon^{\beta-1} \sigma^{n(1-\beta)}$.  

\emph{$K$ does not touch $\bO$:} We distinguish between triangular and rectangular
elements. 

\emph{Step~{1}: $\Kt$ is a triangular element:} As in the case of the approximation 
in Lemma~\ref{lemma:asymptotic-case}, we get from 
Lemma~\ref{lemma:hat_Pi_infty}, (\ref{item:lemma:hat_Pi_infty-iii}) 
and scaling that (for suitably adjusted $C$, $\alpha$) 
\begin{align} 
\label{eq:thm:singular-approx-190} 
 \|\ut - \tPiT \ut\|_{L^{\infty}(\Kt)}  + 
 \|\nabla( \ut - \tPiT \ut)\|_{L^{2}(\Kt)}  
\leq C (h_{\Kt}/\varepsilon)^{1-\beta} e^{-b q} e^{-\alpha h_{\Kt}/\varepsilon}. 
\end{align}
\emph{Step~{2}: $\Kt$ is a rectangular element:} We recall $h_{\Kt,\yt}  \leq h_{\Kt,\xt} \leq 1$. 
By Lemma~\ref{lemma:properties-mesh-patches}, (\ref{item:lemma:properties-mesh-patches-iv}) 
we have $\rt(\cdot)\sim h_{\Kt,\xt}$ on $\Kt$. 
As in the case of Lemma~\ref{lemma:asymptotic-case} we observe for the pull-back to the
reference element $\Kh$
\begin{equation*}
\forall (m,n) \in {\mathbb N}_0^2\colon\;\;
\|\partial_\xh^m \partial_\yh^n \widehat u\|_{L^\infty(\widehat K)} 
\leq C (h_{\Kt,\xt}/\varepsilon)^{1-\beta} e^{-\alpha h_{\Kt,\xt}/\varepsilon} 
\gamma^{n+m} n! m! h_{\Kt,\xt}^{m-(m+n)} h_{\Kt,\yt}^n . 
\end{equation*}
Using  
Lemma~\ref{lemma:hat_Pi_1_infty}, (\ref{item:lemma:hat_Pi_1_infty-ii}) 
(with $\varepsilon_x = \varepsilon_y = 1$ and $h_y = h_{\Kt,\yt}/h_{\Kt,\xt}$, $h_x = 1$ there), we arrive at 
\begin{align} 
\label{eq:thm:singular-approx-195}
 \|\widehat u - \PiGL \widehat u
             \|_{L^{\infty}(\widehat K)} 
& \leq C (h_{\Kt}/\varepsilon)^{1-\beta} e^{-b q} e^{-\alpha h_{\Kt,\xt}/\varepsilon}, \\
\label{eq:thm:singular-approx-200}
 \|\partial_\xh (\widehat u - \PiGL \widehat u)
             \|_{L^{\infty}(\widehat K)} 
& \leq C (h_{\Kt}/\varepsilon)^{1-\beta} e^{-b q} e^{-\alpha h_{\Kt,\xt}/\varepsilon}, \\
\label{eq:thm:singular-approx-210}
 \|\partial_\yh (\widehat u - \PiGL \widehat u)
             \|_{L^{\infty}(\widehat K)} 
& \leq C (h_{\Kt}/\varepsilon)^{1-\beta} \frac{h_{\Kt,\yt}}{h_{\Kt,\xt}} 
e^{-b q} e^{-\alpha h_{\Kt,\xt}/\varepsilon}. 
\end{align}
\emph{Step~{3} ($L^\infty$-bound):}
Since $\sup_{t > 0} t^{1-\beta} e^{-t} < \infty$, the $L^\infty$-estimates
follow easily from 
(\ref{eq:thm:singular-approx-170}), 
(\ref{eq:thm:singular-approx-190})
(\ref{eq:thm:singular-approx-195}).

\emph{Step 4 (energy norm estimate):}  
Proceeding as in Step~3 of the proof of Lemma~\ref{lemma:asymptotic-case}
we set $e_{\Kt}:= \ut - \widetilde \Pi_q \ut$ and get, using
$h_\Kt \lesssim \varepsilon$ for the elements abutting on $\bO$: 
\begin{align*}
& \sum_{\Kt\colon \Kt \text{ abuts on $\bO$}} 
\|e_\Kt\|^2_{L^2(\Kt)} + \varepsilon^2 \|\nabla e_\Kt\|^2_{L^2(\Kt)}  
\\
& 
\stackrel{\eqref{eq:thm:singular-approx-170}}{\lesssim }
q^8 \sum_{\Kt\colon \Kt \text{ abuts on $\bO$}} (h_{\Kt}^2+\varepsilon^2) (h_{\Kt}/\varepsilon)^{2(1-\beta)} 
\lesssim q^8 \varepsilon^{2\beta} \sigma^{2 n (1-\beta)}. 
\end{align*}
For the remaining elements, we consider the triangular elements and the 
rectangular ones. In both cases, we employ the simple observation 
\begin{equation}
\label{eq:thm:singular-approx-211}
\|e_{\Kt}\|_{L^2(\Kt)} \lesssim 
h_\Kt \|e_\Kt\|_{L^\infty(\Kt)} = 
\varepsilon \frac{h_\Kt}{\varepsilon} \|e_\Kt\|_{L^\infty(\Kt)} . 
\end{equation}
The sum over all triangles, collected
in ${\Tref}^\triangle$, yields by combining 
\eqref{eq:thm:singular-approx-190} and 
(\ref{eq:thm:singular-approx-211}) with 
Lemma~\ref{lemma:properties-mesh-patches}, 
(\ref{item:lemma:properties-mesh-patches-vii})
$$
\sum_{\Kt \in {\Tref}^\triangle} 
\|e_\Kt\|^2_{L^2(\Kt)} + \varepsilon^2 \|e_{\Kt}\|^2_{H^1(\Kt)} 
\lesssim \varepsilon^2 e^{-2b q}. 
$$
Likewise, the sum over all rectangular elements, collected in ${\Tref}^\square$, 
yields by combining 
Lemma~\ref{lemma:properties-mesh-patches}, 
(\ref{item:lemma:properties-mesh-patches-viii})
with 
\eqref{eq:thm:singular-approx-195}, 
\eqref{eq:thm:singular-approx-211} for the $L^2$-part and 
with 
\eqref{eq:thm:singular-approx-200}, 
\eqref{eq:thm:singular-approx-210}
for the $H^1$-part 
$$
\sum_{\Kt \in {\Tref}^\square} 
\|e_\Kt\|^2_{L^2(\Kt)} + \varepsilon^2 \|e_{\Kt}\|^2_{H^1(\Kt)} 
\lesssim \varepsilon^2 e^{-2b q}. 
$$
This concludes the proof. 
\end{proof}
\section{$hp$-FE approximation of singularly perturbed problems 
                                        on geometric boundary layer meshes}
\label{sec:sing-approx-geo-mesh}
\begin{figure}
\begin{center}
\begin{overpic}[width=0.35\textwidth]{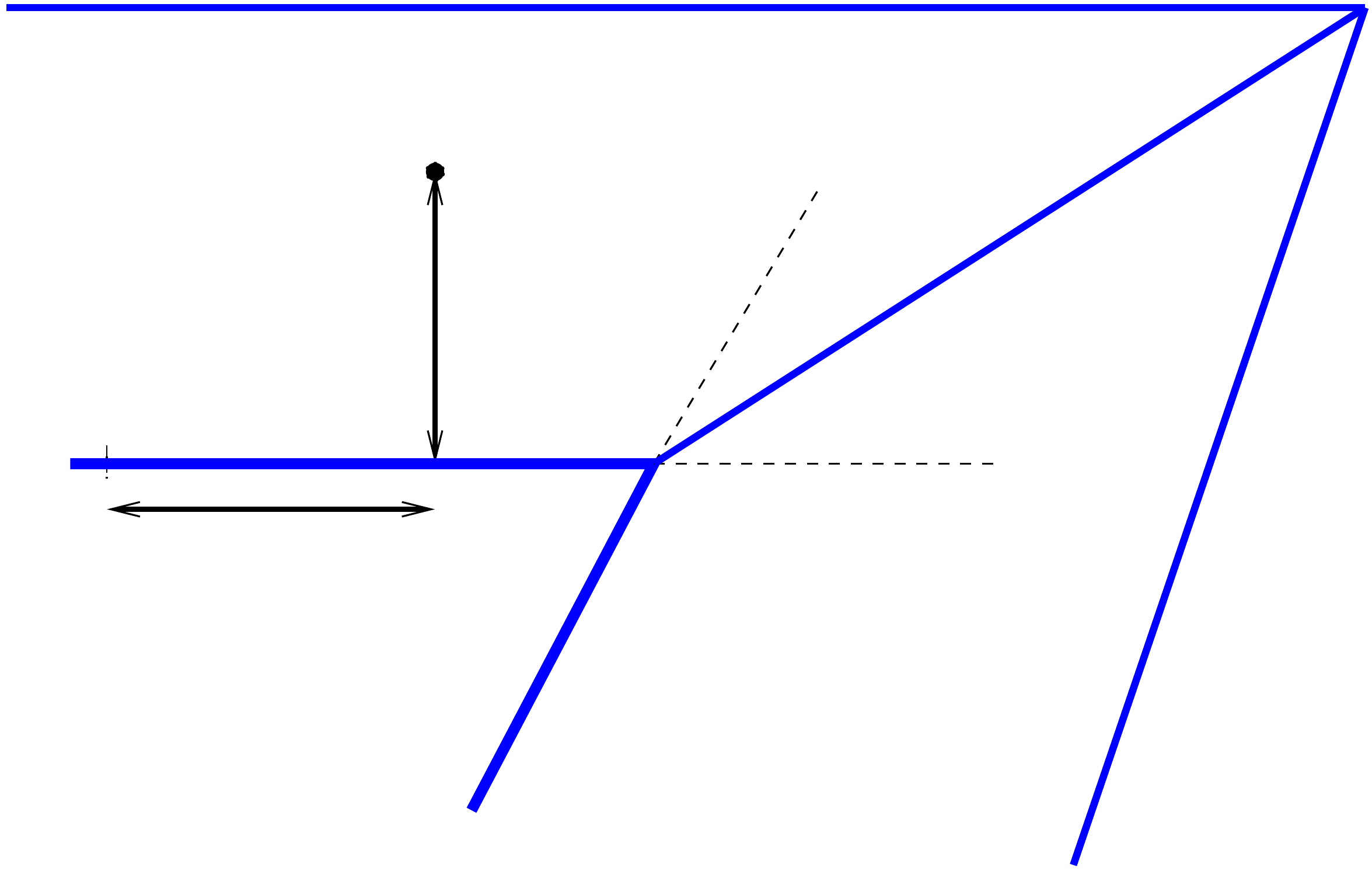}
\put(25,55){$\bx = (x,y)$}
\put(35,40){$\rho_j$}
\put(20,15){$\theta_j$}
\put(45,20){$\bA_j$}
\put(10,40){$\Omega_j$}
\put(60,10){$\Omega_{j+1}$}
\end{overpic}
\hfill 
\begin{overpic}[width=0.5\textwidth]{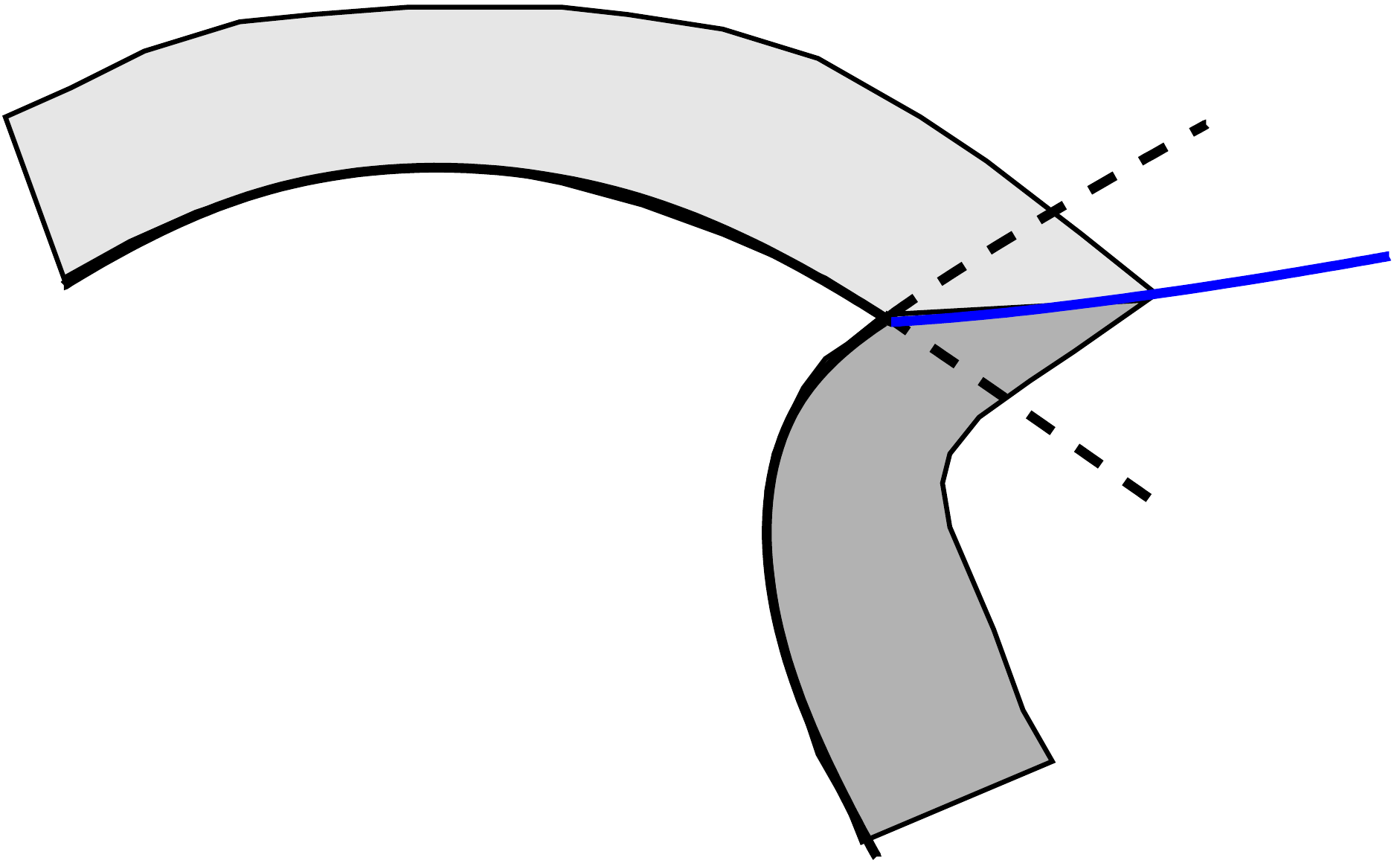}
\put(30,45){$\Gamma_j$}
\put(30,55){$\Omega_j$}
\put(53,38){$\mathbf{A}_j$}
\put(45,10){$\Gamma_{j+1}$}
\put(60,10){$\Omega_{j+1}$}
\put(80,20){$\widetilde \Gamma_{j}^\prime$}
\put(78,55){$\widetilde \Gamma_{j+1}^\prime$}
\put(92,38){$\Gamma_j^\prime$}
\end{overpic}
\end{center}
\caption{\label{fig:bdyfitted-coord} Left: boundary fitted coordinates $\psi_j: (\rho_j, \theta_j) \mapsto (x,y)$. Right: typical situation at a reentrant corner: boundary fitted coordinates $(\rho_j,\theta_j)$ and $(\rho_{j+1},\theta_{j+1})$ are valid in the regions $\Omega_j$, $\Omega_{j+1}$, respectively. 
$\widetilde\Gamma_j$ and $\widetilde\Gamma_{j+1}$ are analytic continuations of $\Gamma_j$, $\Gamma_{j+1}$. 
The analytic arc $\Gamma_j^\prime$ is such that the angles $\angle(\Gamma_j^\prime,\Gamma_j)$ and $\angle(\Gamma_{j+1},\Gamma_j^\prime)$ are both less than $\pi$.}
\end{figure}
\begin{figure}
\begin{center}
\begin{overpic}[width=0.35\textwidth]{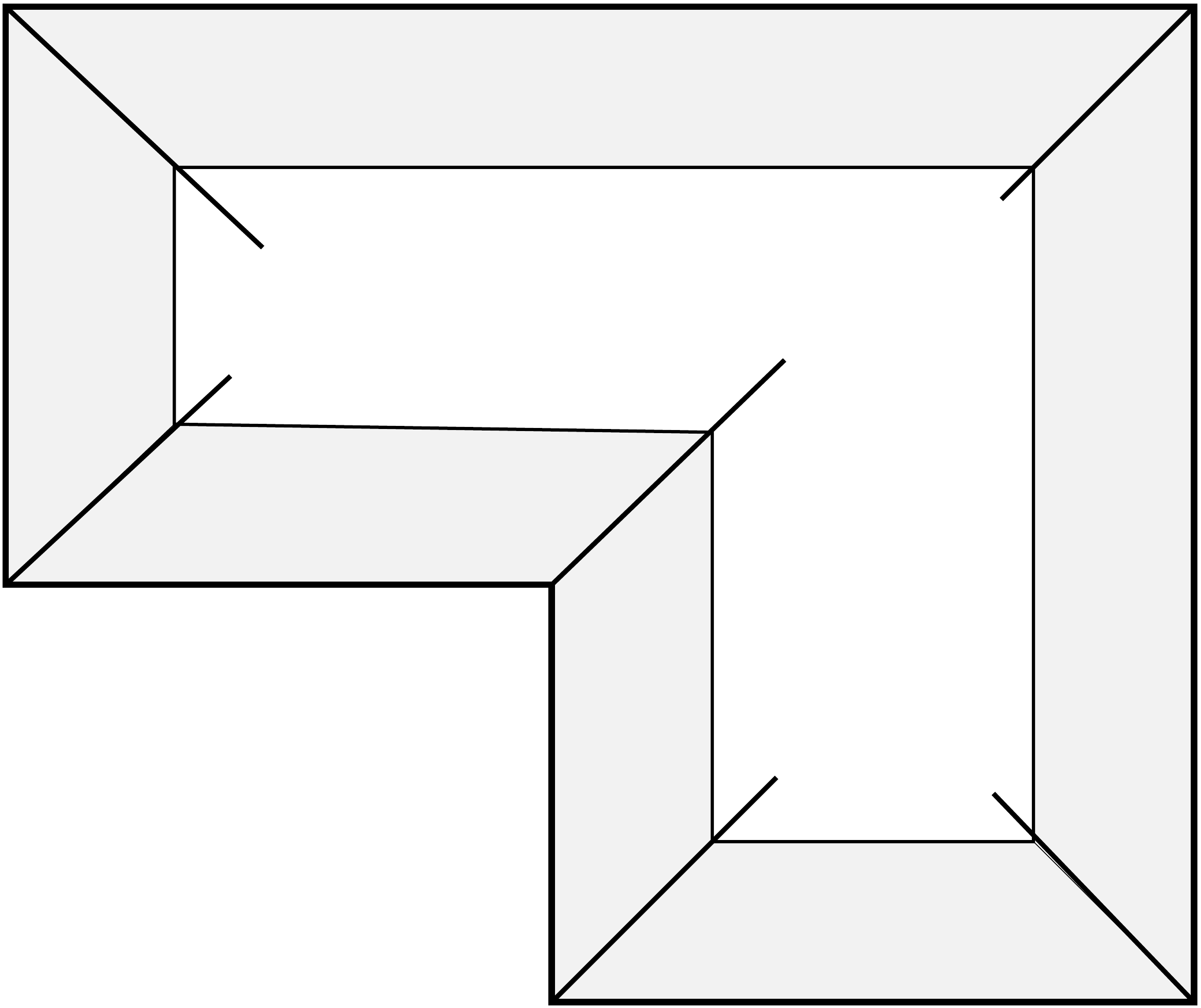}
\put(-20,29){$A_6 = A_0$}
\put(34,29){$\bA_1$}
\put(40,-3){$\bA_2$}
\put(105,0){$\bA_3$}
\put(105,85){$\bA_4$}
\put(-2,85){$\bA_5$}
\put(20,29){$\Gamma_1$}
\put(37,15){$\Gamma_2$}
\put(70,02){$\Gamma_3$}
\put(105,40){$\Gamma_4$}
\put(50,85){$\Gamma_5$}
\put(-6,50){\rotatebox{90}{$\Gamma_6 = \Gamma_0$}}
\put(25,40){$\Omega_1$}
\put(50,20){$\Omega_2$}
\put(70,10){$\Omega_3$}
\put(90,30){$\Omega_4$}
\put(60,75){$\Omega_5$}
\put(5,60){$\Omega_0 = \Omega_6$}
\end{overpic}
\end{center}
\caption{\label{fig:notation-1} The subdomains $\Omega_j$ 
          on which the boundary layer expansion $u^{\BL}_\varepsilon$ is defined 
          in terms of boundary fitted coordinates $(\rho_j,\theta_j)$.}
\end{figure}
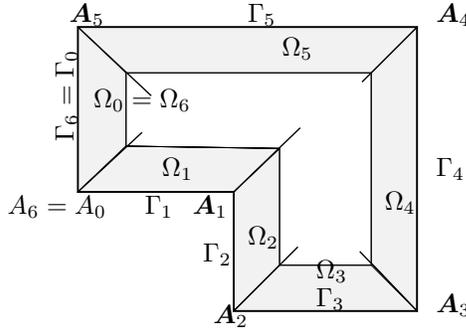
%
The principal result of the present paper is a robust, exponential approximation 
result for solutions of the singular perturbation problem 
\eqref{eq:sg-per}, \eqref{eq:sg-per-asmp} 
in curvilinear polygonal domains from spaces based on geometric boundary 
layer meshes that are able to resolve the length scales present in the problem. 
The meshes are independent of $\varepsilon$ but subject to the (weak) 
scale resolution condition (\ref{eq:L-eps-resolution-foo}).

\begin{theorem}
\label{thm:singular-approx} 
Let the Lipschitz domain $\Omega \subset {\mathbb R}^2$ 
be a curvilinear polygon with $J$ vertices as described in
Section~\ref{sec:GeoPrel}. 
Let $A$, $c$, $f$ satisfy (\ref{eq:sg-per-asmp}). 
Fix $c_1 > 0$. 
Let $\Tg$ be a geometric boundary layer mesh in sense of Definition~\ref{def:bdylayer-mesh}. 

Then 
there are constants $C$, $b >0 $, $\beta \in [0,1)$ 
depending solely on the data $A$, $c$, $f$, $\Omega$, 
on the parameter $c_1$, on the (fixed) macro-triangulation ${\mathcal T}^\M$, 
and on $\sigma \in (0,1)$ such that the following holds: 
If $\varepsilon \in (0,1]$ and $L$ satisfy
the scale resolution condition 
\begin{equation}
\label{eq:L-eps-resolution-foo}
\frac{\sigma^L}{\varepsilon} \leq c_1, 
\end{equation}
then 
for every $q$, $n\in \mathbb{N}$ 
the solution $u_\varepsilon \in H^1_0(\Omega)$ of \eqref{eq:sg-per} 
can be approximated from $S^{q}_0(\Omega,\Tg)$ such that 
\begin{align}
\label{eq:thm:singular-approx-10}
& \inf_{v \in S^{q}_0(\Omega,\Tg)}
\|u_\varepsilon - v\|_{\varepsilon,\Omega} \leq C q^{9} 
\left[\varepsilon^\beta \sigma^{(1-\beta)n}  + e^{-b q}\right], \\ 
N &:= \operatorname{dim} S^{q}_0(\Omega,\Tg)
\leq C \bigl( L^2 q^2 \operatorname{card} {\mathcal T}^\M + n q^2 J \bigr). 
\end{align}
\end{theorem}
\begin{proof}
Before proving the result, let us comment on the scale resolution 
\eqref{eq:L-eps-resolution-foo} and its relation to previous scale 
resolution conditions 
\eqref{eq:lemma:asymptotic-case-20}, 
\eqref{eq:L-eps-resolution-n-foo}, 
and 
\eqref{eq:L-eps-resolution}. 
Condition~\eqref{eq:L-eps-resolution-foo} ensures that $L$ layers of 
anisotropic refinement towards the boundary are performed, which is the 
condition~\eqref{eq:L-eps-resolution} needed to resolve functions of 
boundary layer type. Since $n \ge L$ (by Def.~\ref{def:bdylayer-mesh}), 
condition~\eqref{eq:L-eps-resolution-foo} also enforces the 
condition \eqref{eq:L-eps-resolution-n-foo}, which provides the 
approximation of corner layer functions. Finally, the situation 
\eqref{eq:lemma:asymptotic-case-20} is of a different nature as
in that case, the polynomial degree is so large that already very coarse
meshes can resolve the boundary layers. 

We employ the analytic, parametric regularity theory 
for the solution $u_\varepsilon$ 
presented in \cite[Thms.~{2.3.1}, {2.3.4}]{melenk02}. 
The infimum in \eqref{eq:thm:singular-approx-10} is estimated with the aid
of the interpolation operator $\Pi_q$ that is defined elementwise by 
\begin{align*}
(\Pi_q u)|_K \circ F_K & := 
\begin{cases}
\PiT (u \circ F_K) & \mbox{ if $K$ is a triangle} \\
\PiGL (u \circ F_K) & \mbox{ if $K$ is a rectangle}.
\end{cases}
\end{align*}
Here, the operator $\PiT$ is defined in 
Lemma~\ref{lemma:hat_Pi_infty} and the operator $\PiGL$ in 
Lemma~\ref{lemma:hat_Pi_1_infty}. Since $\PiT$ and 
$\PiGL$ reduce to the Gauss-Lobatto interpolation operator
on the edges of the reference element, 
the operator $\Pi_q$ indeed maps into $S^q_0(\Omega,\Tg)$.  
We recall that the element maps $F_K$ have the form 
$$
F_K = F_{K^\M} \circ A_K, 
$$
where 
$A_K:\Kh \rightarrow \Kt:= A_K(\Kh) = 
F_{K^\M}^{-1}(K)\subset \widetilde S$ is an affine bijection. 
Indeed, for triangular elements it is clear 
that $A_K$ is affine and for rectangular elements, this follows
from the special form of the reference patches (cf.\ also 
Lemma~\ref{lemma:properties-mesh-patches}, 
(\ref{item:lemma:properties-mesh-patches-i})). 

The notation $\widehat u$ denotes the pull-back of $u$ 
to the reference element, i.e., 
$\widehat u := u|_K \circ F_K$ 
whereas 
$\widetilde u:= (u \circ F_{K^\M})|_{\Kt} = \widehat u \circ A_K^{-1}$ 
is the corresponding function on $\Kt$.
We recall the notation $\widetilde \Pi_q$ from (\ref{eq:tildePiq}) and note 
that on a macro-element $K^\M$ we have 
\begin{equation*}
(\Pi_q u)\circ F_{K^\M} = \widetilde \Pi_q \ut. 
\end{equation*}
For $k \in {\mathbb N}_0$ we have for all elements $K \subset K^\M$
with $\Kt = F_{K^\M}^{-1}(K)$ 
\begin{subequations}
\label{eq:norm-equivalence}
\begin{align}
\forall v \in H^k(K)\colon\;\;
\|v \circ F_{K^\M} \|_{H^k(\Kt)} \sim \|v\|_{H^k(K)} , \\
\forall v \in W^{k,\infty}(K) \colon\;\; 
\|v \circ F_{K^\M} \|_{W^{k,\infty}(\Kt)} \sim \|v\|_{W^{k,\infty}(K)}, 
\end{align}
\end{subequations}
where in both cases the constants implied in $\sim$ depend solely on $k$ and
the macro-element $K^\M$. The equivalences (\ref{eq:norm-equivalence}) 
show that 
the approximation error $v - \Pi_q v$ on $K$ is equivalent to the corresponding
error $\widetilde{v} - \widetilde\Pi_q \widetilde{v}$ on $\Kt$. 

The approximation theory distinguishes between the ``asymptotic case''
$q \varepsilon \ge \kappa$ of large polynomial degree $q$ 
and the ``preasymptotic case''
$q \varepsilon \leq \kappa$, where the parameter $\kappa > 0$ 
(depending only on $A$, $c$, $f$, $\Omega$, the macro-triangulation, 
and $\sigma$)
is of size $O(1)$ and will be determined in the course of the analysis 
of the ``asymptotic case'' in Step~{I}. 

\emph{Step I: Asymptotic case $q \varepsilon \ge \kappa$.} 
We consider mesh patches $K^\M$ that abut on a vertex $\bA_j$ and those 
with a positive distance from the vertices separately in Steps~{I.1} and {I.2}. 

\emph{Step I.1: $K^\M$ abuts on a vertex $\bA_j$:} 
The regularity of \cite[Thm.~{2.3.1}]{melenk02} asserts the existence
of $C$, $\gamma > 0$, $\beta_j  \in [0,1)$ such that with 
$r_j(\cdot) := \operatorname{dist}(\cdot,\bA_j)$, there holds 
for every $p \in {\mathbb N}_0$ and for every $0<\varepsilon \leq 1$ 
\begin{equation}
|\nabla^p (u_\varepsilon(\cdot) - u_\epsilon(\bA_j))| 
\leq 
C \gamma^p \varepsilon^{-1} \min\{1,r_j(\cdot)/\varepsilon\}^{1-\beta_j} (r_j(\cdot))^{-p} 
\max\{p+1,r_j(\cdot)/\varepsilon\}^{p+1} . 
\end{equation}
Recall from (\ref{eq:def-rt}) that $\rt(\cdot) = \operatorname{dist}(\cdot,\bO)$. 
Set $\ut_\varepsilon:= u_\varepsilon \circ F_{K^\M}$. Note $F_{K^\M}(\bO) = \bA_j$ and 
$\rt(\boxt) \sim r_j(F_{K^\M}(\boxt))$. 
The analyticity of $F_{K^\M}$ and Lemma~\ref{L:cornerLayTrafo} imply, 
for suitably modified constants $C$, $\gamma$ independent of $\varepsilon\in (0,1]$, 
for every $p\in \mathbb{N}_0$ holds on $\Kt$
\begin{equation}
|\nabla^p (\ut_\varepsilon(\cdot) - \ut(\bO))| 
\leq C \gamma^p \varepsilon^{-1} \min\{1,\rt/\varepsilon\}^{1-\beta_j} (\rt(\cdot))^{-p} 
\max\{p+1,\rt(\cdot)/\varepsilon\}^{p+1} .
\end{equation}
Lemma~\ref{lemma:asymptotic-case} then yields 
$$
\|\ut_\varepsilon - \widetilde \Pi_q \ut_\varepsilon\|_{L^\infty(K^\M)} + 
\|\nabla( \ut_\varepsilon - \widetilde \Pi_q \ut_\varepsilon)\|_{L^\infty(K^\M)} 
\leq C q^9 \left( \sigma^{(1-\beta_j) n} + e^{-bq}\right) 
$$
provided that $\kappa$ is chosen sufficiently large (depending on $\gamma$). 

\emph{Step I.2: $K^\M$ does not abut on a vertex $\bA_j$:} 
\cite[Thm.~{2.3.1}]{melenk02} asserts 
\begin{equation}
\label{eq:asymptotic-case-interior-regularity} 
\forall \bx \in K^\M \quad \forall p \in {\mathbb N}_0 \colon\;\;
|\nabla^p u(\bx)| \leq C \gamma^p \max\{p+1,\varepsilon^{-1}\}^{p+2}  
\end{equation}
for constants $C$, $\gamma > 0$ independent of $\varepsilon \in (0,1]$. 
Since $K^\M$ is a trivial patch, it consists of a single (curvilinear) quadrilateral. 
The analyticity of $F_{K^\M} = F_K$ and Lemma~\ref{L:cornerLayTrafo} imply, for suitably
modified $C$, $\gamma$ independent of $\varepsilon\in (0,1]$, that
\begin{equation}
\label{eq:asymptotic-case-interior-regularity-ref-square} 
\forall \widehat\bx \in \widehat{S} \quad \forall p \in {\mathbb N}_0 \colon\;\;
|\nabla^p \uh (\widehat\bx)| \leq C \gamma^p \max\{p+1,\varepsilon^{-1}\}^{p+2}  
\;.
\end{equation}
Lemma~\ref{lemma:hat_Pi_1_infty} then implies 
that there are $C$, $b>0$ such that for sufficiently large, fixed $\kappa$
and for every $\varepsilon \in (0,1]$ and every $q\in \mathbb{N}$ holds
$$
\|\uh- \widehat \Pi_q \uh\|_{L^\infty(K^\M)} + 
\|\nabla( \uh - \widehat \Pi_q \uh)\|_{L^\infty(\widehat S)} 
\leq C e^{-bq}
\;.
$$

\emph{Step~{I.3}:} 
Combining the approximation results of Steps~{I.1}, {I.2} for the finitely many patches 
leads to the desired estimate (\ref{eq:thm:singular-approx-10}). 

\emph{Step II: Preasymptotic case $q \varepsilon \leq \kappa$.} 
The parameter $\kappa$ has been fixed in Step~{I} through the appeal
to Lemmas~\ref{lemma:asymptotic-case} and \ref{lemma:hat_Pi_1_infty}.  
In the regime $q \varepsilon \leq \kappa$, 
we employ the regularity theory of \cite[Thm.~{2.3.4}]{melenk02}, which 
furnishes the decomposition
$u_\varepsilon = w_\varepsilon + \chi^{\BL} u^{\BL}_\varepsilon + \chi^{\CL} u^{\CL}_\varepsilon + r_\varepsilon$
into a smooth part $w_\varepsilon$, 
a boundary layer part $u^{\BL}_\varepsilon$, 
a corner layer part $u^{\CL}_\varepsilon$, 
and a small remainder $r_\varepsilon$; the functions 
$\chi^{\BL}$, $\chi^{\CL}$ are suitable localizations near the boundary 
and the vertices of $\Omega$.  
We approximate each of these four contributions in turn. 

\emph{Step II.1: Approximation of $w_\varepsilon$.} 
By \cite[Thm.~{2.3.4}]{melenk02} the smooth part $w_\varepsilon$ 
is analytic on $\overline{\Omega}$ with constants independent of 
$\varepsilon$. Therefore, 
one can show $\|w_\varepsilon - \Pi_q w_\varepsilon \|_{W^{1,\infty}(\Omega)} \leq C e^{-bq}$
using similar techniques as in the asymptotic case above (essentially, setting $\varepsilon = 1$
there and ignoring the special treatment of the elements abutting on the vertices of $\Omega$). 

\emph{Step II.2: Approximation of $\chi^{\BL} u^{\BL}_\varepsilon$.} 

\emph{Step II.2.a: Regularity of $u^{\BL}_\varepsilon$:} 
The regularity of $u^{\BL}_\varepsilon$ in \cite[Thm.~{2.3.4}]{melenk02} 
is described in terms of boundary fitted coordinates (cf.\ Fig.~\ref{fig:bdyfitted-coord}). 
Associated
with each edge $\Gamma_j$ are fitted coordinates $(\rho_j,\theta_j)$, 
where $\rho_j$ is the distance from the analytic continuation 
$\widetilde \Gamma_j$ of the boundary arc $\Gamma_j$, 
and $\theta_j$ is a parametrization of $\Gamma_j$. 
The map $\psi_j: (\rho_j,\theta_j) \mapsto (x,y) \in \Omega$ is analytic
with an analytic inverse. An analytic arc $\Gamma_j^\prime$ 
emanates from each vertex $\bA_j$, which can be chosen arbitrarily 
but is assumed to be such that the angles between 
$\Gamma_j$ and $\Gamma_j^\prime$ and between 
$\Gamma_{j+1}$ and $\Gamma_j^\prime$ are both less than $\pi$. 
Condition~\ref{item:def-geo-mesh-9} of Definition~\ref{def:bdylayer-mesh}
ensures that $\Gamma_j^\prime$ can be chosen to be a meshline of 
a boundary layer mesh since it can be chosen as the image of an
edge of $\widetilde S$ or a diagonal of $\widetilde S$ under a patch
map.  

The regions $\Omega_j\subset \{\bx \in \Omega\,|\, 
\operatorname{dist}(\bx,\Gamma_j) <\delta \}$ for a sufficiently small $\delta$ 
are confined by the lines $\Gamma_j$, $\Gamma_j^\prime$, 
and $\Gamma_{j-1}^\prime$ as shown in Fig.~\ref{fig:notation-1}. 
By \cite[Thm.~{2.3.4}]{melenk02}, 
the function $u^{\BL}_\varepsilon$ is analytic on each $\Omega_j$ 
and satisfies there, for constants $C$, $\gamma$, $\alpha > 0$ independent
of $\varepsilon \in (0,1]$ and all $(m,n) \in {\mathbb N}_0$,
\begin{align}
\nonumber 
|\partial_{\rho_j}^n \partial_{\theta_j}^m u^{\BL}_\varepsilon \circ \psi_j(\rho_j,\theta_j)| 
& \stackrel{\text{\cite[Thm.~{2.3.4}]{melenk02}}}{\leq} 
C \varepsilon^{-n} \gamma^{n+m} m! e^{-\alpha \rho_j/\varepsilon}  \\ 
& \leq C \gamma^{n+m} \max\{n+m,\varepsilon^{-1}\}^{n+m} 
e^{-\alpha \rho_j/\varepsilon} 
\label{eq:bdy_regularity}
\;.
\end{align}
Finally, the cut-off function $\chi^{\BL}$ is supported by 
$\cup_j \overline{\Omega_j}$ and is identically $1$ near $\partial\Omega$. 

\emph{Step II.2.b: Approximation of $\chi^{\BL} u^{\BL}_\varepsilon$ far from $\partial\Omega$:} 
In the interest of simplicity of notation, we make the 
assumption that patches $K^\M$ touching $\partial\Omega$ are fully
contained in the tubular neighborhood $\cup_j \overline{\Omega_j}$
of $\partial\Omega$. Since patches $K^\M$ not touching $\partial\Omega$
have a positive distance from $\partial\Omega$, the function 
$\chi^{\BL} u^{\BL}_\varepsilon$ is exponentially small (in $1/\varepsilon$) 
there; in view of the stability (\ref{eq:tildePiq-stable}) 
(and thus the stability of $\Pi_q$) 
$\|\chi^{\BL} u^{\BL}_\varepsilon - \Pi_p (\chi^{\BL} u^{\BL}_\varepsilon)\|_{W^{1,\infty}(K)} \leq C e^{-b/\varepsilon}$ for $K \in K^\M$.  
Since $q/\kappa \leq 1/\varepsilon$ the error contribution of these
patches is controlled in the desired fashion. 

\emph{Step II.2.c: Approximation of $\chi^{\BL} u^{\BL}_\varepsilon$ 
near $\partial\Omega$:} 
Let $K^\M$ be a patch touching $\partial\Omega$. Consider, for a fixed $j$
the pull-back $F_{K^\M}^{-1} (K^\M \cap \Omega_j)$. By the assumptions
of the boundary layer mesh (Def.~\ref{def:bdylayer-mesh}) 
this pull back is either empty, the full square $\widetilde S$, 
half the square $\widetilde T = \{\boxt = (\xt,\yt)\,|\, 0 < \xt < 1, 0 < \yt < \xt\}$,
or the other half $\Tf= \{\boxt = (\xt,\yt)\,|\, 0 < \xt < 1, \xt < \yt < 1\}$. 

To fix ideas, let us assume that $K^\M$ is a mixed patch. 
The reference mixed patch restricted to $\widetilde T$ is the half-patch 
$\widetilde{\calT}^{\Mi,\text{\rm half},L,n}$ and its restriction to $\Tf$ 
is $\widetilde{\calT}^{\Co,\text{\rm half},\text{\rm flip},n}$. We approximate 
$(\chi^{\BL} u^{\BL}_\varepsilon)\circ F_{K^\M}$ 
on these two parts separately, starting with the approximation on 
$\widetilde T$. The assumptions on boundary layer meshes 
(Def.~\ref{def:bdylayer-mesh}) allow us to assume that 
$F_{K^\M} (\widetilde T) \subset \Omega_j$ for some $j$. 
We recall that $F_{K^\M}$ maps the edge $\{\yt = 0\}$ of $\widetilde T$ to 
(a subset of) $\partial\Omega$, which corresponds to $\rho_j = 0$ 
in the boundary fitted coordinates.  The shape-regularity of $F_{K^\M}$ 
implies that $\psi_j^{-1} \circ F_{K^\M}$ has the form 
\begin{equation}
\label{eq:coordinate-trafo}  
\widetilde T \ni (\xt,\yt) \mapsto (\rho_j,\theta_j) = (\yt \rho(\xt,\yt), \theta(\xt,\yt) )
\end{equation}
for a pair of functions $\rho$, $\theta$ with $\rho \ge \rho_0 > 0$. 
The analyticity
of $\psi_j^{-1}$ and $F_{K^\M}$ 
implies that $\rho$ and $\theta$ are in fact 
analytic on $\overline{\widetilde T}$. 
Hence, the transformed function 
\begin{equation}\label{eq:Trnsueps-alt} 
\widetilde u^{\BL}_\varepsilon
:= 
u^{\BL}_\varepsilon \circ F_{K^\M} 
= 
u^{\BL}_\varepsilon \circ \psi_j \circ (\psi_j^{-1} \circ F_{K^\M})
\end{equation}
admits by Lemma~\ref{L:bdyLayTrafo} and (\ref{eq:bdy_regularity})
the analytic regularity 
\begin{equation} 
\label{eq:utbl-regularity}
\forall (m,n) \in {\mathbb N}_0^2 
\quad \forall (\xt,\yt) \in \widetilde T\colon\;\;
|\partial_{\xt}^m \partial_{\yt}^n \ut^{\BL}_\varepsilon(\xt,\yt)| 
\leq 
C \gamma^{m+n} m! \max\{n+1,\varepsilon^{-1}\}^n e^{-b \yt/\varepsilon},  
\end{equation}
where the constants $C$, $\gamma$, $b > 0$ are independent of $\varepsilon \in (0,1]$. 
We decompose the set of elements 
$\widetilde{\calT}^{\Mi,\text{\rm half},L,n}$ into two sets 
$\widetilde{\calT}_1:= \{\Kt \in \widetilde{\calT}^{\Mi,\text{\rm half},L,n}_{geo,\sigma}\,|\, 
\widetilde{\chi}^{\BL}|_{\Kt} \equiv 1\}$ and 
$\widetilde{\calT}_2:= \widetilde{\calT}^{\Mi,\text{\rm half},L,n}_{geo,\sigma}
\setminus \widetilde{\calT}_1$. For the elements of $\widetilde{\calT}_1$, 
Lemma~\ref{lemma:bdy-layer-approx}, (\ref{item:lemma:bdy-layer-approx-i}) 
and (\ref{eq:utbl-regularity}) give that there are $C$, $b > 0$ 
such that for every $q\in \bbN$ and every $\Kt \in \widetilde{\calT}_1$ 
\begin{equation}
\|\widetilde \chi^{\BL} \ut^{\BL}_\varepsilon - \widetilde \Pi_q ( \widetilde \chi^{\BL} \ut^{\BL}_\varepsilon) \|_{L^\infty(\Kt)} 
+ 
\varepsilon 
\|\nabla
(
\widetilde \chi^{\BL} \ut^{\BL}_\varepsilon - \widetilde \Pi_q ( \widetilde \chi^{\BL} \ut^{\BL}_\varepsilon)
)
 \|_{L^\infty(\Kt)}  
\leq 
C e^{-b q}. 
\end{equation}
For the elements of the set $\widetilde{\calT}_2$, we use 
(\ref{eq:coordinate-trafo}) to see that $\Kt \in \widetilde{\calT}_2$ implies
$\operatorname{dist}(\Kt,\{\yt=  0\}) > c$ for some $c > 0$ that depends 
solely on $F_{K^\M}$ and $\psi_j$. Hence, the smoothness of 
$\widetilde \chi^{\BL}$ and (\ref{eq:utbl-regularity}) 
provide 
$\|\widetilde \chi^{\BL} \ut^{\BL}_\varepsilon\|_{W^{1,\infty}(\Kt)} 
\leq C e^{-b/\varepsilon}$  for suitable $C$, $ b> 0$ 
and every $\Kt \in \widetilde{\calT}_2$. 
Hence, the 
stability properties of $\widetilde\Pi_q$ provided in 
(\ref{eq:tildePiq-stable}) and $q/\kappa \leq 1/\varepsilon$ imply 
for all $\Kt \in \widetilde{\calT}_2$ 
\begin{equation}
\|\widetilde \chi^{\BL} \ut^{\BL}_\varepsilon - 
\widetilde \Pi_q ( \widetilde \chi^{\BL} \ut^{\BL}_\varepsilon) \|_{L^\infty(\Kt)} 
+ 
\varepsilon 
\|\nabla(\widetilde 
\chi^{\BL} \ut^{\BL}_\varepsilon 
- 
\widetilde \Pi_q ( \widetilde \chi^{\BL} \ut^{\BL}_\varepsilon) )\|_{L^\infty(\Kt)}  
\leq C e^{-b q}.
\end{equation}
Let us now sketch the arguments for the approximation 
of $\widetilde \chi^{\BL} \ut^{\BL}_\varepsilon$ on $\Tf$. 
For notational simplicity, 
assume that $F_{K^\M}(\Tf) \subset \Omega_j$.  (If 
$F_{K^\M} (\Tf) \subset \Omega_{j'}$ for some different
$j'$, then replace $j$ with $j'$ in what follows.) The regularity 
assertion (\ref{eq:utbl-regularity}) is still valid. 
Next, one observes that on $\Tf$, one has 
$\yt \sim \rt(\xt,\yt) = \operatorname{dist}((\xt,\yt),\bO)$. 
Hence, recalling 
(\ref{eq:utbl-regularity}), 
$\ut^{\BL}_\varepsilon$ satisfies, 
for suitable $C$, $b > 0$ and for all $p \in {\mathbb N}_0$,  
\begin{equation}
|\nabla^p \ut^{\BL}_\varepsilon(\cdot)| 
\leq 
C \max\{p+1,\varepsilon^{-1}\}^p 
e^{-b \rt(\cdot)/\varepsilon} \quad \mbox{ on $\Tf$.}
\end{equation}
Replacing the appeal to 
Lemma~\ref{lemma:bdy-layer-approx}, (\ref{item:lemma:bdy-layer-approx-i})
with a reference to
Lemma~\ref{lemma:bdy-layer-approx}, (\ref{item:lemma:bdy-layer-approx-ii}), 
we may argue as above to obtain 
$$
\|\widetilde \chi^{\BL} \ut^{\BL}_\varepsilon - 
\widetilde \Pi_q \widetilde \chi^{\BL} \ut^{\BL}_\varepsilon 
\|_{L^\infty(\Tf)} + 
\varepsilon \|\nabla (\widetilde \chi^{\BL} \ut^{\BL}_\varepsilon - 
\widetilde \Pi_q \widetilde \chi^{\BL} \ut^{\BL}_\varepsilon )
\|_{L^\infty(\Tf)} 
\leq C e^{-b q}. 
$$
This concludes the arguments for the 
approximation of $\ut^{\BL}_\varepsilon$ on a mixed patch 
$\widetilde{\calT}^{\Mi,L,n}_{geo,\sigma}$. The approximation on 
corner patches $\widetilde{\calT}^{\Co,n}_{geo,\sigma}$, 
tensor patches $\widetilde{\cal T}^{\Te,n}_{geo,\sigma}$, or 
boundary layer patches $\widetilde{\cal T}^{\BL,L}_{geo,\sigma}$ is similar.  

\emph{Step II.3: Approximation of $\chi^{\CL} u^{\CL}_\varepsilon$:}
Structurally, the proof is similar to the procedure in Step~{II.2}. 
{}From \cite[Thm.~{2.3.4}]{melenk02} we have in a neighborhood 
$B_j$ of vertex $\bA_j$ that $u^{\CL}_\varepsilon$ satisfies on 
$(B_j \cap \Omega_j) \cup (B_j \cap \Omega_{j+1})$ with 
$r_j(\cdot) = \operatorname{dist}(\cdot,A_j)$ 
\begin{equation}
\forall p \in {\mathbb N}_0 \colon\;\; 
|\nabla^p u^{\CL}_\varepsilon(\cdot)|
\leq 
C \gamma^p p! \varepsilon^{\beta_j-1} (r_j(\cdot))^{1-p-\beta_j} 
e^{-\alpha r_j(\cdot)/\varepsilon},
\end{equation}
where $C$, $\alpha >0$ and $\beta_j \in [0,1)$ are independent of 
$\varepsilon\in (0,1]$. 
Let $K^\M$ be a patch abutting on $\bA_j$. 
Such a patch has to be either a corner patch or a mixed patch. 
Then $K^\M \cap \Omega_j$ (and similarly $K^\M \cap \Omega_{j+1}$) 
consists of one or two half-patches that are push-forwards of 
$\widetilde{\calT}^\prime 
\in 
\{
\widetilde{\calT}^{\Mi,\text{\rm half},L,n}_{geo,\sigma}, 
\widetilde{\calT}^{\Co,\text{\rm half},n}_{geo,\sigma}, 
\widetilde{\calT}^{\Co,\text{\rm half},\text{\rm flip},n}_{geo,\sigma}
\}$. 
For simplicity of exposition, assume that $K^\M \subset B_j$. 
By the analyticity of the patch-map $F_{K^\M}$, 
the shape regularity of $F_{K^\M}$ together with $F_{K^\M}(\bO) = \bA_j$, 
and by Lemma~\ref{L:cornerLayTrafo},
we get that 
$\ut^{\CL}_\varepsilon
:= 
u^{\CL}_\varepsilon \circ F_{K^\M}$ 
satisfies on $\calO:= F_{K^\M}^{-1} (K^\M \cap \Omega_j)$
\begin{equation}
\forall p \in {\mathbb N}_0 \colon\;\;
|\nabla^p \ut^{\CL}_\varepsilon(\cdot)|
\leq 
C \gamma^p p! \varepsilon^{\beta_j-1} (\rt(\cdot))^{1-p-\beta_j} 
e^{-\alpha \rt(\cdot)/\varepsilon},
\end{equation}
with possibly adjusted values for $C$, $\gamma$, $\alpha > 0$.
We also note that the pull-back $\widetilde \chi^{\CL}$ is 
smooth and identically $1$ near $\bO$. 
Hence, using Lemma~\ref{lemma:corner-layer-approx} we obtain 
\begin{align*}
& 
\|\widetilde \chi^{\CL} \ut^{\CL}_\varepsilon 
- \widetilde \Pi_q (\widetilde \chi^{\CL} \ut^{\CL}_\varepsilon)\|_{L^2(\calO)} 
+ 
\varepsilon \|\nabla\bigr( \widetilde \chi^{\CL} \ut^{\CL}_\varepsilon 
- \widetilde \Pi_q (\widetilde \chi^{\CL} \ut^{\CL}_\varepsilon)\bigl)\|_{L^2(\calO)} 
\\
& \qquad \leq C \left( \varepsilon e^{-bq} + \varepsilon^{\beta_j} q^4 \sigma^{n (1-{\beta_j})}\right). 
\end{align*}

\emph{Step II.4: Approximation of $r_\varepsilon$:}
We approximate $r_\varepsilon$ by zero. 
We note that 
\cite[Thm.~{2.3.4}]{melenk02} asserts that 
$r_\varepsilon|_{\partial\Omega} = 0$
and that 
$\|r_\varepsilon\|_{H^1(\Omega)} \leq C e^{-b/\varepsilon}$ 
for suitable $C$, $b>0$ independent of $\varepsilon \in (0,1]$.
\end{proof}
\begin{corollary}
\label{coro:ExpCnv}
Assume the hypotheses on $\Omega$ and the data $A$, $c$, $f$ as
in Theorem~\ref{thm:singular-approx}.
Let $\Tg$ be a geometric boundary  layer  mesh 
as in Definition~\ref{def:bdylayer-mesh}.

Then, 
for every fixed $0<\sigma < 1$ and $c_1 > 0$
there exist constants $C$, $b>0$ such that,
for every $0 < \eps \leq 1$, 
with the choices
$q \simeq n \geq L \geq c_1 |\log \eps|$,
the solution 
$u_\eps\in H^1_0(\Omega)$ of \eqref{eq:sg-per} 
can be approximated from $S^q_0(\Omega,\Tg)$ 
at an exponential rate: 
$$
\inf_{v\in S^q_0(\Omega,\Tg)} \| u_\eps - v \|_{\eps,\Omega} 
\leq 
C \exp(-b\sqrt[4]{N})\;,
\quad 
N = {\rm dim}(S^q_0(\Omega,\Tg)).
$$
\end{corollary}
\begin{remark}
\label{rem:balanced-norm}
In addition to the approximation in the energy norm in Cor.~\ref{coro:ExpCnv},  
exponential approximation results in the so-called ``balanced norm''
$\|v\|^2_{\sqrt{\eps}}:= \eps \|\nabla v\|^2_{L^2(\Omega)} + \|v\|^2_{L^2(\Omega)}$
(or even in $H^1(\Omega)$) are possible under slightly stronger conditions: 
for sufficiently large $C_1$, the constraint  
$q \simeq n \geq L \geq C_1 |\log \eps|$ yields 
$$
\inf_{v\in S^q_0(\Omega,\Tg)} \| u_\eps - v \|_{\sqrt{\eps},\Omega} 
\leq 
C \exp(-b'\sqrt[4]{N})\;
$$
for suitable $b' > 0$ since a factor $\eps^{-1/2}$ can be compensated by the 
exponentially decaying terms $e^{-b \sqrt[4]{N}  } $ in Cor.~\ref{coro:ExpCnv}. 
\eremk
\end{remark}
Theorem~\ref{thm:singular-approx} is restricted to $\varepsilon \in (0,1]$.
For $\varepsilon \ge 1$, 
\eqref{eq:sg-per} is a regularly perturbed elliptic boundary value problem
and 
exponential convergence of $hp$-FEM with mere geometric corner refinement
follows by standard results \cite{phpSchwab1998,melenk02}. 
\begin{proposition}
\label{prop:singular-approx-eps>1}
Assume the hypotheses on $\Omega$ and the data $A$, $c$, $f$ as
in Theorem~\ref{thm:singular-approx}. 
Let $\Tg$ be a geometric boundary  layer  mesh.
Then, there are constants $C$, $b>0$, $\beta \in [0,1)$ 
depending solely on $A$, $c$, $f$, the analyticity properties of the
patch maps for the macro-triangulation, and $\sigma$   such that for 
any $\varepsilon \ge 1$, the solution
$u_\varepsilon$ of \eqref{eq:sg-per} satisfies for every $n$, $L$, $q \in \bbN$
\begin{equation}
\label{eq:thm:singular-approx-eps>1}
\inf_{v \in S^{q}_0(\Omega,\Tg)}
\|u_\varepsilon - v\|_{H^1(\Omega)} \leq C \varepsilon^{-2} 
\left( q^9 \sigma^{(1-\beta)n}   + e^{-b q}\right).
\end{equation}
\end{proposition}
\begin{proof}
The solution $u_\varepsilon \in H^1_0(\Omega)$ satisfies 
\begin{equation}
\label{eq:thm:singular-approx-eps>1-10}
-\nabla \cdot (A \nabla u_\varepsilon) + \varepsilon^{-2} c u_\varepsilon 
=
\varepsilon^{-2} f \quad \mbox{ in $H^{-1}(\Omega)$}.
\end{equation}
For $\varepsilon \geq 1$, 
the term $\varepsilon^{-2} c$ represents 
a regular perturbation and the analytic regularity theory for
linear, second order elliptic boundary value problems 
(e.g. \cite{BabGuoCurved1988} and the references there) is applicable.
The resulting regularity assertions are then those employed in
the ``asymptotic case'' in the proof of
Theorem~\ref{thm:singular-approx} with $\varepsilon = 1$ there.
The factor $\varepsilon^{-2}$ in (\ref{eq:thm:singular-approx-eps>1})
is a reflection of the fact that the right-hand side of
(\ref{eq:thm:singular-approx-eps>1-10}) include the factor $\varepsilon^{-2}$.
\end{proof}
\section{Numerical experiments}
\label{S:NumExp}
For $0<\varepsilon \leq 1$ and $f\equiv 1$ 
we consider the Dirichlet problem:
find $u_\varepsilon \in H_0^1(\Omega)$ such that
\[
-\varepsilon^2 \Delta u_\varepsilon + u_\varepsilon = f \;\; \text{ in } \;H^{-1}(\Omega).
\]
Here, the domain $\Omega$ is either the unit square 
$\Omega_1 = (0,1)^2$, 
the so-called ``$L$-shaped, polygonal domain'' 
$\Omega_2 \subset \R^2$ determined by the vertices
$\{ (0,0), (1,0), (1,1), (-1,1), \linebreak[4] (-1,-1), (0,-1)\}$, 
or the square domain with a slit 
$\Omega_3 = (-1,1)^2 \setminus (-1,0] \times \{0\}$.
%

In Figures~\ref{fig:square}--\ref{fig:slit}
we show examples of the meshes used in our computations on the three domains. 
These are constructed using the NGSolve/Netgen package \cite{netgen1}. 
For the square domain $\Omega = \Omega_1$ the resulting mesh 
is the geometric boundary layer mesh $\calT^{L,L}_{geo, \sigma}$  
with $L = 4$ and $\sigma = 0.25$. 
The same parameters are used in NGSolve/Netgen to construct the meshes for the other two domains, 
with the resulting meshes differing slightly from the strict definition of $\calT^{L,L}_{geo, \sigma}$ 
near the re-entrant corners. Nevertheless, 
we denote these meshes also by $\calT^{L,L}_{geo, \sigma}$ 
and make use of the finite element spaces $S^q_0(\Omega,\calT^{L,L}_{geo,\sigma})$. 
We also mention that in accordance with Remark~\ref{rem:bl-meshes}, 
the meshes shown in Figs.~\ref{fig:square}--\ref{fig:slit} do not satisfy 
requirement \ref{item:def-geo-mesh-9} of Definition~\ref{def:bdylayer-mesh}.

For each $p = 1,2,3,\dots$, we use the finite element space 
$S^q_0(\Omega,\calT^{L,L}_{geo,\sigma})$ 
with uniform polynomial order $q = p$ and with $L = p$ refinement levels 
towards boundaries and corners with refinement factor $\sigma = 0.25$.  
We denote by $u^h_\varepsilon \in S^p_0(\Omega,\calT^{p,p}_{geo,\sigma})$
    the corresponding finite element solution.
We measure the error in energy norm
\begin{equation}\label{eq:Erreh}
\text{error} 
= 
\left(\varepsilon^2 \|\nabla(u_\varepsilon-u^h_\varepsilon)\|^2_{L^2(\Omega)}
+ 
      \|u_\varepsilon-u^h_\varepsilon \|^2_{L^2(\Omega)}\right)^{1/2},
\end{equation}
where $u^h_\varepsilon$ denotes the discrete solution.
In place of the (unknown, for the considered examples)
exact solution  $u_\varepsilon$ we use a numerical approximation on a sufficiently fine mesh.   
The plots of the estimated numerical errors for the three domains are depicted 
in Figures~\ref{fig:square}--\ref{fig:slit}. 
Evidently, exponential convergence occurs.
In agreement with the theoretical analysis,
the experimentally observed exponential convergence has two regimes:
(i) an asymptotic regime 
in which the scale resolution condition $\sigma^p \lesssim \varepsilon$ 
is satisfied and 
(ii) a pre-asymptotic regime with $\sigma^p \gtrsim \varepsilon$. 

The observed exponential convergence in the preasymptotic regime 
(not rigorously shown in Theorem~\ref{thm:singular-approx}) 
is plausible for the following reason: 
the approximation error for boundary layer functions is dominated by the error on 
the elements touching the boundary and is of size $O(\sigma^{p/2})$ 
for every $p\in\mathbb{N}$.
The approximation error of the corner layer 
functions is likewise dominated by the error on the elements abutting on the vertices
of $\Omega$ and is of size $O(\sigma^{p(1-\beta)})$ 
for every $p\in\mathbb{N}$ and some fixed $\beta \in [0,1)$. 

%
\begin{figure}
  \centering
    \begin{overpic}[width=0.4\textwidth]{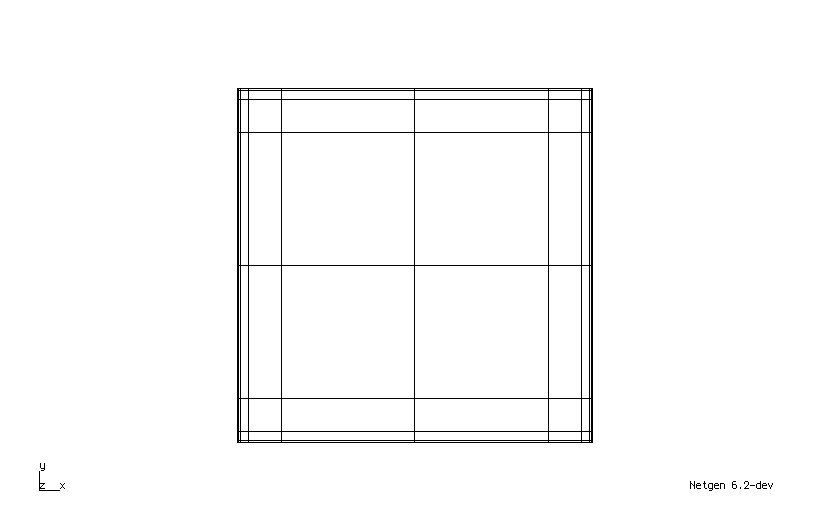}
    \end{overpic}
  \includegraphics[width=.5\textwidth]{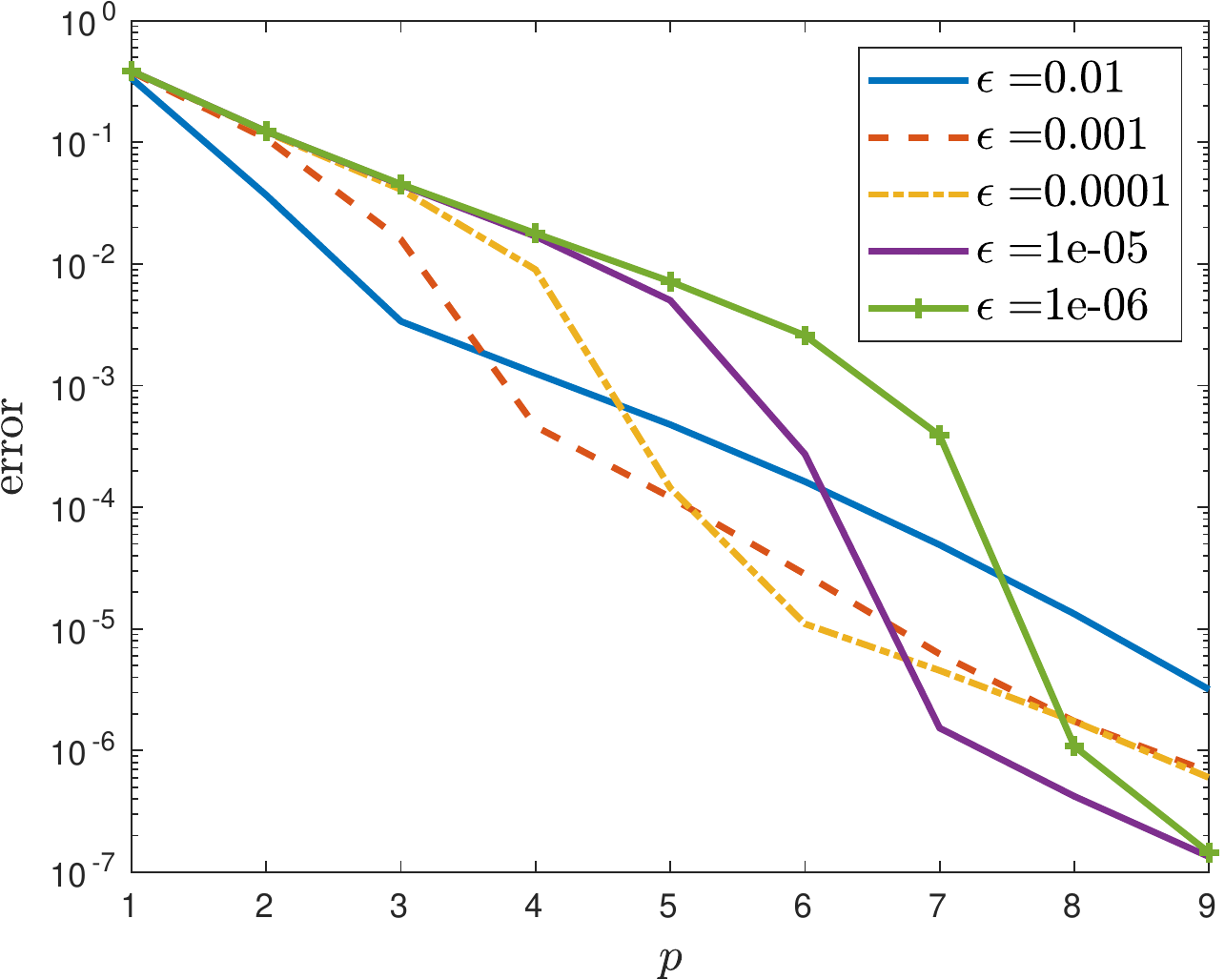}
  \caption{Right panel: Convergence in the energy norm (\ref{eq:Erreh}) 
   for the square domain for different values of $\varepsilon$ and $q = L = n = p$. 
Left panel: a Netgen-generated mesh used for the computations.}
  \label{fig:square}
\end{figure}

\begin{figure}
  \centering
    \begin{overpic}[width=0.4\textwidth]{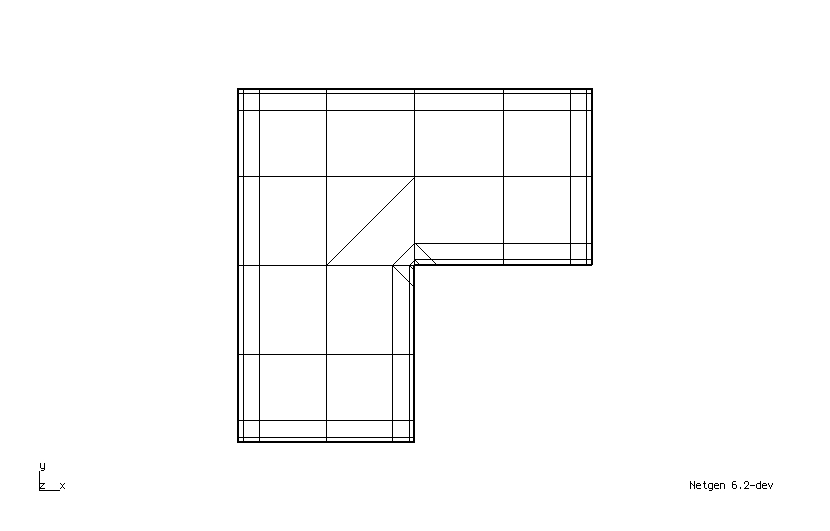}
    \end{overpic}
  \includegraphics[width=.5\textwidth]{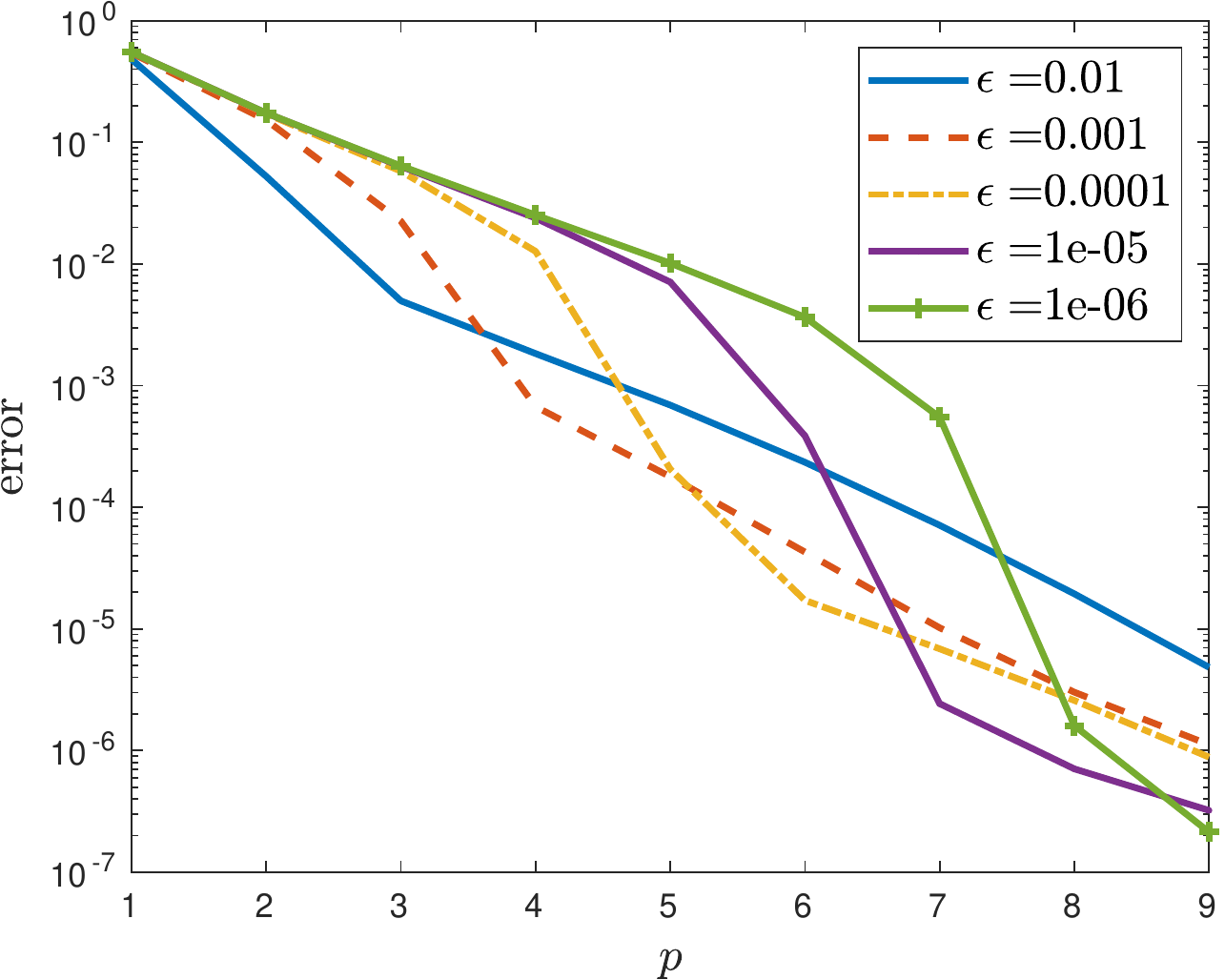}
  \caption{Right panel: 
   Convergence in the energy norm (\ref{eq:Erreh}) for the $L$-shaped domain for different values of $\varepsilon$ and 
           $q = L = n = p$.
Left panel: a Netgen-generated mesh used for the computations.}
  \label{fig:lshape}
\end{figure}

\begin{figure}
  \centering
        \begin{overpic}[width=0.4\textwidth]{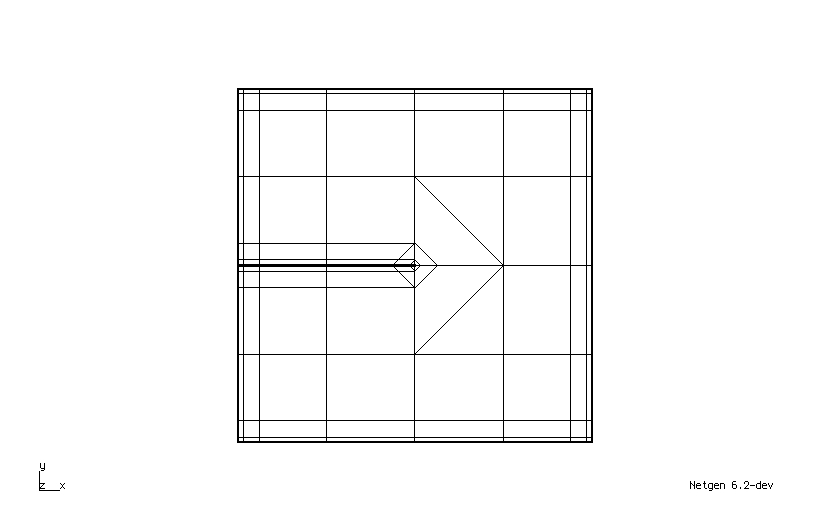}
     \end{overpic}
  \includegraphics[width=.5\textwidth]{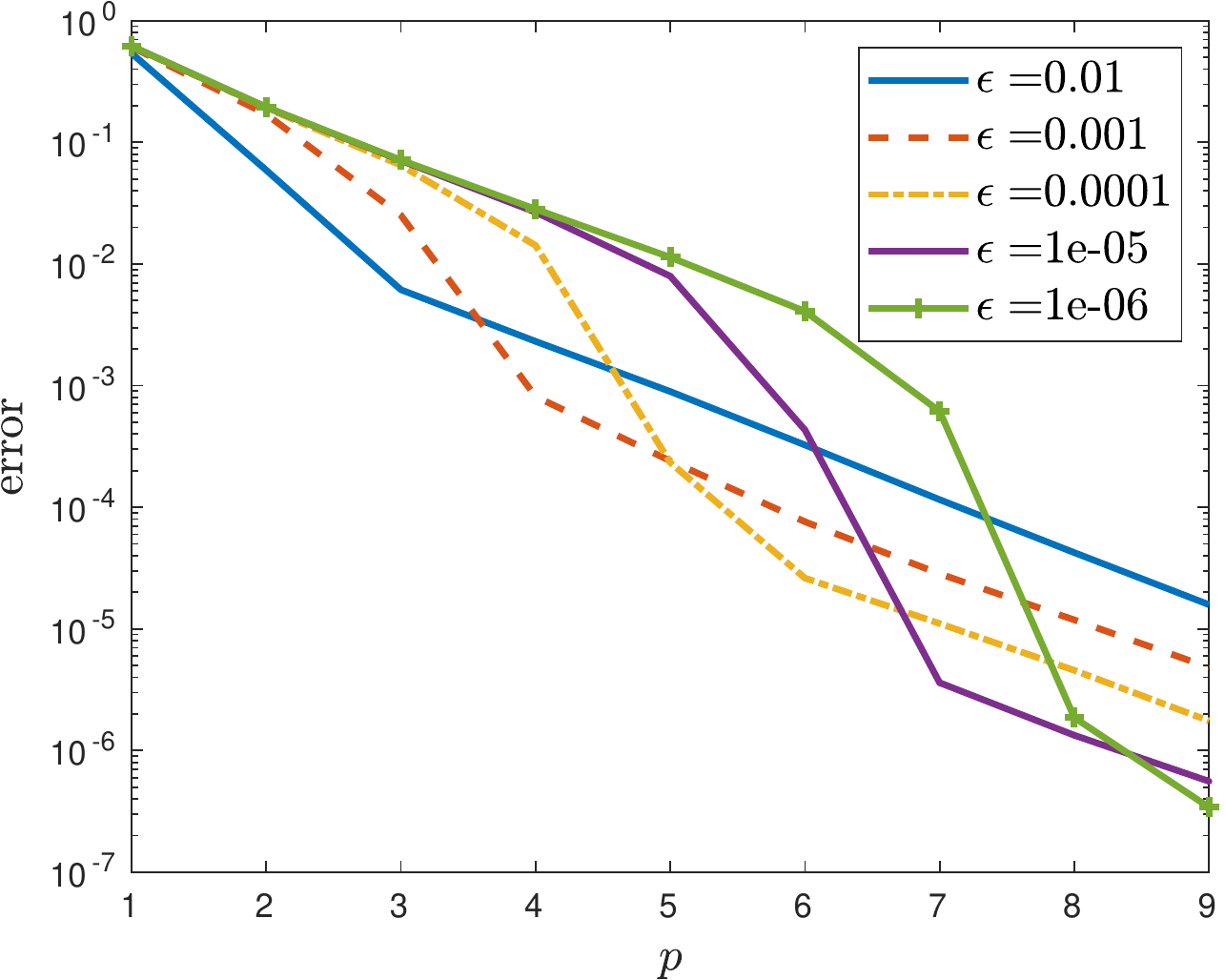}
  \caption{Right panel: Convergence in the energy norm (\ref{eq:Erreh}) 
           for the slit domain for different values of $\varepsilon$ and $q = L = n = p$.
           Left panel: a Netgen-generated mesh used for the computations.}
  \label{fig:slit}
\end{figure}


\section{Conclusions}
\label{S:Concl}
We established robust exponential convergence of $hp$-FEM 
for solutions of elliptic singular perturbation problems in polygons.
These solutions contain, usually, 
boundary layers, corner singularities and combinations of the two. 
We admitted possibly multiple length scales, and built the $hp$-FE
approximations on (patches of) geometric boundary layer meshes 
as described in Section~\ref{sec:MacTriGBLMes},
and depicted in Fig.~\ref{fig:patches}. 
The $hp$-FEM on this class of partitions 
is capable to resolve exponential boundary- and corner-layers
with multiple physical length scales under a scale resolution condition
that incorporates the smallest physical length scale.
The number of geometric mesh refinements to achieve
this grows only logarithmically with respect to the smallest 
length scale.
The proposed, spectral numerical boundary layer resolution by $hp$-FEM 
is based on boundary fitted, 
structured mesh-patches in the physical domain $\Omega$.
They are pushforwards from a finite catalog ${\mathfrak P}$ 
of canonical, highly structured, anisotropic reference mesh patterns. 
Such pushforwards are readily available in industrial CAD / CAM environments
such as, e.g., NgSolve \cite{netgen1,netgen2}.
For deployment,
it only requires (a lower bound on) the smallest physical length-scale $\eps$.
This can often be deduced from heuristic, physical considerations, e.g. scaling
or dimensional analysis.
The methodology should be contrasted with so-called 
\emph{``augmented/enriched spectral discretizations''}
proposed even recently in \cite{TemamSgPrt2018,TemamEtAl2020} and in references there. 
For this numerical approach, 
\emph{explicit, functional forms of boundary and corner layer components of the solution} 
are required. 
The analytic derivation of closed forms
for such solution components in general geometries 
for possibly nonlinear PDEs requires an elaborate asymptotic analysis, which 
is unnecessary in the present approach.

As we explained in the numerical experiments section 
patch-structured meshes as required here can be generated, in general geometries, 
by specialized mesh generators such as Netgen \cite{netgen1}.
We hasten to add, however, that our analysis can readily be extended to cover 
more general partitions, such as geometric boundary layer meshes that also contain 
anisotropic triangles. 

The focus of the present work was on robust 
exponential convergence rate bounds for singular perturbation problems
in nonsmooth domains by $hp$ finite element methods. 
We proved that they  
afford robust, exponential convergence on patchwise structured meshes 
with possibly anisotropic, geometric refinement towards the ``support set'' (i.e., 
the subset of $\overline{\Omega}$ off which the layer components decay exponentially), 
of the boundary and corner layers.
As a rule, robust exponential convergence requires genuine $hp$-FE capabilities, i.e., 
simultaneous mesh refinement and polynomial degree increase, as featured in the 
$hp$-FE spaces $\{S^p_0(\Omega,\calT^{p,p}_{geo,\sigma})\}_{p\geq 1}$ 
used in our numerical experiments. 
The corresponding, patchwise structured triangulations
can be automatically generated by specialized mesh generators, 
in domains of engineering interest (see, e.g.,\cite{netgen1}).
Although we mainly considered the model linear, second order elliptic singular
perturbation problem \eqref{eq:sg-per}, corresponding solution families are
known to arise for several common models in solid and fluid mechanics, see,
e.g. \cite{ArnFalkBLRM1996,GMSShell1998,DaugeShellBl,TemamSgPrt2018} and 
the references there.

The present $hp$-error analysis implies exponential upper bounds on 
Kolmogorov $N$-widths of solution sets $\{ u_\eps : 0< \eps \leq 1 \}$ of 
\eqref{eq:sg-per}. 
We recall that, for a normed linear space $X$
(with norm $\| \circ \|_X$) and for a subset $B\subset X$,
the $N$-width is given by 
\begin{equation}\label{eq:KlmNwid}
d_N(B,X) = \inf_{E_N} \sup_{f\in B} \inf_{g\in E_N} \| f - g \|_X,
\end{equation}
where the first infimum is taken over all subspaces $E_N$ of $X$ 
of dimension $N\in \bbN$. 
Subspace sequences $\{ E_N \}_{N\geq 1}$ which attain the rates 
of $d_N(B,X)$ in \eqref{eq:KlmNwid}
can be realized numerically by non-polynomial, 
so-called \emph{reduced bases} (see \cite{RBM}). 

In \eqref{eq:KlmNwid} 
we choose 
$(X,\|\circ\|_X) = (H^1_0(\Omega), \|\circ\|_{\eps,\Omega})$
with the energy norm $\|\circ\|_{\eps,\Omega}$ of \eqref{eq:epsnrm}. 
Given a complex neighborhood $G \subset {\mathbb C}^2$ of 
$\overline{\Omega}$ we take 
$B\subset X$ as the set of solutions 
of \eqref{eq:sg-per} corresponding right-hand sides $f$ 
that admit a holomorphic extension to $G$ with $\|f\|_{L^\infty(G)} \leq 1$.
From Corollary~\ref{coro:ExpCnv}, 
with $E_N = S^q_0(\Omega,\Tg)$ we obtain 
$d_N(B,X)\lesssim \exp(-b\sqrt[4]{N})$ 
with a continuous, piecewise polynomial 
interpolant to bound the inner infimum in \eqref{eq:KlmNwid}.
Remark that \cite[Theorem 3.2]{JMMnwidth} stipulates for \eqref{eq:sg-per}
the (sharp) majorization $d_N(A,X) \lesssim \exp(-b'\sqrt{N})$
with $b'>0$ possibly different from $b$ and
with nonpolynomial $E_N$, but with $b'$ and 
the constant hidden in $\lesssim$ independent of $\eps$.
For analytic $\partial \Omega$ and a single, known
boundary layer length scale $\eps \in (0,1]$,
this rate is
attained by $hp$-FEM on so-called minimal boundary layer meshes
(e.g. \cite{SSX98_321,melenk-schwab98}) 
which are $\eps$-dependent, however.

The underlying concept of 
using patchwise structured meshes to approximate parametric solution families 
to linear, elliptic singularly perturbed boundary value problems
extends also to $h$-version FEM. 
Here, 
in patches abutting on the boundary analogs of so-called ``Shishkin meshes'', 
see, e.g., \cite{ShishCrnerBlr87}, \cite[Sec.~{3.5.2}]{RoosStynsTobiska2ndEd},
could be employed to achieve \emph{robust, algebraic rates of convergence} under
weaker, finite order differentiability assumptions on the data $A$, $c$, and $f$
than the presently assumed analyticity in $\overline{\Omega}$ of these data.
The present results will constitute a foundation for proving
exponential convergence of several $hp$ discretizations of 
(spectral) fractional diffusion problems 
as presented in \cite{BMNOSS17_732}
in curvilinear polygonal domains $\Omega$.
Details will be developed in \cite{BMS20}.

The model problem (\ref{eq:sg-per}) considers homogeneous
Dirichlet boundary conditions. The approximation result 
Theorem~\ref{thm:singular-approx} relies on the regularity results of
\cite{melenk02}, which decomposes the solution (\ref{eq:sg-per}) 
into boundary and corner layer components. Similar decompositions
can be expected to hold also for other boundary conditions. Then the 
approximation results of Section~\ref{S:AppRefElt} are applicable 
indicating that $hp$-FEM on similarly patchwise structured meshes 
will likewise lead to robust exponential convergence. 
\appendix
\section{Analytic changes of variables}
\label{AppA:AnChVar}
The following lemma shows how boundary layer functions are transformed under the patch maps
if the edge $\{\yt = 0\}$ of $\widetilde S$ is mapped to a subset of $\partial\Omega$: 
\begin{lemma}
\label{L:bdyLayTrafo}
Let $G_x \subset {\mathbb R} \times {\mathbb R}^+$ be a domain.
Let the map $M: (\xt,\yt) \mapsto (\theta,\rho)$ be of the form 
$M(\xt,\yt) = (\check \theta(\xt,\yt) , y \check \rho(x,y))$ for some 
functions $\check \theta$, $\check\rho \ge \rho_0>0$ that are analytic on 
$\operatorname{closure}(G_x)$, i.e., there are constants 
$C_M, \gamma_M > 0$ such that 
$\|\nabla^n \check\theta\|_{L^\infty(G_x)}$, $\|\nabla^n \check\rho\|_{L^\infty(G_x)} 
\leq 
C_M \gamma_M^n n!$ for all $n \in {\mathbb N}_0$. 
Let ${\mathcal O}_x \subset G_x$ be open and 
let ${\mathcal O}$ be an open neighborhood of $M({\mathcal O}_x)$. 
Let $u$ be analytic on ${\mathcal O}$ and assume that, 
for some function $C_u$ and some constants $b > 0$, $\gamma > 0$, 
there holds 
\begin{align*}
\forall (m,n)\in{\mathbb N}_0^2\quad \forall (\rho,\theta) \in {\mathcal O}\colon\;\;
|\partial_\rho^n \partial_\theta^m u(\theta,\rho)| 
\leq 
C_u(\theta,\rho)  e^{-b \rho/\varepsilon} \gamma^{n+m} m! \max\{n,\varepsilon^{-1}\}^n 
.
\end{align*}
Then there are constants $b'$, $\widetilde \gamma > 0$ (depending only on $b$, $\gamma$, and $M$) 
such that the function $\widetilde u:= u \circ M$ satisfies with the 
notation $(\rho,\theta) = M(\xt,\yt)$ 
\begin{align*}
\forall (m,n) \in {\mathbb N}_0^2 
\quad \forall (\xt,\yt) \in \mathcal{O}_x \colon\;\;
|\partial_{\yt}^n \partial_{\xt}^m \widetilde u(\xt,\yt)| 
\leq 
C_u(\theta,\rho) e^{-b' \yt/\varepsilon} \gamma^{n+m} m! \max\{n,\varepsilon^{-1}\}^n . 
\end{align*}
\end{lemma}
\begin{proof}
The proof uses arguments employed in \cite[Sec.~{4.3}]{melenk02}. 
Consider a fixed $(\xt,\yt) \in {\mathcal O}_x$ and set $(\theta',\rho') = M(\xt,\yt)$. 
Then $(\theta,\rho)\mapsto u(\theta,\rho)$ is holomorphic on the polydisc 
$$
B_{1/\gamma}(\theta') \times B_{1/(\gamma e)}(\rho') \subset {\mathbb C}^2
$$
with the bound 
\begin{align}
\label{eq:L:bdyLayTrafo-10}
|u(\theta'+\zeta_1,\rho'+\zeta_2)| 
\leq 
C_u(\rho',\theta') e^{-b \rho'/\varepsilon} \frac{1}{1 - \gamma |\zeta_1|} 
\left[ \frac{1}{1 - \gamma e |\zeta_2|} + \exp{(\gamma |\zeta_2|/\varepsilon})\right].
\end{align}
Since the functions $\check\theta$, $\check\rho$ are holomorphic on the 
closure of $M(G_x)$, there are $C_1$, $\delta > 0$ 
(independent of $(\xt,\yt) \in G$) such that 
for $\zeta_1$, $\zeta_2 \in B_\delta(0) \subset {\mathbb C}$ there holds 
\begin{align*}
\left|\check \theta(\xt + \zeta_1,y+\zeta_2) - \check\theta(\xt,\yt) \right| & \leq 
C_1 \left[ |\zeta_1| + |\zeta_2|\right], \\
\left| (\yt+\zeta_2)\check\rho(\xt+\zeta_1,\yt+\zeta_2) - \yt \check\rho(\xt,\yt)\right| &\leq 
C_1 \left[ y (|\zeta_1| + |\zeta_2|) + |\zeta_2| \right], 
\end{align*}
and we may assume that $\delta>0$ is such that for $\zeta_1$, $\zeta_2 \in B_\delta(0)$ we have 
$M(\xt+\zeta_1,\yt+\zeta_2) \in B_{1/(2\gamma)}(\theta') \times B_{1/(2 \gamma e)}(\rho')$. 
This implies in view of (\ref{eq:L:bdyLayTrafo-10}) the bounds
\begin{align}
|\widetilde u(\xt +\zeta_1,\yt+\zeta_2)| 
& = 
|u(M(\xt+\zeta_1,\yt+\zeta_2))| \\
& \leq 
C C_u(\theta',\rho') e^{-b \rho'/\varepsilon} 
  \exp(C_1 \gamma |\zeta_2|/\varepsilon) 
   \exp \bigl(C_1 \gamma \yt\, \bigr[|\zeta_1|+|\zeta_2|\bigl]/\varepsilon\bigr).
\end{align}
For $\delta_1$, $\delta_2 < \delta$ Cauchy's integral formula for derivatives gives  
\begin{align*} 
\partial^{\alpha_1}_\xt \partial^{\alpha_2}_\yt \widetilde u(\xt,\yt) = - \frac{\alpha_1! \alpha_2!}{4 \pi^2} 
\int_{\zeta_1 \in \partial B_{\delta_1}(0)} 
\int_{\zeta_2 \in \partial B_{\delta_2}(0)} 
\frac{\widetilde u(x+\zeta_1,y+\zeta_2)}{(-\zeta_1)^{\alpha_1+1} (-\zeta_2)^{\alpha_2+1}}\,d\zeta_1 d\zeta_2 
\end{align*}
so that 
\begin{align*} 
\left| \partial^{\alpha_1}_\xt \partial^{\alpha_2}_\yt \widetilde u(\xt,\yt) 
\right| \leq C C_u(\rho',\theta') e^{-b\rho'/\varepsilon} \frac{\alpha_1!}{\delta_1^{\alpha_1} } \frac{\alpha_2!}{\delta_2^{\alpha_2} } 
\exp(C_1 \gamma \delta_2/\varepsilon) 
\exp(C_1 \gamma y (\delta_1+\delta_2)/\varepsilon) 
\end{align*}
Selecting $\delta_1 = \delta:=  b \rho_0/(4 C_1)$ 
and $\delta_2 = \min\{(|\alpha_2|+1) \varepsilon,\delta\}$ yields 
the desired result with $b' = b/2$ 
since $C_1 \yt (\delta_1 + \delta_2)/\varepsilon \leq 
2 \delta C_1 \yt/\varepsilon \leq 2 \delta C_1 \rho'/\rho_0 = b/2$.  
\end{proof}

The following lemma shows how functions that may have a singular behavior 
are transformed under analytic changes of variables: 
\begin{lemma}[\protect{\cite[Lemma~{4.3.3}]{melenk02}}]
\label{L:cornerLayTrafo}
Let $\widetilde G \subset {\mathbb R}^2$ be a domain and 
$M: \widetilde G \rightarrow {\mathbb R}^2$ be analytic on $\operatorname{closure}(\widetilde G)$. 
Let $\widetilde {\mathcal O} \subset \widetilde G$ be open 
and ${\mathcal O}$ be an open neighborhood
of $M(\widetilde {\mathcal O})$. Let $u$ be analytic on ${\mathcal O}$ and assume that 
for some (positive) function 
$\Lambda$, $r: {\mathcal O} \rightarrow {\mathbb R}$ and 
some $\gamma\ge 0$ there holds 
\begin{equation}
\label{eq:L:cornerLayTrafo-10}
\forall n \in {\mathbb N}_0 
\quad \forall \bx \in {\mathcal O} \colon\;\;
|\nabla^n u(\bx)| \leq \Lambda(\bx) \gamma^{n} \max\{(n+1)/r(\bx),\varepsilon^{-1}\}^n . 
\end{equation}
Then the function $\widetilde u:= u \circ M$ is analytic on 
$\widetilde {\mathcal O}$ and 
there are constants $C$, $\widetilde \gamma > 0$ depending solely on 
$M$ and $\gamma$ such that 
for each $\boxt \in \widetilde {\mathcal O}$ there holds with the notation $\bx = M(\boxt)$
$$
\forall n \in {\mathbb N}_0 \colon\;\
|\nabla^n \widetilde u(\boxt)| \leq 
C \Lambda(\bx) \widetilde \gamma^{n} \max\{(n+1)/r(\bx),\varepsilon^{-1}\}^n . 
$$
\end{lemma}
\begin{proof}
The statement is taken from \cite[Lemma~{4.3.3}]{melenk02} except that 
we explicitly allow $r$ to be a function of $\bx$. The proof is similar to that
of Lemma~\ref{L:bdyLayTrafo}. We fix $\boxt \in \widetilde G$ 
and set $\bx = M(\boxt)$. 
The assumption \eqref{eq:L:cornerLayTrafo-10} implies that $u$ has a 
holomorphic extension to $B_{c r(\bx)}(\bx) \subset {\mathbb C}^2$ with $c > 0$
depending solely on $\gamma$. Additionally, we have the bound 
for $z \in B_{c r(\bx)}(\bx) \subset {\mathbb C}^2$ 
(we write 
$r = r(\bx)$)
\begin{equation*}
|u(\bx+z)| \leq \Lambda(\bx) 
\sum_{n=0}^\infty \frac{1}{n!} |z|^n |\nabla^n u(\bx)|
\leq C \left[ \frac{1}{1 - |z|/(cr)} + \exp(C' |z|/\varepsilon)\right].
\end{equation*}
for suitable $C$, $C'$. 
The analyticity of $M$ on 
$\operatorname{closure}(\widetilde G)$ implies the existence of 
$\delta > 0$ (independent of $\boxt = (\xt,\yt) \in \widetilde G$) 
such that 
$$
M(\xt + B_{\delta r}(0) , \yt + B_{\delta r}(0)  ) 
\subset B_{\frac{1}{2} c r(\bx)}(\bx). 
$$
For $\alpha \in {\mathbb N}^2_0$ let 
$\theta:= \min\{\delta r(\bx), (|\alpha|+1) \varepsilon\}$. 
{}The Cauchy integral theorem for derivatives gives 
\begin{align*}
\partial_{\xt}^{\alpha_1} \partial_{\yt}^{\alpha_2} 
\widetilde u(\xt,\yt) = 
- \frac{\alpha_1! \alpha_2!}{4 \pi^2} 
\int_{z_1 \in \partial B_{\theta}(0)} \int_{z_2 \in \partial B_{\theta}(0)}
\frac{\widetilde u(M(\boxt+(z_1,z_2)))}
     {(-z_1)^{\alpha_1+1} (-z_2)^{\alpha_2+1}}
\end{align*}
so that we get 
\begin{align*}
|\partial_{\xt}^{\alpha_1} \partial_{\yt}^{\alpha_2}
\widetilde u(\xt,\yt) | 
& \lesssim \alpha_1! \alpha_2! \theta^{-|\alpha|} 
\left[ 1 + \exp(C' \theta/\varepsilon)\right] \\
& \lesssim |\alpha|! \max\{(\delta r(\bx))^{-1}, (|\alpha|+1)^{-1} \varepsilon^{-1}\}^{|\alpha|} 
\left[ 1 + \exp(C' |\alpha|)\right], \\
\end{align*}
which proves the asserted estimate. 
\end{proof}
\section{Univariate Approximation} 
\label{sec:AppB}
\begin{lemma}
\label{lemma:1D-poly-approx}
Let $I = (-1,1)$ and $u \in C^\infty(I)$ satisfy, for  
some constants $C_u$, $\gamma_u>0$,
for some $h\in(0,1]$, $\varepsilon \in (0,1]$ 
the bound
\begin{equation}
\forall n \in {\mathbb N}_0 \colon\;\;
\|D^n u\|_{L^\infty(I)} \leq C_u (\gamma_u h)^n \max\{n,\varepsilon^{-1}\}^n . 
\end{equation}
Then there are constants $C$, $\eta$, $\delta> 0$ depending solely on 
$\gamma_u$ such that under the constraint 
\begin{equation}
\label{eq:1d-constraint}
\frac{h}{\varepsilon q} \leq \delta 
\end{equation}
there holds 
\begin{equation}\label{eq:1d-consist}
\forall q \in {\mathbb N}\colon\;\;
\inf_{v \in {\mathbb P}_q} \|u - v\|_{W^{1,\infty}(I)} 
\leq 
C C_u \left( \left( \frac{h}{h+\eta}\right)^{q+1} + 
\left(\frac{h}{\eta \varepsilon q}\right)^{q+1} 
\right).
\end{equation}
\end{lemma}
\begin{proof}
We start with the observation that Taylor's theorem yields 
for $x > 0$ 
\begin{equation}
\label{eq:taylor-exp}
\sum_{n \ge q+1} \frac{1}{n!} x^n
= e^x - \sum_{n=0}^q \frac{x^n}{n!} 
= \frac{1}{q!} \int_0^x (x-t)^q e^t\,dt
\leq \frac{x^{q+1}}{q!} e^x.
\end{equation}
\emph{Case~1:} Let $e \gamma_u h < 1/2$. 
Then the Taylor series of $u$ about $x_0 = 0$ 
converges in $I$ 
and 
the Taylor polyomials $T_q \in {\mathbb P}_q$ 
satisfy the error bounds 
\begin{align*}
\|u - T_q\|_{L^\infty(I)} &\leq \sum_{n=q+1}^\infty \frac{|D^n u(0)|}{n!} 
\leq 
C_u \sum_{n=q+1}^\infty 
(\gamma_u e h)^n + \frac{(\gamma_u h/\varepsilon)^n}{n!} 
\\
& 
\
\stackrel{(\ref{eq:taylor-exp})}{ \leq } 
C_u \left( \frac{(\gamma_u e h)^{q+1}}{1 - (\gamma_u e h)} + 
\frac{(\gamma_u h/\varepsilon)^{q+1}}{q!} e^{\gamma_u h/\varepsilon}
\right) \\
& \stackrel{(\ref{eq:stirling})}{\leq} 
   C_u \left( 2 (\gamma_u e h)^{q+1} + 
C (\gamma^\prime h/(\varepsilon q))^{q+1} e^{\gamma_u h/\varepsilon}
\right),  
\end{align*}
for suitable $\gamma^\prime > \gamma_u$. 
The assumption (\ref{eq:1d-constraint}) allows us to estimate 
$e^{\gamma_u h/\varepsilon} \leq e^{\gamma_u \delta q}$ 
and the desired result follows for the $L^\infty$-estimate. 
An analogous argument applies for the $W^{1,\infty}$-estimate. 

\emph{Case~2:} 
Let $1 \ge h > 1/(2e\gamma_u)$. 
Introduce for $\rho > 1$ the ellipse 
${\mathcal E}_\rho := 
\{z \in {\mathbb C}\,|\, |z - 1| + |z + 1| < \rho + 1/\rho\}$ 
and set $G_\kappa(I):= \cup_{x \in I} B_\kappa(x)$. By geometric
considerations (e.g., with the aid of \cite[Lemma~{3.14}]{boerm-loehndorf-melenk05}) 
one has ${\mathcal E}_{1+\kappa} \subset G_\kappa(I)$. 
Taylor's theorem gives that $u$ is holomorphic on $G_{1/(\gamma_u h)}(I)$ 
and for every $\kappa< 1/(\gamma_u h)$ we have 
\begin{equation}
\label{eq:lemma:1d-hol}
\|u\|_{L^\infty(G_{\kappa})} \leq C_u \!
\sum_{n=0}^\infty \frac{1}{n!} (h\gamma_u \kappa)^n \max\{n,\varepsilon^{-1}\}^n 
\leq 
\!C_u\! \left[ \frac{1}{1-e \gamma_u h \kappa} + \exp(\kappa\gamma_u h /\varepsilon)\right]. 
\end{equation}
Well-established polynomial approximation results 
(see, e.g., \cite[Thm.~{6}]{apel-melenk17})
then yield for fixed 
$\kappa>0$ the existence of $\rho_1 = \rho_1(\kappa) > 1$ such that 
\begin{align*}
\inf_{v \in {\mathbb P}_q} \|u - v\|_{W^{1,\infty}(I)} 
\leq C C_u \rho_1^{-q} \|u\|_{L^\infty(G_\kappa)}
\leq C C_u \rho_1^{-q} e^{\kappa \gamma_u h/\varepsilon} 
\leq C C_u \rho_1^{-q} e^{\kappa \gamma_u \delta q}.
\end{align*}
Fix $1 < \rho_2 < \rho_1$. 
Then we may select $\delta >0 $ sufficiently small so 
that there exists a constant $C>0$ such that
\begin{align*}
\forall q\in \bbN: \;\;
\inf_{v \in {\mathbb P}_q} \|u - v\|_{W^{1,\infty}(I)} 
\leq C C_u \rho_2^{-q}. 
\end{align*}
Using $h \ge 1/(2e \gamma_u)$ and suitably choosing $\eta$, we can estimate 
\begin{align*}
\rho_2^{-q}  \leq \left(\frac{h}{h+\eta}\right)^q.  
\end{align*}
\end{proof}
\begin{lemma}[stability of the 1d-Gauss-Lobatto (GL) interpolant]
\label{lemma:1d-GL}
Let $I = [-1,1]$. 
There exists a constant $C>0$ such that for any $q\in \bbN$,
the Gauss-Lobatto interpolation  operator 
$i_q:C(I) \rightarrow {\bbP}_q$ satisfies: 
\begin{align}
\label{eq:lemma:1d-GL-i}
\|u - i_q u\|_{L^\infty(I)} &\leq (1 + \Lambda_q) \inf_{v \in {\bbP}_q}
\|u - v\|_{L^\infty(I)}, \qquad \Lambda_q  = C \ln (q+1), \\
\label{eq:lemma:1d-GL-ii}
\|(u - i_q u)^\prime\|_{L^\infty(I)} &\leq C(1 + q^2 \Lambda_q) 
\inf_{v \in {\bbP}_q} \|(u - v)^\prime\|_{L^\infty(I)}. 
\end{align}
\end{lemma}
\begin{proof}
The bound \eqref{eq:lemma:1d-GL-i} 
follows from the projection property 
of the Gauss-Lobatto interpolation; the logarithmic growth of the 
Lebesgue constant $\Lambda_q$ is shown in \cite{suendermann83}. 
 
For (\ref{eq:lemma:1d-GL-ii}), 
we estimate for arbitrary $v \in {\bbP}_q$ 
\begin{align*}
& \|(u - i_q u)^\prime\|_{L^\infty(I)} 
\leq 
\|(u - v)^\prime\|_{L^\infty(I)} + \|(i_q (u - v))^\prime\|_{L^\infty(I)} \\
& \qquad \lesssim \|(u - v)^\prime\|_{L^\infty(I)} + q^2 \|i_q (u - v)\|_{L^\infty(I)} 
 \lesssim \|(u - v)^\prime\|_{L^\infty(I)} + q^2 \Lambda_q \|u - v\|_{L^\infty(I)}.  
\end{align*}
Constraining $v$ to satisfy $v(-1) = u(-1)$ the result follows 
from a Poincar\'e inequality. 
\end{proof}
Finally, we recall two inequalities of Stirling's type.
\begin{align}
\label{eq:stirling}
\forall n\in \bbN\colon \;\; 
\sqrt{2\pi} n^{n+1/2} e^{-n} & \leq n! \leq e n^{n+1/2} e^{-n}  
\;,
\\
\label{eq:binom}
\forall n\in \bbN_0 \quad \forall \alpha\in \bbN_0\colon \;\;
 \alpha! n! &\ge 2^{-(\alpha+n)} (\alpha +n)! 
\ge 
(2 e)^{-(\alpha+n)} (\alpha+n)^{\alpha+n}. 
\end{align}
(\ref{eq:stirling}) follows from \cite{robbins55}. In 
\ref{eq:binom}, the first bound follows from the binomial formula 
$\sum_{\nu=0}^m \binom{m}{\nu} x^\nu = (1+x)^m$ 
and 
the second bound follows from (\ref{eq:stirling}). 

\end{document}